\documentclass[11pt,reqno]{amsart}
\usepackage[letterpaper,top=1in,bottom=1in,left=1in,right=1in,marginparwidth=1.75cm]{geometry}

\usepackage[english]{babel}
\usepackage[T1]{fontenc}
\usepackage[utf8]{inputenc}
\usepackage{amsmath}
\usepackage{amsthm}
\usepackage{amssymb}

\usepackage{multirow}
\PassOptionsToPackage{hyphens}{url}\usepackage{hyperref}
\usepackage{algpseudocode}
\usepackage{accents}
\usepackage{mathrsfs}
\usepackage{graphicx,mathtools,tikz,hyperref}
\usepackage[export]{adjustbox}
\usepackage{caption}
\usepackage{subcaption}
\usepackage{algorithm}
\usepackage{stackengine}
\usepackage{tikz}
\usepackage{enumitem,kantlipsum}

\usetikzlibrary{calc,intersections}
\usetikzlibrary{positioning}

\makeatletter
\def\blfootnote{\gdef\@thefnmark{}\@footnotetext}
\makeatother

\PassOptionsToPackage{hyphens}{url}\usepackage{hyperref}

\newtheorem{theorem}{Theorem}
\newtheorem{corollary}{Corollary}[theorem]
\newtheorem{lemma}[theorem]{Lemma}
\newtheorem{remark}{Remark}

\newcommand*{\rmu}{\ensuremath{\mathrm{u}}}
\newcommand*{\bu}{\ensuremath{\boldsymbol{u}}}
\newcommand*{\bp}{\ensuremath{\boldsymbol{p}}}

\providecommand{\keywords}[1]
{
  \small	
  \textbf{\textit{Keywords---}} #1
}

\numberwithin{equation}{section}

\title[Numerical analysis of the PDHG method for reaction-diffusion equations]{Numerical analysis of a first-order computational algorithm for reaction-diffusion equations via the primal-dual hybrid gradient method}

\author[Liu]{Shu Liu}
\email{shuliu@math.ucla.edu}
\address{Department of Mathematics, University of California, Los Angeles}

\author[Zuo]{Xinzhe Zuo}
\email{zxz@math.ucla.edu}
\address{Department of Mathematics, University of California, Los Angeles}

\author[Osher]{Stanley Osher}
\email{sjo@math.ucla.edu}
\address{Department of Mathematics, University of California, Los Angeles}

\author[Li]{Wuchen Li}
\email{wuchen@mailbox.sc.edu}
\address{Department of Mathematics, University of South Carolina}

\date{}

\begin{document}

\maketitle

\begin{abstract}
In \cite{liu2024first}, a first-order optimization algorithm has been introduced to solve time-implicit schemes of reaction-diffusion equations. In this research, we conduct theoretical studies on this first-order algorithm equipped with a quadratic regularization term. We provide sufficient conditions under which the proposed algorithm and its time-continuous limit converge exponentially fast to a desired time-implicit numerical solution. We show both theoretically and numerically that the convergence rate is independent of the grid size, which makes our method suitable for large scale problems. The efficiency of our algorithm has been verified via a series of numerical examples conducted on various types of reaction-diffusion equations. The choice of optimal hyperparameters as well as comparisons with some classical root-finding algorithms are also discussed in the numerical section.
\end{abstract}

\blfootnote{S. Liu and X. Zuo are partially funded by AFOSR YIP award No. FA9550-23-1-0087. S. Liu, X. Zuo, and S. Osher are partially funded by AFOSR MURI FA9550-18-502 and ONR N00014-20-1-2787. W. Li’s work is supported by AFOSR YIP award No. FA9550-23-1-0087, NSF DMS-2245097, and NSF RTG: 2038080.}

\section{Introduction}\label{intro}
Reaction-diffusion equations (RD) are well-known time-dependent partial differential equations (PDEs). They are originally used to model the density evolution of chemical systems with local reaction processes in which substances get transformed, and diffusion processes in which the substances get spread over. Since the same type of equations describe many systems, the RD equation finds its applications in broad scientific areas. This includes the study of phase-field models in which the Allen-Cahn and the Cahn-Hilliard equations \cite{allen1979microscopic, cahn1961spinodal} are used to depict the development of microstructures of multiple materials; the research on the evolution of species distribution in ecology system \cite{murray2001mathematical}; the study of the reaction processes of multiple chemicals \cite{schnakenberg1979simple, pearson1993complex}; and the modeling \& prediction of crimes \cite{crime_RD}. 

Time-implicit schemes are often used when solving RD equations numerically. This is because in simulations, explicit or semi-explicit schemes are often encountered with Courant–Friedrichs–Lewy (CFL) conditions, under which the time step size is restricted to be very small. Conversely, employing time-implicit schemes allows for the use of larger time step sizes, leading to a more efficient computation of the equilibrium state in RD equations. Moreover, computing RD equations with a weak diffusion and a strong reaction term is of great interest to the computational math community. The performance of explicit and semi-implicit schemes could be unstable under these circumstances. However, it has been shown that implicit schemes still work very well on these models \cite{xu2019stability, liu2024first}. In addition, time-implicit schemes are also known to be energy-stable \cite{xu2019stability}. 

In a recent work \cite{liu2024first}, the primal-dual hybrid gradient (PDHG) algorithm which is an easy-to-implement optimization algorithm, has been used for computing the time-implicit solution of RD equations. The PDHG algorithm \eqref{intr: abstract precond PDHG1 } is a first-order optimization algorithm with tunable hyperparameters. Notably, it does not require computing the inverse of the Jacobian matrix in the time-implicit scheme. It converges robustly regardless of the choice of the initial value, which is a key distinction from many classical methods, such as Newton’s methods. This property makes the PDHG algorithm easy to implement and computationally efficient for solving nonlinear equations. Another motivating feature of the PDHG method is that it allows for the design of customized preconditioning matrices based on the structure of the specific RD equation, resulting in a notable grid-size-independent convergence rate throughout the algorithm. 

Nevertheless, the prototype PDHG algorithm presented in \cite{liu2024first} faces theoretical and practical challenges. The time-implicit scheme results in a nonlinear equation, and the PDHG algorithm introduces nonlinear coupling in both the primal and dual variables. In addition, there is a lack of convergence analysis for the proposed PDHG algorithm. Furthermore, the nonlinearity inherent in RD equations poses a challenge in resolving the optimal choice of hyperparameters. In this paper, we provide the convergence study of the PDHG algorithm for computing the time-implicit scheme of RD equations. We also present a series of numerical experiments on the choices of hyperparameters.

Let us consider the general form of the RD equation on a region $\Omega \subset \mathbb{R}^d$ with prescribed boundary (e.g., periodic, Neumann, or Dirichlet) and initial conditions.
\begin{equation} 
  \frac{\partial u(x, t) }{\partial t} = - \mathcal{G} (a\mathcal{L}u(x,t) + bf(u(x,t))),  \quad x\in\Omega, \quad u(\cdot, 0) = u_0(\cdot),
  \label{general RD type PDE}
\end{equation}
Here we assume $\mathcal L, \mathcal G$ are self-adjoint, non-negative definite linear operators. $f(\cdot)$ is the reaction term (usually nonlinear). $a\geq 0$ is the diffusion coefficient. And $b\geq 0$ is the reaction coefficient. To compute the numerical solution of \eqref{general RD type PDE}, we adopt the following time-implicit scheme with a time step size $h_t>0$ and solve for the numerical solution on $N_t$ intervals:
\begin{equation}
  \frac{u^{t+1}-u^t}{h_t} = - \mathcal G(a\mathcal{L}u^{t+1} + b f(u^{t+1})), \quad 0\leq t\leq N_t-1.  \label{intr: semi discrete RD equ1}
\end{equation}

Assume that at each time step, the numerical solution $u^t$ belongs to a certain Hilbert space $\mathcal X$ with an inner product $(\cdot, \cdot)$. Let us denote $\boldsymbol{u}=[u^1, \dots, u^t, \dots, u^{N_t}]^\top\in\mathcal{X}^{N_t}$. Define the function $\mathcal{F}(\cdot):\mathcal{X}^{N_t}\rightarrow \mathcal{X}^{N_t}$ as
\begin{equation}
  \mathcal{F}(\boldsymbol{u}) = [\dots, u^{t+1} - u^t + h_t\mathcal G(a\mathcal L u^{t+1} + b f(u^{t+1})), \dots ]_{0\leq t\leq N_t-1}^\top.  \label{implicit scheme1}
\end{equation}
Then, solving the time-implicit scheme \eqref{intr: semi discrete RD equ1} is equivalent to obtaining the root of the problem $\mathcal F(\boldsymbol{u})=0$.

We now reformulate the time-implicit scheme \eqref{intr: semi discrete RD equ1} as an inf-sup problem with a tunable parameter $\epsilon$>0 following the treatment in \cite{zuo2023primal}
\begin{align}
  \inf_{\bu\in\mathcal X^{N_t}} \sup_{\bp\in\mathcal X^{N_t}}  &  (\bp, \mathcal F(\bu)) - \frac{\epsilon}{2}\|\bp\|_{ \mathcal X^{N_t} }^2.  \label{intr: min-max problem with quad1 }
\end{align}
Here we write $\bp = [p_1, \dots, p_t, \dots, p_{N_t}] \in \mathcal X^{N_t}.$ Compared with the saddle point scheme considered in \cite{liu2024first}, a quadratic regularization term is introduced in \eqref{intr: min-max problem with quad1 } to enhance the performance of the proposed algorithm both theoretically and numerically. It is not hard to verify that \eqref{intr: min-max problem with quad1 } is equivalent to the residue-minimizing problem $\underset{\boldsymbol{u}}{\inf} \ \frac{1}{2\epsilon}\|\mathcal F(\boldsymbol{u})\|_{\mathcal X^{N_t}}^2$, and we further point out that the saddle point of this inf-sup problem \eqref{intr: min-max problem with quad1 } exists and solves $\mathcal{F}(\bu)=0$ whenever the root-finding problem admits a unique solution. 

As demonstrated in \cite{liu2024first}, we deal with the inf-sup saddle problem by applying the primal-dual hybrid gradients (PDHG) algorithm \cite{chambolle2011first, zhu2008efficient}. We further substitute the proximal step of variable $\bu$ with an explicit update to obtain 
\begin{align}
\label{intr: abstract original PDHG1}
\begin{split}
   \bp_{n+1} = & \frac{1}{1+\epsilon \tau_P} \left(\bp_n + \tau_P \mathcal{F}(\bu_n)\right),\\
   \widetilde{\bp}_{n+1} = & \bp_{n+1} + \omega (\bp_{n+1} - \bp_n),  \\
   \bu_{n+1} = & \bu_n - \tau_U D\mathcal{F}(\bu_n)^*\widetilde{\bp}_{n+1}.
\end{split}
\end{align}
Here $\omega >0$ is the extrapolation coefficient, and $\tau_P, \tau_U>0$ are PDHG step sizes. $D\mathcal F(\bu)$ is a linear operator on $\mathcal{X}^{N_t}$, which denotes the Fr\'{e}chet derivative of $\mathcal{F}(\cdot)$ at $\bu$. $D\mathcal{F}(\bu)^*$ is the adjoint operator of $D\mathcal{F}(\bu)$ on $\mathcal{X}^{N_t}$. It is not hard to verify that the equilibrium state of PDHG scheme \eqref{intr: abstract original PDHG1} is the desired $(\bu_*, 0)$ with $\mathcal F(\bu_*)=0$ whenever $D\mathcal F(\bu)^*$ is invertible for arbitrary $\bu\in\mathcal{X}^{N_t}$. 

The PDHG algorithm \eqref{intr: abstract original PDHG1} converges slowly when $\mathcal F(\cdot)$ possesses a large condition number. To improve the convergence speed, it is necessary to consider preconditioning $\mathcal F(\cdot)$. We consider an invertible linear operator $\mathfrak M: \mathcal X^{N_t}\rightarrow \mathcal X^{N_t}$, where $\mathfrak M$ is extracted from the linear part of $\mathcal{F}(\cdot )$. Then we introduce the preconditioned functional $\widehat{\mathcal{F}}(\bu) = \mathfrak M^{-1} \mathcal F(u).$ We apply the PDHG algorithm \eqref{intr: abstract original PDHG1} to $\mathcal{\widehat F}(u)=0$ to obtain 
\begin{align}
\label{intr: abstract precond PDHG1 }
\begin{split}
  \bp_{n+1} = & \frac{1}{1+\epsilon\tau_P}(\bp_n + \tau_P {\mathfrak{M}}^{-1}\mathcal{F}(\bu_n)), \\ 
   \widetilde{\bp}_{n+1} = & \bp_{n+1} + \omega (\bp_{n+1} - \bp_n), \\ 
   \bu_{n+1} = & \bu_n - \tau_U D\mathcal{F}(\bu_n)^*(\mathfrak M^{-1})^*\widetilde{\bp}_{n+1} . 
\end{split}
\end{align}
The above treatment \eqref{intr: abstract precond PDHG1 } will significantly improve the algorithm's convergence speed while leaving the equilibrium state invariant. 

It is worth noting that the original approach proposed in \cite{liu2024first} solves the implicit scheme \eqref{implicit scheme1} while preserving the time causality: the algorithm sequentially computes \( u^t \) at each time step, using the previous solution as the initial condition. In contrast, our approach generalizes by accumulating multiple time steps into a single root-finding problem and computing the multi-step solution in a forward manner. More precisely, we solve \( \mathcal{F}(\boldsymbol{u}_j) = 0 \) for sequential blocks of solutions, where each block is defined as:
\[
\boldsymbol{u}_j = \begin{bmatrix} u^{j\cdot N_t + 1}, \dots, u^{j\cdot N_t + t}, \dots, u^{j\cdot N_t + N_t} \end{bmatrix}^\top \in \mathcal{X}^{N_t}, \quad j = 0,1,2,\dots
\]
When updating from \( \boldsymbol{u}_j \) to \( \boldsymbol{u}_{j+1} \), we set \( u^0 = u^{j\cdot N_t + N_t} \) as the initial condition in \eqref{intr: semi discrete RD equ1}. Unlike step-by-step update, the new approach preserves time causality among solution blocks \( \boldsymbol{u}_j, \boldsymbol{u}_{j+1} \).

In this paper, we analyze the aforementioned preconditioned PDHG algorithm \eqref{intr: abstract precond PDHG1 } to establish sufficient conditions under which the method is guaranteed to converge. We remark that there are two types of convergence analysis, which may cause confusion in this manuscript. One refers to the convergence of the numerical solution to the real solution as the 
 number of grid points increases; the other one refers to the convergence of $(\bu_n, \bp_n)$ to the equilibrium state of the PDHG algorithm \eqref{intr: abstract precond PDHG1 } as $n$ increases. In this research, we mainly focus on analyzing the second type of convergence. We now briefly summarize the main contributions: 
\begin{itemize}
    \item (Theoretical aspect) Suppose that the reaction term $f(\cdot)$ is Lipschitz. Assume that the discretization of the differential operators $\mathcal L_h, \mathcal G_h$ are positive-definite, self-adjoint, and commute. We establish the following theoretical results for our algorithm.
    \begin{enumerate}
        \item  We study the PDHG flow \eqref{PDHG cont time}, which is the time-continuous limit of \eqref{Preconded PDHG} as $\tau_U, \tau_P\rightarrow 0, (1+\omega)\tau_P \rightarrow \gamma > 0$. We give conditions on $h_t, N_t$ under which we can pick $\gamma, \epsilon$ such that the residual term exponentially decays to $0$. The convergence results for general RD equations are discussed in Theorem \ref{thm: exponential convergence regardless of b T } and Theorem \ref{thm: O(1) exponential convergence result for general RD equ }; We establish convergence rates that are independent of the grid-size $N_x$ for both Allen-Cahn type and Cahn-Hilliard type equations in Corollary \ref{practical corollary }.
        \item We analyze the convergence speed of the PDHG method \eqref{Preconded PDHG} in Theorem \ref{thm: time discrete PDHG converge with Lip }. We show that under certain conditions of  $h_t, N_t$, we are able to select suitable hyperparameters $\tau_U, \tau_P, \omega, \epsilon$ that guarantee the exponential convergence of the $L_2$ error term. We establish convergence rates that are independent of the grid-size $N_x$ for both Allen-Cahn type and Cahn-Hilliard type equations in Corollary \ref{coro: simplify PDHG alg converge}. 
    \end{enumerate}
    \item (Numerical aspect) In section \ref{subsec: convergence result for time continuous case } and \ref{sec: lyapunov time discrt case} we justify our theoretical results stated above. In section \ref{sec: tested eq }, we demonstrate the effectiveness of our algorithm on different RD equations, including the standard Allen-Cahn and Cahn-Hilliard equations, as well as equations with variable mobility terms or higher-order diffusion terms whose linear operator $\mathscr{M}$ (c.f. \eqref{def: scr M }) cannot be directly inverted. In section \ref{sec: PDHG convergence rate independency wrt grid resolution }, we validate that the convergence rate of our method is independent of the grid size $N_x$. In section \ref{sec: hyperparam select}, we investigate the optimal, or at least near-optimal hyperparameters of our algorithm for achieving efficient performance. We demonstrate the efficiency of our method by comparing it with some of the classical methods 
    in section \ref{sec : long-range} and section \ref{sec: compare speed with 3 methods }. 
\end{itemize}

There exist plenty of references regarding the numerical schemes for RD equations, which include studies on finite difference methods \cite{ ceniceros2013new, christlieb2014high, eyre1998unconditionally, hou2023linear, hundsdorfer2003numerical, li2010unconditionally,  merriman1994motion, shen2019new, shen2010numerical, xu2019stability, yang2016linear}, 
and finite element methods \cite{fu2023high, fu2023generalized, hundsdorfer2003numerical, liu2021structure, liu2022second, liu2022convergence, LiuLagrangian, zhu2009application}. A series of benchmark problems \cite{church2019high, jokisaari2017benchmark} has also been introduced to verify the effectiveness of the proposed methods. Recently, machine learning or deep learning algorithms such as deep-learning-based backward stochastic differential equations (BSDE) \cite{han2018solving, han2017deep}, physics-informed neural networks (PINNs) \cite{raissi2019physics, wight2020solving, xu2022numerical}, and Gaussian processes \cite{chen2021solving} have also been applied to deal with various types of nonlinear equations including the RD equations.

The primal-dual hybrid gradients (PDHG) method was first introduced in \cite{chambolle2011first, zhu2008efficient} to deal with constrained optimization problems arising in image processing. This method later finds its applications in various branches such as nonsmooth PDE-constrained optimization \cite{clason2017primal}, Magnetic resonance imaging (MRI) \cite{valkonen2014primal}, large-scale optimization problems including image denoising and optimal transport \cite{jacobs2019solving}, computing gradient flows in Wasserstein-like transport metric spaces \cite{CarrilloCraigWangWei2022_primala, CarrilloWangWei2023_structure, fu2023high}, as well as design fast optimization algorithms \cite{zuo2023primal}, etc. 

In \cite{fahroo1996optimum}, the authors introduce damping terms to the wave equation to achieve faster stabilization, which resembles the time-continuous limit (the PDHG flow) \eqref{PDHG cont time} of our proposed algorithm. However, \cite{fahroo1996optimum} focuses on the linear case while our research deals with nonlinear RD equations. In recent work \cite{chen2023transformed}, the authors conduct certain transformations to enhance the convergence of a saddle point algorithm. Although the transformed algorithm shares similarities with our method, the target functionals considered in both researches are distinct. In \cite{chen2023accelerated}, the authors apply the splitting method to propose an accelerating algorithm for the root-finding problem $\mathcal A(x)=0$, where $\mathcal A$ can be decomposed as the sum of the gradient function and the skew-symmetric operator. In contrast, our proposed method deals with a time-dependent root-finding problem, which generally can not be cast into the settings of \cite{chen2023accelerated}. We refer our readers to \cite{liu2024first} for more detailed discussions on related references.

Our research is inspired by \cite{liu2022primal} in which the authors apply the PDHG algorithm to compute time-implicit conservation laws. Our former research \cite{liu2024first} mainly focuses on the conceptual and experimental aspects of the PDHG method applied to RD equations. In addition, the primal-dual method also finds its application in the computation of Hamilton-Jacobi equations \cite{meng2023primal}. The aforementioned works \cite{liu2024first, liu2022primal, meng2023primal} do not address the convergence speed of the PDHG algorithm. In this work, we establish the convergence guarantee for the nonlinearly coupled primal-dual system. Moreover, we prove a convergence property of our method, where the convergence rate is independent of the space grid size.

This paper is organized as follows. In section \ref{sec: derivation }, we provide a detailed derivation of our algorithm applied to RD equations. In section \ref{sec: unique existence}, we establish the existence and uniqueness result regarding the time-implicit scheme of the RD equation. In section \ref{sec: Lyapunov analysis PDHG flow time continue }, we focus on the PDHG flow, which is the time-continuous limit of the proposed algorithm. We first establish convergence results for the general root-finding problem and then apply our theory to the time-implicit schemes of RD equations. In section \ref{sec: lyapunov time discrt case}, we prove exponential convergence of our algorithm. We also investigate necessary conditions that guarantee such convergence. In section \ref{sec: numerical example }, we demonstrate the effectiveness of our method on different types of RD equations and make comprehensive comparisons with the IMEX scheme as well as some classical root-finding algorithms.

\section{Derivation of the method}\label{sec: derivation }
In this section, we give a detailed derivation of the PDHG method when applied to the reaction-diffusion (RD) equation \eqref{general RD type PDE}. From now on, we assume that the domain $\Omega=[0, L]^2$ is a square region.

Suppose we solve \eqref{general RD type PDE} on the time interval $[0, T]$. We divide the time interval into $N_t$ subintervals, and divide the domain $\Omega$ into $N_x \times N_x$ grids. Applying time-implicit finite difference scheme yields
\begin{equation}
  \frac{\rmu^{t+1}-\rmu^t}{h_t} = - \mathcal G_{h} ( a\mathcal L_{h} \rmu^{t+1} + b f(\rmu^{t+1}) ), \quad \textrm{for } t=0, 1, \dots, N_t, \textrm{ with } \rmu^0  \textrm{ given.}  \label{time-implicit scheme}
\end{equation}
Denote $h_t = \frac{T}{N_t}$, and $h_x = \frac{L}{N_x}$. Write $U^t\in\mathbb{R}^{N_x\times N_x}$ as the numerical solution at the $t-$th time node. We denote $\mathcal{G}_{h}, \mathcal{L}_{h}$ as $N_x^2\times N_x^2$ matrices, which represents the discretization of the operator $\mathcal L, \mathcal G$ w.r.t. the spatial step size $h_x$ and the boundary condition. 
\begin{remark}[Allen-Cahn and Cahn-Hilliard equations]\label{rk: AC and CH eq }
  For Allen-Cahn equation \cite{allen1979microscopic}, we have $\mathcal G = \mathrm{Id}$, $\mathcal L = -\Delta$; for Cahn-Hilliard equation \cite{cahn1961spinodal}, we have $\mathcal{G} = -\Delta$, $\mathcal L = -\Delta.$ And $f(\cdot) = W'(\cdot)$ where $W(\xi) = \frac 14 (\xi^2 - 1)^2$ is the double-well potential for both equations. We can impose periodic or homogeneous Neumann boundary conditions for both equations. Furthermore, suppose we apply the central difference scheme to discretize the Laplace operator $\Delta$. 
  We obtain $\Delta_{h_x}^P = I_{N_x} \otimes \mathrm{Lap}_{h_x}^P + \mathrm{Lap}_{h_x}^P\otimes I_{N_x}$ for periodic boundary condition, and $\Delta_{h_x}^N = I_{N_x} \otimes \mathrm{Lap}_{h_x}^N + \mathrm{Lap}_{h_x}^N \otimes I_{N_x}$ for Neumann boundary condition, where $\otimes$ is the Kronecker product and we define
  \begin{equation}
      \mathrm{Lap}_{h_x}^P = \frac{1}{h_x^2}\left[ \begin{array}{ccccc}
          -2  & 1 &   &  & 1 \\
          1& -2  & 1 &  &  \\
             & \ddots & \ddots & \ddots &\\
             &        &   1   &   - 2   &  1\\
          1  &        &        &   1  & -2 
      \end{array} \right], \quad
      \mathrm{Lap}_{h_x}^N = \frac{1}{h_x^2}\left[ \begin{array}{ccccc}
          -1  & 1 &   &  &  \\
          1& -2  & 1 &  &  \\
             & \ddots & \ddots & \ddots &\\
             &        &   1   &   - 2   &  1\\
             &        &        &   1  & -1 
      \end{array} \right].\label{def: periodic discrete Laplace operator }
  \end{equation}  
\end{remark}

\subsection{PDHG method for preconditioned root-finding problem}
In this section, we provide a more detailed derivation for our algorithm. 

Let us treat $\mathcal X  = \mathbb{R}^{N_x^2}$. We denote $U=[{\rmu^{1}}^\top, \dots, {\rmu^{N_t}}^\top]^\top\in\mathbb{R}^{N_tN_x^2}$ as the numerical solution. $\mathcal L_h, \mathcal G_h$ indicate the discrete approximations of $\mathcal L, \mathcal G$. We formulate the time-implicit scheme \eqref{time-implicit scheme} as a root-finding problem   
\begin{equation} 
  F(U) = 0,  \label{root-finding}
\end{equation}
with $F:\mathbb{R}^{ N_t N_x^2  }\rightarrow \mathbb{R}^{N_tN_x^2}$ defined as
\begin{equation}
  F(U) = \mathscr D U + h_t\mathscr{G}_h( a \mathscr{L}_{h} U + bf(U) )-V.  \label{def: F }
\end{equation}
Here we denote the time difference matrix $\mathscr{D} = D_{N_t} \otimes I_x$, where $I_x$ is the identity matrix on $\mathbb{R}^{N_x^2}$, and
\begin{equation} 
  D_N = \left[\begin{array}{ccccc}
    1 & & & &   \\
    -1& 1 & & &  \\
      & -1& 1 & & \\
      &   & \ddots & \ddots & \\
      &   &        &     -1 & 1
\end{array}\right] \quad \textrm{is an}~~N\times N ~ \textrm{matrix}.  \label{def: mat D N }
\end{equation}
On the other hand, we define
\begin{equation}
 \mathscr{G}_h = I_t \otimes \mathcal G_h, \quad \mathscr{L}_h = I_t \otimes \mathcal{L}_h,\label{def; caligraph Gh, Lh }
 \end{equation}
with $I_t$ representing the identity matrix on $\mathbb{R}^{N_t}.$ The reaction function $f(\cdot)$ acts element-wisely on vector $U$. The constant vector $V\in\mathbb{R}^{N_tN_x^2}$ depends on both the initial condition and the boundary condition of the equation.

We aim to solve $F(U)=0.$ In \cite{liu2024first}, an indicator function $\iota(u)=\begin{cases}
    0      & \textrm{if } u=0;\\
    +\infty & \textrm{if } u\neq 0; 
\end{cases}$ is introduced to reformulate the root-finding problem as an 
optimization problem 
\begin{equation}
  \underset{U \in \mathbb{R}^{N_tN_x^2}}{\inf} ~ \iota(F(U)),  \label{min res }
\end{equation}
which can be further reduced to an inf-sup saddle problem 
\begin{equation}
   \inf_{ U \in \mathbb{R}^{N_tN_x^2}} \sup_{ P \in \mathbb{R}^{N_tN_x^2}} ~ P^\top{F}(U).  \label{intr: original min max with iota }
\end{equation}

Inspired by \cite{zuo2023primal}, we replace $\iota$ in \eqref{min res } by a milder quadratic function $\frac{1}{2\epsilon}\|\cdot\|^2 $ to obtain  
\begin{equation}
    \inf_{U\in\mathbb{R}^{N_tN_x^2}} ~ \frac{1}{2\epsilon}\|F(U)\|^2.  \label{intr: min problem with q epsilon }
\end{equation}
By introducing the dual variable $P\in\mathbb{R}^{N_tN_x^2}$, one can reformulate \eqref{intr: min problem with q epsilon } as an inf-sup problem with a tunable parameter $\epsilon$,
\begin{align}
  \inf_{U\in\mathbb{R}^{N_tN_x^2}} \sup_{P\in\mathbb{R}^{N_tN_x^2}} ~ L(U, P) & \triangleq P^\top F(U) - \frac{\epsilon}{2}\|P\|^2.  \label{intr: min-max problem with quad }
\end{align}

We tackle this saddle point problem by leveraging the primal-dual hybrid gradient (PDHG) algorithm and obtain 
\begin{align}
\label{derivation : original PDHG }
\begin{split}
   P_{n+1} = & \frac{1}{1+\epsilon \tau_P} \left(P_n + \tau_P F(U_n)\right),  
   \\
   \widetilde{P}_{n+1} = & P_{n+1} + \omega (P_{n+1} - P_n),
   \\
   U_{n+1} = & U_n - \tau_U DF(U_n)^\top\widetilde{P}_{n+1}.
\end{split}
\end{align}
When $DF(U)$ is nonsingular for arbitrary $U\in\mathbb{R}^{N_tN_x^2}$, the equilibrium state of the above discrete dynamic is $(U_*, 0)$ with $F(U_*)=0$. As discussed in the introduction, a large condition number of $F(\cdot)$ may significantly slow down the convergence speed of \eqref{derivation : original PDHG }. To mitigate this, we consider suitable preconditioning of $F(\cdot)$. Let us decompose $F(U)$ into its linear part and nonlinear part,
\begin{align}
 F(U) = & \mathscr D U + h_t \mathscr{G}_h ( a \mathscr{L}_{h} U + bf(U) ) - V \nonumber  \\ 
 = & (\mathscr D + a h_t \mathscr G_h \mathscr{L}_h)U + b h_t \mathscr G_h (f(\overline{U}) + J_f(U - \overline{U}) + R(U)) - V.    \label{decomp f }
\end{align}
Here we assume $\overline{U}$ is a certain point in $\mathbb{R}^{N_x^2  }$ at which we expand $f(U) = f(\overline{U}) + J_f(U-\overline{U}) + R(U)$. 
We choose matrix $J_f$ as an approximation of the Jacobian matrix $Df(\overline{U})=\mathrm{diag}(\dots,f'(\overline{U}_{ij}), \dots)$. We denote $R(U)\triangleq f(U)-f(\overline U)-J_f(U-\overline U)$ as the remainder term.
\begin{remark}
In practice, we usually choose $J_f = Df(u_{\mathrm{e}}\boldsymbol{1})$ where $\boldsymbol{1}$ is the $1-$vector, and $u_{\mathrm{e}}$ is one of the stable equilibrium states, i.e. $f(u_{\mathrm{e}}) = 0$. For example, in Allen-Cahn equation, $f(u) = u^3-u$, then $u_{\mathrm{e}}=\pm 1$, we always have $f'(u_{\mathrm{e}}) = 2$. Thus, we set $J_f = 2I$.
\end{remark}
By writing
\begin{align}
  & \mathscr{M} = \mathscr D + a h_t \mathscr G_h \mathscr L_h + bh_t \mathscr G_h J_f = \left[\begin{array}{cccc}
      X & & & \\
      -I & X & & \\
         & \ddots & \ddots & \\
         &        &   -I   & X
  \end{array}\right] ~ \textrm{with } X = I + a h_t \mathcal{G}_h\mathcal{L}_h + b h_t 
     \mathcal G_h J_f,  \label{def: scr M }   \\
  &  \widetilde{w} = bh_t\mathscr{G}_h(f(\overline U) - J_f \overline U)-V,  \nonumber
\end{align}
we decompose $F(U)$ as $\mathscr M U + b h_t  \mathscr{G}_h R(U) - \widetilde{w}.$ It is beneficial to consider the preconditioned function
\begin{equation}
  \widehat F(U) = \mathscr M^{-1} F(U) = U + \mathscr M^{-1} (bh_t\mathscr{G}_h R(U) ) - \widetilde{w} \overset{\textrm{denote as }}{=} U + \eta(U). \label{def: preconded F}
\end{equation}
We discuss the sufficient condition under which $\mathscr{M}$ is invertible in the following remark.
\begin{remark}[Invertibility of $\mathscr{M}$]
  Suppose $a,b\geq 0$, $\mathcal G_h$, $\mathcal L_h$ are self-adjoint, non-negative definite, and commute. Assume $J_f=cI$ with $c\geq 0$. Then $\mathscr{M}$ is invertible for any $h_t>0$. To prove this, it suffices to show that each $X$ is invertible. By similar arguments of the proof in Lemma \ref{lemm: posdef I+GL }, it is not hard to verify that $X$ is equivalent to $I+ah_t\Lambda_{\mathcal{G}_h}\Lambda_{\mathcal{L}_h}+bch_t\Lambda_{\mathcal{G}_h}$, which is invertible for $h_t>0$. Here $\Lambda_{\mathcal{G}_h}, \Lambda_{\mathcal{L}_h} $ are diagonal matrices equivalent to $\mathcal G_h, \mathcal L_h.$
\end{remark}
The corresponding root-finding problem $\widehat F(U) = 0$ is equivalent to the original problem \eqref{root-finding} whenever $\mathscr{M}$ is invertible. 

We now apply \eqref{derivation : original PDHG } to the inf-sup saddle problem with respect to $\widehat{F}(\cdot)$
\begin{equation}
    \inf_{U\in\mathbb{R}^{N_tN_x^2}}\sup_{Q\in\mathbb{R}^{N_tN_x^2}} \widehat{L}(U, Q) \triangleq Q^\top\widehat{F}(U) - \frac{\epsilon}{2}\|Q\|_2^2. \label{def: hat L(U, Q)}
\end{equation}
And our PDHG method with \textit{implicit} update in $Q$ and \textit{explicit} update in $U$ yields
\begin{equation}
\label{Preconded PDHG}
\begin{split}
  & Q_{k+1} = \frac{1}{1+\epsilon \tau_P}(Q_k + \tau_P(\widehat{F}(U_k))); \\
  & \widetilde{Q}_{k+1} = Q_{k+1} + \omega(Q_{k+1} - Q_k);\\
  & U_{k+1} = U_k - \tau_U ( D\widehat F(U_k)^\top \widetilde{Q}_{k+1} ).
\end{split}
\end{equation}
We then iterate \eqref{Preconded PDHG} so that $\{U_k\}$ approaches the desired root $U_*$. We terminate the iteration whenever the $\ell^{\infty}$ norm of the residual term
\begin{equation}
  \mathrm{Res}(U_k) = F(U_k)/h_t = \left[\dots, \left(\frac{\rmu_k^{t+1} - \rmu^t_k}{h_t} + \mathcal G_h(a\mathcal L_h \rmu^{t+1}_k  +  b f(\rmu^{t+1}_k))\right)^\top, \dots \right]^\top_{0\leq t\leq N_t-1}  .  \label{def: residual }
\end{equation}
is less than a certain tolerance $tol$, i.e., $\|\mathrm{Res}(U_k)\|_\infty < tol$.

\subsection{Complexity of the algorithm}
We apply the Fast Fourier Transform (FFT) \cite{cooley1965algorithm, strang1986proposal} to evaluate the multiplication of $\mathcal L_h, \mathcal G_h$ for periodic boundary conditions. Furthermore, the Discrete Cosine Transform (DCT) \cite{strang1999discrete} can be utilized to handle the no-flux or more general Neumann boundary conditions. We refer interested readers to \cite{liu2024first} for more details. Thus, computing $F(U)$ requires $\mathcal O(N_tN_x^2\log N_x)$ steps of operations. Furthermore, since $\mathscr{M}$ is block lower triangular, applying back substitution together with FFT/DCT to solve the linear system involving $\mathscr{M}$ requires $\mathcal O (N_tN_x^2\log N_x)$ steps of operations. Thus, the complexity at each iteration of our algorithm equals $\mathcal O(N_tN_x^2\log(N_x))$.

\subsection{Computing with time causality}\label{subsection  compute with time causality}
As mentioned in Section \ref{intro}, time causality can be incorporated into the numerical scheme by solving sequential blocks of numerical solutions:

\[  {U}^j = \begin{bmatrix} {\rmu^{j\cdot N_t + 1}}^\top, \dots, {\rmu^{j\cdot N_t + t}}^\top, \dots, {\rmu^{j\cdot N_t + N_t}}^\top \end{bmatrix}^\top  \in \mathbb{R}^{N_x^2N_t}, \quad j = 0, 1, 2, \dots  \]

More precisely, to compute \( U^j \) over the \( j \)-th time interval \([(j-1) \cdot N_t \cdot h_t, j \cdot N_t \cdot h_t]\), the proposed PDHG algorithm is applied to the root-finding problem \( F(U^j) = 0 \), with \( \rmu_0 = \rmu^{(j-1)\cdot N_t + N_t} \). That is, the initial value is set as the final state from the previous block \( U^{j-1} \). 

From a practical perspective, increasing \( N_t \) leads to higher memory consumption. From a theoretical point of view, as justified in Corollary \ref{practical corollary } and \ref{coro: simplify PDHG alg converge}, fixing the time step size \( h_t \) while selecting a large \( N_t \) may result in an ill-conditioned root-finding problem, posing challenges to the convergence of the method. In practice, choosing a moderate \( N_t \) (generally not exceeding $5$) mitigates these issues and ensures the efficient performance of the algorithm. Further discussions regarding hyperparameter selections are provided in Section \ref{sec: hyperparam select}.

A standard choice of the initial values \( (U_0, Q_0) \) for the PDHG algorithm \eqref{Preconded PDHG} upon solving \( F(U^j) = 0 \) is $U_0 = \widetilde{U}^{j}, Q_0 = 0,$ where \( \widetilde{U}^j \) denotes the numerical solution precomputed using the IMEX scheme with initial condition \( \rmu^{(j-1)\cdot N_t + N_t} \). A simpler alternative is to set $U_0 = U^{j-1}, Q_0 = 0.$ Both choices are efficient in practice as long as $N_t$ is not too large.

\subsection{Relation with G-prox PDHG method}
The G-prox primal-dual hybrid gradients algorithm \cite{jacobs2019solving} was recently invented to improve the convergence of optimization and root-finding problems. The algorithm can be formulated as
\begin{equation}
\begin{split}
  & P_{k+1} = \underset{P\in \mathbb{R}^{N_tN_x^2} }{\textrm{argmin}} ~ \left\{ \frac{1}{2\tau_P} \|P-P_k\|^2_G - \widehat{L}(U_k, P) \right\} = \frac{1}{1+\epsilon \tau_P}(P_k + \tau_P G^{-1} {F}(U_k)); \\
  & \widetilde{P}_{k+1} = P_{k+1} + \omega(P_{k+1} - P_k);\\
  & U_{k+1} = \underset{U\in\mathbb{R}^{N_tN_x^2}}{\textrm{argmin}} \left\{\frac{1}{\tau_P} \|U-U_k\|_2^2 + \widehat{L}(U,\widetilde{P}_{k+1}) \right\}.
\end{split}
\label{Gprox PDHG}
\end{equation}
Here we define the $G$-weighted norm as $\|v\|_G^2 = v^\top G v$, and pick $G = \mathscr{M}\mathscr{M}^\top$. In practice, we substitute the following explicit update of $U_k$ for the implicit update,
\begin{equation} 
  U_{k+1} = U_k - \tau_U  DF(U_k)^\top \widetilde{P}_{k+1} . \label{G-prox PDHG explicit update }
\end{equation}
Now, we multiply $\mathscr{M}^\top$ on both sides of \eqref{Gprox PDHG} (but with the third line replaced by \eqref{G-prox PDHG explicit update }) to obtain
\begin{equation}
\begin{split}
  & \mathscr{M}^\top P_{k+1} = \frac{1}{1+\epsilon \tau_P}(\mathscr{M}^\top P_k + \tau_P \mathscr{M}^{-1} {F}(U_k)); \\
  & \mathscr{M}^\top\widetilde{P}_{k+1} = \mathscr{M}^\top P_{k+1} + \omega(\mathscr{M}^\top P_{k+1} - \mathscr{M}^\top P_k);\\
  & U_{k+1} = U_k - \tau_U D\widehat{F}(U_k)^\top (\mathscr{M}^\top \widetilde{P}_{k+1}).
\end{split}\label{equivalent PDHG dynamic }
\end{equation}
By denoting $Q_k = \mathscr{M}^\top P_k$ and noticing that $\widehat{F}(U) = \mathscr{M}^{-1}F(U)$, \eqref{equivalent PDHG dynamic } reduces exactly to \eqref{Preconded PDHG}. This verifies the equivalence between the $G$-prox PDHG algorithm and our proposed method.

\section{Numerical analysis of the proposed method}
In this section, we study the numerical convergence properties of the proposed PDHG algorithm. In subsection \ref{sec: unique existence}, we prove the unique solvability of the time-implicit scheme \eqref{time-implicit scheme} of RD equations. In subsection \ref{sec: Lyapunov analysis PDHG flow time continue }, we study the convergence of the time-continuous limit of the PDHG algorithm. In subsection \ref{sec: lyapunov time discrt case}, we prove the convergence of the PDHG algorithm.

\subsection{Unique solvability of the time-implicit scheme}\label{sec: unique existence}
In this research, we mainly focus on reaction functions $f$ that belong to the functional space $\mathcal F$, where 
\begin{equation}
  \mathcal{F} = \left\{ f \in C^1(\mathbb{R}) \; \middle| \;
  \begin{tabular}{@{}l@{}}
     $f$ can be decomposed as $f = V' + \phi$,\\
     where $V\in C^1(\mathbb{R})$ is convex, and $\phi\in C(\mathbb{R})$ {is Lipschitz.}
  \end{tabular}
  \right\}.  \label{space F }
\end{equation}
The space $\mathcal F$ covers a majority of reaction functions that arise in classical RD equations such as the Allen-Cahn and the Cahn-Hilliard equations.

Before we present the result, we assume the spectral decomposition of $\mathcal G_h$: 
\begin{equation}
\mathcal G_h = \left[\begin{array}{c|c}
    Q_1 & Q_2 
\end{array}\right] \left[\begin{array}{cc}
    \Lambda &  \\
     &  O
\end{array}\right]\left[ \begin{array}{c}
     Q_1^\top  \\
     Q_2^\top
\end{array} \right],\label{spectral decomp   Gh}
\end{equation}
where $\Lambda = \mathrm{diag}({\lambda}_1,\dots,\lambda_{r})$ is a diagonal matrix with positive entries $\lambda_{ 1 }\geq\dots\geq\lambda_{r}>0$, $r = \mathrm{rank}(\mathcal G_h)$. 
\begin{theorem}[Existence and uniqueness of \eqref{root-finding}]\label{thm: unique existence root-finding}
  Suppose that $\mathcal G_h$, $\mathcal L_h$ used in the finite difference scheme \eqref{time-implicit scheme} are self-adjoint and positive semidefinite. Assume $\mathcal G_h$ has the spectral decomposition as in \eqref{spectral decomp   Gh}. We also assume that $f\in\mathcal{F}$, such that the convex function $V$ satisfies
   \[ (V'(x)-V'(y), x-y) \geq K|x-y|^2, \]
   for some $K\geq 0.$ If the time step size $h_t$ in \eqref{time-implicit scheme} satisfies 
   \begin{equation}
     \lambda_{\min}\left(\frac{\Lambda^{-1}}{h_t} + a \; Q_1^\top \mathcal L_h Q_1 \right) + b K > b \; \mathrm{Lip}(\phi),  \label{condition on existence and uniqueness}
   \end{equation}
   then the root-finding problem \eqref{root-finding} admits a unique solution.
\end{theorem}
The proof of the theorem is deferred to Appendix \ref{app: proof th}.

\begin{remark}\label{remark on condition ht exist & unique }
The condition \eqref{condition on existence and uniqueness} can be simplified for some specific equations.
\begin{itemize}
    \item (Allen-Cahn equation with periodic boundary condition) $\mathcal G = \mathrm{Id}, \mathcal L = - \Delta$, $f(x) = x^3 - x$. We set $\mathcal G_h = I_{N_x^2}$, and $\mathcal L_h = - \Delta_{h_x}^P=I_{N_x} \otimes (-\mathrm{Lap}_{h_x}^P) + (-\mathrm{Lap}_{h_x}^P) \otimes I_{N_x}$, where $\mathrm{Lap}_{h_x}^P$ is defined in \eqref{def: periodic discrete Laplace operator }. In this case, the condition \eqref{condition on existence and uniqueness} yields $h_t < \frac{1}{2b}.$ 
    \item (Cahn-Hilliard equation with periodic boundary condition) $\mathcal G = -\Delta$, $\mathcal L= -\Delta$, $f(x)=x^3 - x.$ We set $\mathcal G_h = \mathcal L_h = -\Delta_{h_x}^P.$ A sufficient condition for \eqref{condition on existence and uniqueness} is $h_t < \frac{a^2}{b^2}.$
\end{itemize}
Similar results regarding both Allen-Cahn and Cahn-Hilliard equations have also been done in \cite{xu2019stability}. Theorem \ref{thm: unique existence root-finding} applies to general RD equation \eqref{general RD type PDE}. We refer interested readers to Appendix \ref{append: discuss on condition on ht} for more detailed discussions.
\end{remark}

\subsection{Lyapunov analysis for the PDHG flow}\label{sec: Lyapunov analysis PDHG flow time continue }
We are ready to present the main result of this paper. In subsection \ref{subsec: lyapunov analysis for general root-finding }, we first prove the convergence of the time-continuous limit of the PDHG algorithm \eqref{Preconded PDHG} for the general root-finding problem. In subsection \ref{subsec: convergence result for time continuous case }, we apply the previous theory to the time-implicit scheme of RD equations. In subsection \ref{subsec: numerical verification }, we provide numerical justifications for the theoretical study. To alleviate the notation, we denote $\|\cdot\|$ as the $2-$norm for both vectors and matrices in the following discussion.

\subsubsection{Convergence analysis for the general root-finding problem }\label{subsec: lyapunov analysis for general root-finding }
Firstly, we establish the convergence result for a general root-finding problem $\widehat{F}(U)=0$ regardless of the exact form of $\widehat F(U)$. Our main results are summarized in Theorem \ref{theorem : exponential decay} and Corollary \ref{main coro}.

Recall \eqref{Preconded PDHG}, we substitute $\widetilde{Q}_{k+1}$ with 
\begin{align*} 
  \widetilde Q_{k+1} = Q_k + (1+\omega)(Q_{k+1} - Q_k) =&  Q_k + (1+\omega)\tau_P\left(-\frac{\epsilon}{1+\epsilon \tau_P}Q_k + \frac{1}{1+\epsilon \tau_P}\widehat{F}(U_k)\right) \\
  = & \left(1-\frac{(1+\omega)\tau_P\epsilon}{1+\epsilon\tau_P}\right)Q_k + \frac{(1+\omega)\tau_P}{1+\epsilon\tau_P}\widehat{F}(U_k) .
\end{align*}
Then, the PDHG iteration \eqref{Preconded PDHG} can be formulated as
\begin{equation}
\label{ reformulate Preconded PDHG for prove limit }
\begin{split}
  & \frac{Q_{k+1}-Q_k}{\tau_P} = -\frac{\epsilon}{1+\epsilon \tau_P}Q_k+\frac{1}{1+\epsilon \tau_P}\widehat{F}(U_k); \\
  & \frac{U_{k+1}-U_k}{\tau_U} = -D\widehat F(U_k)^\top \left(\left(1-\frac{(1+\omega)\tau_P\epsilon}{1+\epsilon\tau_P}\right)Q_k + \frac{(1+\omega)\tau_P}{1+\epsilon\tau_P}\widehat{F}(U_k)\right).
\end{split}
\end{equation}
Suppose we send the step sizes $\tau_U, \tau_P \rightarrow 0$, and keep $\omega$ increasing such that $(1+\omega)\tau_P\rightarrow \gamma>0$. Then the above time-discrete dynamic will converge to the following time-continuous dynamic of $(U_t, Q_t)$ which we denote as the ``PDHG flow''.
\begin{equation}
\label{PDHG cont time}
\left\{\begin{aligned}
    \dot Q  = &  -\epsilon Q + \widehat F(U), \\
    \dot U  = &  - D\widehat{F}(U)^\top((1-\gamma\epsilon)Q  +  \gamma \widehat F(U)).
\end{aligned}\right.
\end{equation}
We introduce two notations that will be commonly used in the following discussion,
\begin{align} 
    \underline \sigma = & \inf_{U\in \mathbb{R}^{ N_t N_x^2}} \{ \sigma_{\min}(D\widehat F(U)) \} = \inf_{U\in \mathbb{R}^{ N_t N_x^2}} \{\sigma_{\min}(I + bh_t\mathscr{M}^{-1}\mathscr{G}_hDR(U))\},  \label{def: low bdd sigma}\\
    \overline \sigma = & \sup_{U\in \mathbb{R}^{N_t N_x^2}} \{ \sigma_{\max}(D\widehat F(U)) \} = \sup_{U\in \mathbb{R}^{N_t N_x^2}} \{\sigma_{\max}(I + bh_t\mathscr{M}^{-1}\mathscr{G}_hDR(U))\},  \label{def: up bdd sigma}
\end{align}
where $\sigma_{\min}(A) $($\sigma_{\max}(A)$) denotes the minimum (maximum) singular value of matrix $A$. The condition number is defined by 
\begin{equation}
  \kappa = {\overline\sigma}/{\underline\sigma}. \label{def: global cond num of DF }
\end{equation}

We consider the following Lyapunov function of $(U, Q)$ associated with a parameter $\mu>0$, 
\begin{equation}
    \mathcal I_{\mu}(U, Q) = \frac{1}{2} \|\widehat{F}(U)\|^2 + \frac{\mu}{2}\|Q\|^2. \label{def: Lyapunov function }
\end{equation}
The parameter $\mu$ enables us to establish the exponential decay of $\mathcal I_\mu(U_t, Q_t)$ along the PDHG flow whenever $0<\underline\sigma\leq \overline{\sigma}<\infty$. We have the following Lemma.

\begin{lemma}[Exponential decay of $\mathcal I_\mu(U_t, Q_t)$]\label{theorem : exponential decay}
  Suppose that $0<\underline{\sigma} \leq \overline{\sigma} < \infty$. We pick the parameter $\mu > 0$ satisfying 
  \begin{equation}
     \frac{1}{\underline \sigma} - \frac{1}{\overline \sigma} < \frac{2}{\sqrt{\mu}}.     \label{condition on low sigma up sigma}
  \end{equation}
  Furthermore, we choose $\gamma, \epsilon>0$ satisfying 
  \begin{equation}
     \max\left\{ \left(1-\frac{\sqrt{\mu}}{\overline\sigma}\right)^2, \left(1-\frac{\sqrt{\mu}}{\underline\sigma}\right)^2 \right\} < 
    \gamma\epsilon < \left( 1 + \frac{\sqrt{\mu}}{\overline \sigma } \right)^2 .  \label{condition on gamma epsilon}
  \end{equation}
  Under the above choices of $\mu$, $\gamma$ and $\epsilon$, let $(U_t, Q_t)$ be the solution to the PDHG flow \eqref{PDHG cont time} with arbitrary initial condition $(U_0, Q_0)$. Then we have, 
  \begin{equation*}
    \mathcal I_\mu (U_t, Q_t) \leq \exp \left( {- \frac{2 \beta \; t }{ \max\{1, \mu \} }  } \right) \; \mathcal I_\mu (U_0, Q_0).
  \end{equation*}
  Here we denote
  \begin{equation*}
    \beta =  \min_{z\in [\underline\sigma^2, \overline\sigma^2]} \; \{\varphi_{ \mu,  \gamma,  \epsilon }(z)\} > 0,
  \end{equation*}
  with $\varphi_{\mu, \gamma, \epsilon}(z) = \frac{1}{2}(\gamma z + \mu \epsilon - \sqrt{(\gamma z - \mu \epsilon)^2 + (\mu - (1-\gamma\epsilon)z)^2} ).$
\end{lemma}
We defer the proof of this Lemma to Appendix \ref{append: lyapunov analysis of general root finding }. Lemma \ref{theorem : exponential decay} provides a sharp convergence rate for $\mathcal I_\mu(U_t, Q_t)$. However, $\beta$ does not take an explicit form. In the following theorem, we relax the bound in Lemma \ref{theorem : exponential decay} to obtain an explicit convergence rate for $\|\widehat{F}(U_t)\|$.

\begin{theorem}[Exponential decay of the residual $\|\widehat{F}(U_t)\|$] \label{main coro} 
Assume that $(U_t, Q_t)$ solves \eqref{PDHG cont time} with an arbitrary initial position $(U_0, Q_0)$. Then, as long as $\underline \sigma $ is bounded away from $0$ and $\overline{\sigma}$ is finite, one can always pick suitable parameters $\gamma$, $\epsilon$ such that the residual $\|\widehat F(U_t)\|$ decays exponentially fast to $0$. In particular, if we set $\epsilon = (1-\delta)\kappa$ and $\gamma = \frac{1-\delta}{\kappa}$ with $|\delta| < \frac{1}{\kappa} $, then we have 
  \begin{equation*}
    \|\widehat{F}(U_t)\|_2 \leq \exp\left(-(1-\kappa|\delta|) (3-\delta)\frac{\min\{\underline\sigma^2, 1\}}{8\kappa} \; t \right) \; \sqrt{ \|\widehat F(U_0)\|^2 + \underline\sigma^2 \|Q_0\|^2 }.
  \end{equation*}
\end{theorem}
The proof is provided in Appendix \ref{append: lyapunov analysis of general root finding }. We can further improve the convergence rate by fixing $\gamma\epsilon=1$ in Theorem \ref{thm : exponential decay version 2} of Appendix \ref{append: lyapunov analysis of general root finding }.

\subsubsection{Convergence analysis for our specific root-finding problem \eqref{def: preconded F}}\label{subsec: convergence result for time continuous case }

In this section, we discuss the exponential decay of the PDHG flow \eqref{PDHG cont time} when it is applied to the time-implicit scheme \eqref{time-implicit scheme} of the RD equation \eqref{general RD type PDE} when the reaction term $f(\cdot)$ is Lipschitz. The main results of this section are Theorem \ref{thm: O(1) exponential convergence result for general RD equ } and Corollary \ref{practical corollary }.

Before demonstrating our result, we list several conditions regarding equation \eqref{general RD type PDE} and its numerical scheme \eqref{time-implicit scheme}. These conditions will be used later.
\begin{enumerate}
    \item Suppose the coefficients $a,b$ are non-negative, i.e.,
    \begin{equation}
        a\geq 0, \quad b\geq 0. \label{condition: a b >= 0 } \tag{A}
    \end{equation}
    \item Assume that
    \begin{equation}
      f(\cdot) \textrm{ is Lipschitz with constant } \mathrm{Lip}(f).  \label{condition: f Lip } \tag{B}
    \end{equation}
    \item In the numerical scheme \eqref{time-implicit scheme} of \eqref{general RD type PDE}, suppose 
    \begin{equation}
      \mathcal L_h , \mathcal G_h  ~ \textrm{are self-adjoint, non-negative definite, and commute, i.e., } \mathcal G_h \mathcal L_h = \mathcal L_h \mathcal G_h. \label{condition: G,L>=0, symm, commut } \tag{C}
    \end{equation} 
    \item Recall $J_f$ mentioned in \eqref{decomp f }. We assume 
    \begin{equation}
      J_f \textrm{ is a constant diagonal matrix } cI \textrm{ with } c\geq 0. \label{condition Jf=cI c>=0 }\tag{D}
    \end{equation}
\end{enumerate}

\begin{remark}
  We point out that many reaction-diffusion equations do not possess Lipschitz reaction terms $f(\cdot)$: Double-well polynomial potential in the phase field model, as well as logarithmic Flory-Huggins potential, does not yield Lipschitz reaction functions \cite{langer1986models, flory1942thermodynamics}. However, the Lipschitz assumption can still be applied if one can prove an \textit{a priori} estimation on $\ell^\infty$ norm of the numerical solution $U_k$ for all PDHG iteration $k$. This may serve as our future research topic.
\end{remark}

As stated in Theorem \ref{main coro}, we need $ \underline\sigma > 0 $ and $\overline{\sigma}<\infty$ in order to establish the exponential decay of $\|\widehat F(U)\|$. Lemma \ref{lemm: est singular values DF } provides a sufficient condition for this to hold.

\begin{lemma}\label{lemm: est singular values DF }
  Suppose \eqref{condition: a b >= 0 }, \eqref{condition: f Lip }, \eqref{condition: G,L>=0, symm, commut } hold. When $h_t < \frac{1}{|b|\lambda_{\max}(\mathcal G_h)\mathrm{Lip}(f)}$, we always have $\underline\sigma > 0$ and $\overline{\sigma}<\infty$.
\end{lemma}
We prove this Lemma in Appendix \ref{append: lyapunov analysis of specific root-finding }. Combining Theorem \ref{main coro} and Lemma \ref{lemm: est singular values DF } leads to the following Theorem \ref{thm: exponential convergence regardless of b T }. 
\begin{theorem}[First convergence result of $\|\widehat{F}(U_t)\|$]\label{thm: exponential convergence regardless of b T }
  Consider the RD equation \eqref{general RD type PDE} on $[0,T]$. Suppose \eqref{condition: a b >= 0 }, \eqref{condition: f Lip } and \eqref{condition: G,L>=0, symm, commut } hold. We apply the PDHG flow \eqref{PDHG cont time} to solve the time-implicit scheme \eqref{time-implicit scheme} with time step size $h_t < \frac{1}{|b|\|\mathcal G_h\|\mathrm{Lip}(f)}$. Suppose $\gamma = \frac{1-\delta}{\kappa}$, and $\epsilon=(1-\delta)\kappa$ with $\kappa={\overline{\sigma}}/{\underline \sigma}$, and $|\delta|<\frac{1}{\kappa}$. Then $\|\widehat F(U_t)\|$ converges exponentially fast to $0$.
\end{theorem}
\begin{remark}
It is worth mentioning that we do not assume condition \eqref{condition on existence and uniqueness} of Theorem \ref{thm: unique existence root-finding}. Then $\widehat{F}(U)=0$ might not admit a unique solution, but the exponential decay of $\|\widehat{F}(U_t)\|$ is still guaranteed.    
\end{remark}

Although Theorem \ref{thm: exponential convergence regardless of b T } guarantees the exponential convergence of $\|\widehat F(U_t)\|$ for arbitrarily large $b$ and $T$ as long as $h_t, \gamma, \epsilon$ are suitably chosen, both the time step size $h_t$ and the convergence rate may depend on the spatial discretization $N_x$. To get rid of this dependency, we provide sufficient conditions under which $\underline\sigma$ and $\overline{\sigma}$ are bounded away from the constants that are independent of $N_x$. Thus, we achieve a convergence rate that is independent of $N_x$. Recall the remainder term $R(U) = f(U)-f(\overline{U})-Df(\overline{U})(U - \overline{U})$ mentioned in \eqref{decomp f }. We have the following Lemma.

\begin{lemma}\label{lemm: more sophisticate est on singular values of DF }
Consider the reaction-diffusion type equation \eqref{general RD type PDE} on $[0, T ].$ Suppose the conditions \eqref{condition: a b >= 0 }, \eqref{condition: f Lip }, \eqref{condition: G,L>=0, symm, commut } and \eqref{condition Jf=cI c>=0 } hold. Since \eqref{condition: f Lip } requires $f$ to be Lipschitz, so does $R$. And we denote its Lipschitz constant as $\mathrm{Lip}(R).$ Define
 \begin{equation*}
    \zeta_{a,b,c}(h_t) = \max_{1\leq k\leq N_x^2} \left\{ \frac{\lambda_k(\mathcal G_h)}{1 + h_t(a\lambda_k(\mathcal{G}_h)\lambda_k(\mathcal{L}_h) + bc\lambda_k(\mathcal{G}_h))}  \right\},
 \end{equation*}
 where $\lambda_k(\mathcal G_h), \lambda_k(\mathcal L_h)$ are the eigenvalues of $\mathcal G_h$, $\mathcal L_h$ which are simultaneously diagonalizable by an orthogonal matrix $Q$. Recall that in \eqref{def: preconded F}, we have $\eta (U) = b h_t \mathscr{M}^{-1} \mathscr{G}_h R(U)-\widetilde{\mathbf{w}}$, then
 \begin{equation*}
   \|D\eta(U)\|\leq bT \zeta_{a,b,c}(h_t) \mathrm{Lip}(R)\,.
 \end{equation*}
 And we also have
 \begin{equation*} 
   \underline\sigma \geq 1 - bT \zeta_{a,b,c}(h_t) \mathrm{Lip}(R), \quad \overline{\sigma} \leq1 + bT \zeta_{a,b,c}(h_t) \mathrm{Lip}(R).
 \end{equation*}
\end{lemma}

We prove this lemma in Appendix \ref{append: lyapunov analysis of specific root-finding }. A direct corollary of Lemma \ref{lemm: more sophisticate est on singular values of DF } and Theorem \ref{main coro} is Theorem \ref{thm: O(1) exponential convergence result for general RD equ }, which not only guarantees the unique solvability of $\widehat{F}(U)=0$, but also establishes exponential convergence for $\|\widehat{F}(U_t)\|$.

\begin{theorem}[Unique existence of the root \& the second convergence result of $\|\widehat{F}(U_t)\|$]\label{thm: O(1) exponential convergence result for general RD equ }
Suppose conditions \eqref{condition: a b >= 0 }, \eqref{condition: f Lip }, \eqref{condition: G,L>=0, symm, commut } and \eqref{condition Jf=cI c>=0 } hold. We pick $h_t$ and $T=N_t h_t$ ($N_t\in\mathbb{N}_+$) satisfying
\begin{equation}
  bT\mathrm{Lip}(R)\zeta_{a,b,c}(h_t) < 1. \label{condition on ht zeta less than theta }
\end{equation}
Then there exists a unique root of $\widehat{F}$. Furthermore, we denote $\theta = bT\mathrm{Lip}(R)\zeta_{a,b,c}(h_t)<1$. Suppose we set $\epsilon = \kappa-\frac12$ and $\gamma = \frac{1}{\kappa} - \frac{1}{2\kappa^2}$. Then we have
  \begin{equation}
    \|\widehat{F}(U_t)\| \leq \exp\left(-\frac{5}{32} \cdot \frac{(1-\theta)^3}{1+\theta} \; t \right) \sqrt{\|\widehat F(U_0)\|^2 + (1+\theta)\|Q_0\|^2}. 
    \label{O(1) convergence rate }
 \end{equation}
\end{theorem}
\begin{proof}
  The unique existence of the root for $\widehat{F}(\cdot)$ is due to Lemma \ref{lemm: sufficient cond unique existence of root }. 
   
  We now prove the exponential convergence \eqref{O(1) convergence rate }. According to Lemma \ref{lemm: more sophisticate est on singular values of DF }, by letting $\theta = bT\mathrm{Lip}(R)\zeta_{a,b,c}(h_t)$, we obtain 
  \begin{equation}
    \underline\sigma \geq 1-\theta, \quad \overline{\sigma}\leq 1+\theta, \quad \textrm{and thus} \quad \kappa \leq \frac{1+\theta}{1-\theta}.  \label{bounded sigma and k } 
  \end{equation}
  Now recall Theorem \ref{main coro}. To alleviate our discussion, we choose $\delta = \frac{1}{2\kappa}$. After setting the parameters $\epsilon = \kappa-\frac12$ and $\gamma = \frac{1}{\kappa} - \frac{1}{2\kappa^2}$, we have
\begin{align}
  \|\widehat{F}(U_t)\| & \leq \exp\left( - \frac{1}{2} (3-\frac{1}{2\kappa})\frac{\min\{\underline\sigma^2, 1\}}{8\kappa} \; t \right) \; \sqrt{ \|\widehat{F}(U_0)\|^2 + \underline\sigma^2 \|Q_0\|^2 } \nonumber\\
  & \leq \exp\left( -\frac12\cdot\frac{5}{2} \cdot \frac{(1-\theta)^3}{8(1+\theta)} \; t \right)\sqrt{\|\widehat F(U_0)\|^2 + (1+\theta)\|Q_0\|^2} 
 \nonumber \\
  & = \exp\left(-\frac{5}{32}\cdot\frac{(1-\theta)^3}{1+\theta} \; t \right) \sqrt{\|\widehat F(U_0)\|^2 + (1+\theta)\|Q_0\|^2},  \nonumber
\end{align}
 where the second inequality is due to \eqref{bounded sigma and k } and the fact that $\kappa \geq 1 $. 
\end{proof}

We can simplify condition \eqref{condition on ht zeta less than theta } for specific types of RD equations. This is summarized in the following Corollary.

\begin{corollary}[$N_x$-independent convergence rate for specific RD equations]\label{practical corollary }
  Suppose the conditions \eqref{condition: a b >= 0 }, \eqref{condition: f Lip }, \eqref{condition: G,L>=0, symm, commut } and \eqref{condition Jf=cI c>=0 } hold. 
  We pick $T=N_t h_t$ ($N_t\in\mathbb{N}_+$) such that
  \begin{itemize}
     \item (Allen-Cahn type, $\mathcal G_h = I$, $\mathcal L_h$ is self-adjoint, non-negative definite)
     $T < \frac{1}{b\mathrm{Lip}(R)}$, or equivalently, pick $h_t < \frac{1}{b\mathrm{Lip}(R)} \textrm{ and } N_t \leq \Bigl\lfloor\frac{1}{b\mathrm{Lip}(R)h_t}\Bigr\rfloor$. We denote $\widetilde\theta = b \mathrm{Lip}(R) T < 1 $.
     \item (Cahn-Hilliard type, $\mathcal G_h = \mathcal L_h$ are self-adjoint, and non-negative definite)
     $T < \frac{2\sqrt{ah_t} + bc h_t}{b\mathrm{Lip}(R)}$, or equivalently, pick $h_t <  \frac{4a}{b^2(\mathrm{Lip}(R)-c)_{+}^2} \textrm{ and } N_t  \leq  \Bigl\lfloor \frac{2\sqrt{a/h_t} + bc}{b\mathrm{Lip}(R)} \Bigr\rfloor$. We denote $\widetilde\theta = \frac{b\mathrm{Lip}(R) T }{2\sqrt{ah_t}+bch_t} < 1$.
  \end{itemize}
  Suppose further that $\epsilon = \kappa-\frac12$ and $\gamma = \frac{1}{\kappa} - \frac{1}{2\kappa^2}$, then $\|\widehat{F}(U_t)\|$ convergences to $0$ exponentially fast,
  \begin{equation}
    \|\widehat{F}(U_t)\| \leq \exp\left(-\frac{5}{32} \cdot \frac{(1-\widetilde\theta)^3}{1 + \widetilde\theta} \; t \right) \sqrt{\|\widehat F(U_0)\|^2 + (1+\widetilde\theta)\|Q_0\|^2}.   \label{O(1) convergence for AC CH equations }
 \end{equation}
\end{corollary}

\begin{proof}
Recall that we have $\theta = bT\mathrm{Lip}(R)\zeta_{a,b,c}(h_t)$. We prove $\widetilde\theta \geq \theta$ under both cases. 
\begin{itemize}
    \item (Allen-Cahn type) Note that $\zeta_{a,b,c}(h_t) = \max_k\left\{\frac{1}{1+h_t(a\lambda_k(\mathcal L_h) + bc)}\right\} \leq 1$. Thus,
    $$\theta = bT\mathrm{Lip}(R)\zeta_{a,b,c}(h_t) \leq bT\mathrm{Lip}(R) = \widetilde \theta.$$

    \item (Cahn-Hilliard type)  We have
    \begin{align*} 
      \zeta_{a,b,c}(h_t) &= \max_{k} \left\{ \left( {\frac{1}{\lambda_k(\mathcal G_h)} + h_t a \lambda_k(\mathcal L_h) + h_tbc } \right)^{-1} \right\} \\
     &= \max_{k} \left\{ \left( {\frac{1}{\lambda_k(\mathcal L_h)} + h_t a \lambda_k(\mathcal L_h) + h_tbc } \right)^{-1} \right\}\\
     & \leq \frac{1}{2\sqrt{ah_t} + bch_t}. 
    \end{align*}
    Then, 
    $$\theta = bT\mathrm{Lip}(R)\zeta_{a,b,c}(h_t) \leq \frac{bT\mathrm{Lip}(R)}{2\sqrt{ah_t}+bch_t} = \widetilde\theta. $$
\end{itemize}
Since $\widetilde\theta < 1$ in both cases, we have $\theta \leq \widetilde{\theta}<1$. Applying Theorem \ref{thm: O(1) exponential convergence result for general RD equ } yields \eqref{O(1) convergence rate }. Note that $\frac{(1-\theta)^3}{1+\theta}\geq \frac{(1-\widetilde\theta)^3}{1+\widetilde \theta}$ for $ 0\leq\theta \leq \widetilde\theta < 1$. This implies our result \eqref{O(1) convergence for AC CH equations }.
\end{proof}

\subsubsection{Numerical verification}\label{subsec: numerical verification }
We apply our algorithm to solve the Allen-Cahn equation \eqref{AC_equ} with $\epsilon_0=0.01$ on a $64\times 64$ grid. We use $\tau_U = \tau_P = 0.5$, $\omega = 1$, $\epsilon = 0.1$. At each iteration $k$, denote $U_k$ as the numerical solution. We define $r_k = -\log_{10}(\|\widehat{F}(U_{k+1})\|/\|\widehat{F}(U_k)\|)$ to be the convergence rate of the residual term $\|\widehat{F}(U_k)\|$ at $k$th iteration. The residual is expected to converge linearly to $0$. We denote by $\bar{r}$ the average convergence rate of the first $500$ iterations. 
By \eqref{O(1) convergence for AC CH equations }, when $\widetilde\theta$ is small, the convergence rate is $\frac{5}{32}(1-4\widetilde{\theta}+\mathcal{O}(\widetilde{\theta}^2))$, which is linear w.r.t. $N_th_t$ (recall that $\widetilde{\theta}\propto T = N_t h_t$). Such linear relation is verified in the first two figures of Figure \ref{fig: res converge rate vs ht Nt AC }. In the third figure, we observe fast decay of the average convergence rate $\bar{r}$ as $\widetilde\theta \propto N_th_t$ keeps increasing. Furthermore, we verify the dependence of the convergence rate on $N_th_t$ via the left plot of Figure \ref{fig: plt rate vs log Nt vs log ht}.
\begin{figure}[htb!]
    \centering
    \begin{subfigure}{0.325\textwidth}
    \captionsetup{justification=centering}
          \includegraphics[trim={3cm 8.5cm 4cm 8.5cm}, clip, width=\linewidth]{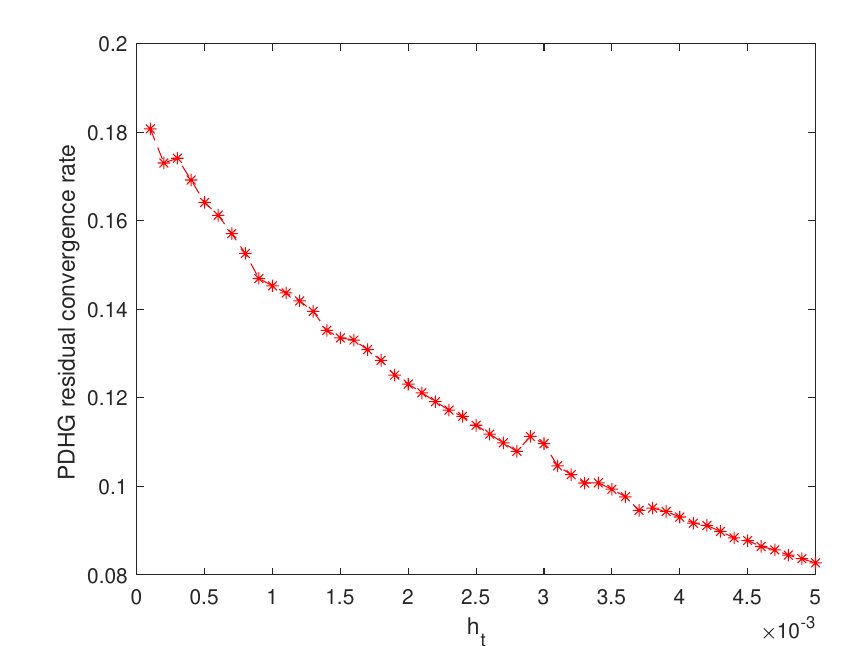}
          \subcaption{Plot of 
          $\bar{r}$ vs $h_t$. Fix $N_t=1$, $h_t=10^{-4}k$, $1 \leq k \leq 50$.}
    \end{subfigure}
    \begin{subfigure}{0.325\textwidth}
    \captionsetup{justification=centering}
          \includegraphics[trim={3cm 8.5cm 4cm 8.5cm}, clip, width=\linewidth]{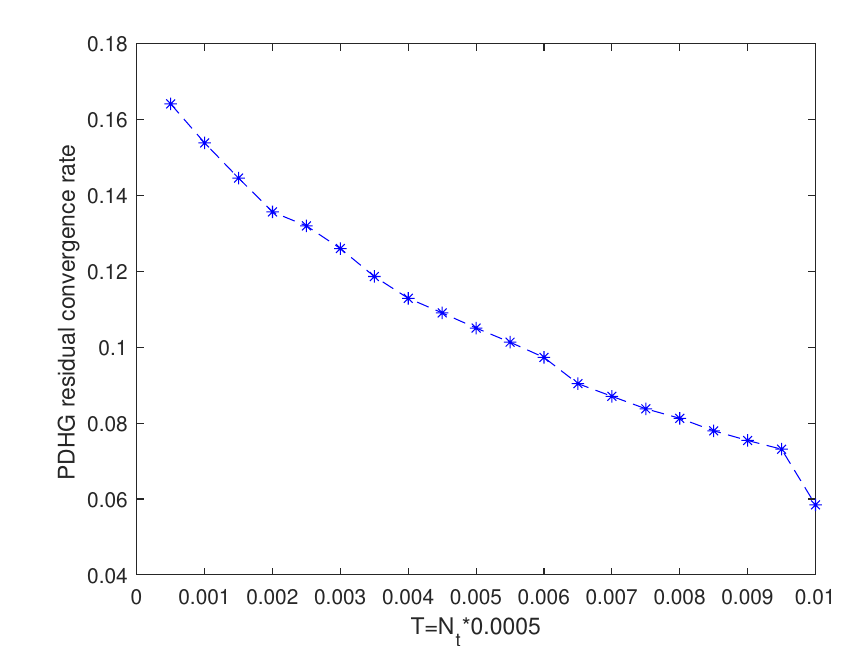} 
          \subcaption{ Plot of $\bar{r}$ vs $N_t$. \\
          Fix $h_t=5\times10^{-4}$, $1\leq N_t\leq 20$.}
    \end{subfigure}
    \begin{subfigure}{0.325\textwidth}
    \captionsetup{justification=centering}
          \includegraphics[trim={3cm 8.5cm 4cm 8.5cm}, clip, width=\linewidth]{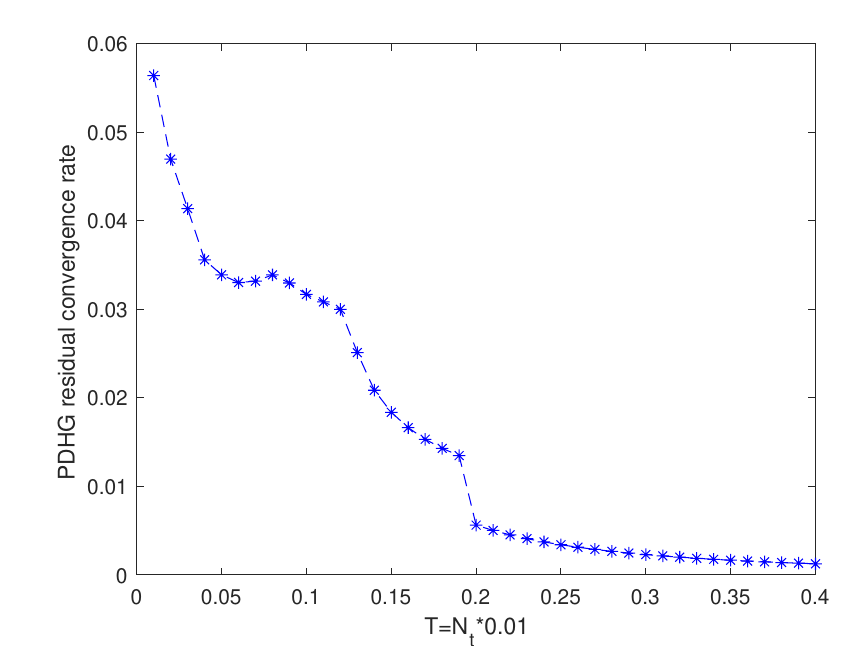} 
          \subcaption{ Plot of $\bar{r}$ vs $N_t$. \\
          Fix $h_t=10^{-2}$, $1\leq N_t\leq 40$.}
    \end{subfigure}
    \caption{Convergence rate of the residual term $\|\widehat{F}(U_k)\|$ w.r.t. $h_t, N_t$ for Allen-Cahn equation.  
    }
    \label{fig: res converge rate vs ht Nt AC }
\end{figure}
We also apply our algorithm to the Cahn-Hilliard equation \eqref{CH equ} with $\epsilon_0 = 0.1$ on a $64\times 64$ grid. We keep the hyperparameters the same as in the case of Allen-Cahn. The average convergence rate $\bar{r}$ is computed by the first $500$ iterations of the algorithm. By \eqref{O(1) convergence for AC CH equations }, the convergence rate is linear w.r.t. $N_t(\sqrt{h_t}+o(\sqrt{h_t}))$ when $\widetilde{\theta}\propto N_t \sqrt{h_t}$ is small. This is reflected in Figure \ref{fig: res converge rate vs ht Nt CH}. Unlike the case of Allen-Cahn, in which the PDHG algorithm converges as $\widetilde\theta$ increases, the iterations for Cahn-Hilliard diverges as $\widetilde\theta\propto N_t \sqrt{h_t} $ increases. This is reflected on the right plot of Figure \ref{fig: res converge rate vs ht Nt CH}. 
 
For a fixed time step size $h_t$, denote by $N_{\max}$ the maximum number of time steps that guarantees the convergence of the PDHG algorithm. We plot the relation between $N_{\max}$ and $h_t$ on a logarithmic scale in Figure \ref{fig: N max h_t log log plot }. We observe the relation $N_{\max}=\mathcal{O}(\frac{1}{\sqrt{h_t}})$ when the step size $h_t$ is not too small. The dependence of the convergence rate w.r.t. $N_t\sqrt{h_t}$ is shown in the right plot of Figure \ref{fig: plt rate vs log Nt vs log ht}. 
 
\begin{remark}
  It is worth mentioning that some of the tested values of $h_t$ in Figure \ref{fig: res converge rate vs ht Nt CH}, \ref{fig: N max h_t log log plot } may have exceeded the theoretical bounds for uniqueness (cf. Remark \ref{remark on condition ht exist & unique }) and convergence (cf. \ref{coro: simplify PDHG alg converge}). We point out that these bounds, derived from ensuring convexity and positive definiteness in the numerical analysis, are sufficient but not necessary. The figures primarily aim to illustrate the dependence of convergence rate on $N_t \cdot h_t$ and $N_t \cdot \sqrt{h_t}$, rather than strictly adhering to the bounds.
\end{remark}

\begin{figure}[htb!]
    \centering
    \begin{subfigure}{0.33\textwidth}
    \captionsetup{justification=centering}
        \includegraphics[trim={3cm 8.5cm 4cm 8.5cm}, clip, width=\linewidth]{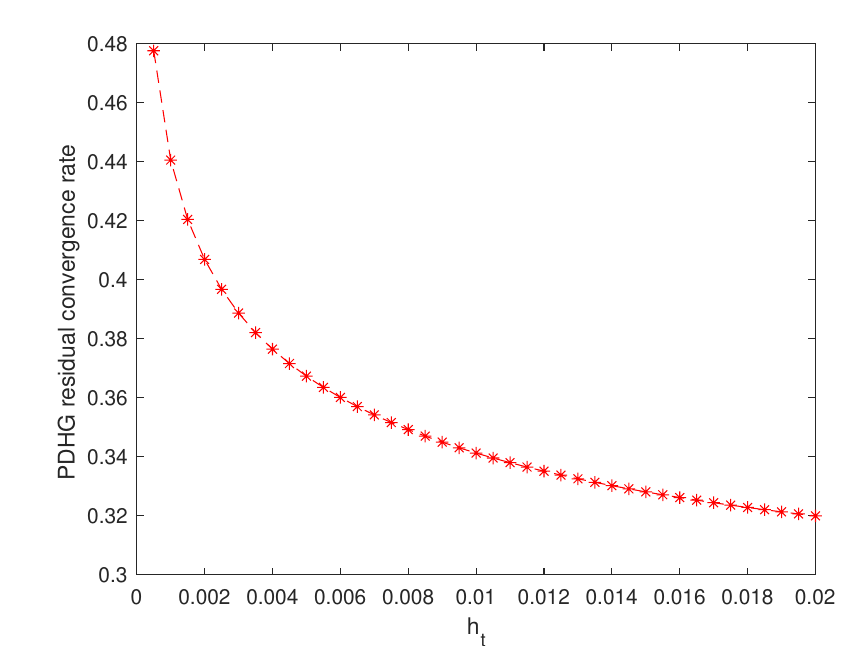}    
        \subcaption{Plot of $\bar{r}$ vs $h_t$. Fix $N_t = 1$, \\
        $h_t = 5\times 10^{-4}  k$, $1\leq k \leq 40$.
        }
    \end{subfigure}
    \begin{subfigure}{0.33\textwidth}
    \captionsetup{justification=centering}
        \includegraphics[trim={3cm 8.5cm 4cm 8.5cm}, clip, width=\linewidth]{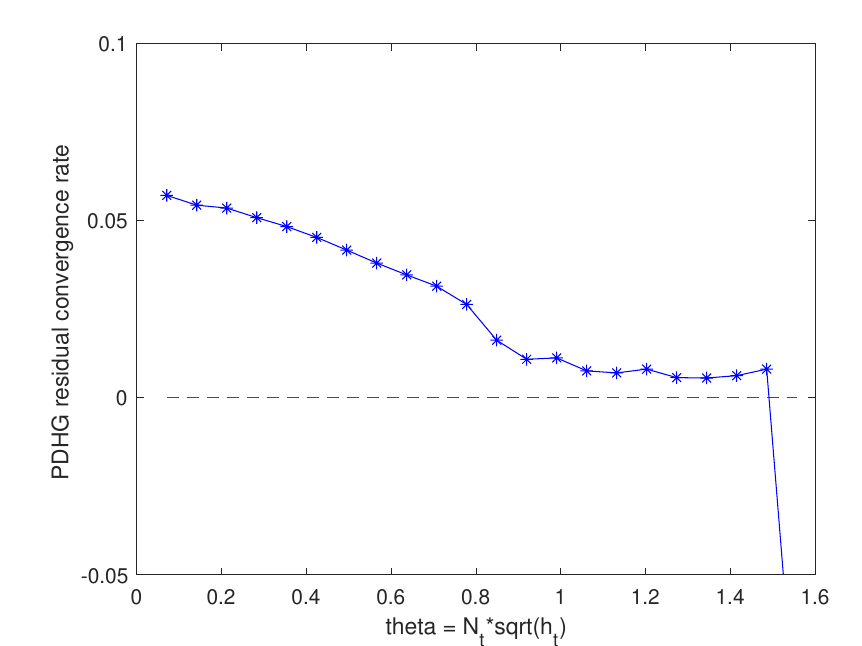}   
        \subcaption{Plot of $\bar{r}$ vs $N_t. $\\
        Fix $h_t = 0.005$, $1\leq N_t \leq 22$.
        }
    \end{subfigure}
    \caption{Convergence rate of the residual term $\|\widehat{F}(U_k)\|$ w.r.t. $h_t, N_t$ for Cahn-Hilliard equation.}
    \label{fig: res converge rate vs ht Nt CH}
\end{figure}

\begin{figure}[htb!]
    \centering
    \includegraphics[width=0.54\textwidth]{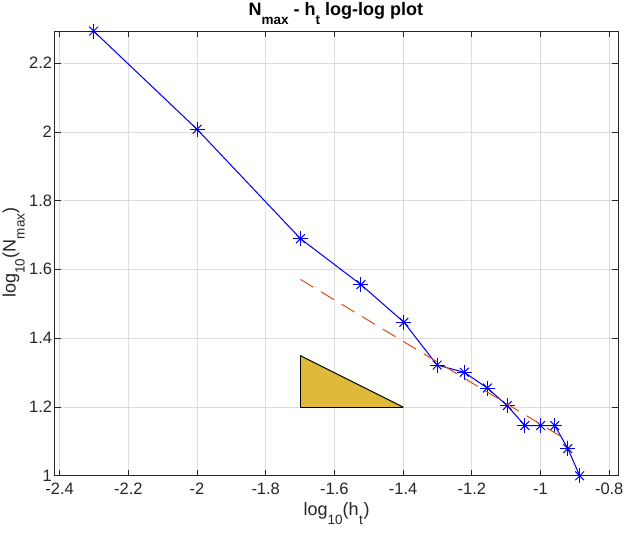}
    \caption{$N_{\mathrm{max}} - h_t$ log-log plot for Cahn-Hilliard equation  \eqref{CH equ}. We solve the equation on $64\time 64$ grid with $h_t = 0.01\cdot k$, $k=0.5, 1, 2, \dots, 13$. The yellow triangle has slope equals to $\frac{1}{2}$. The orange dashed line is the linear regression of data points with rather large $h_t = 0.01\cdot k$ with $5\leq k\leq 11$.}
    \label{fig: N max h_t log log plot }
\end{figure}

\begin{figure}[htb!]
    \centering
    \begin{subfigure}{0.4\textwidth}
    \captionsetup{justification=centering}
        \includegraphics[trim={1cm 5cm 1cm 5cm}, clip, width=\linewidth]{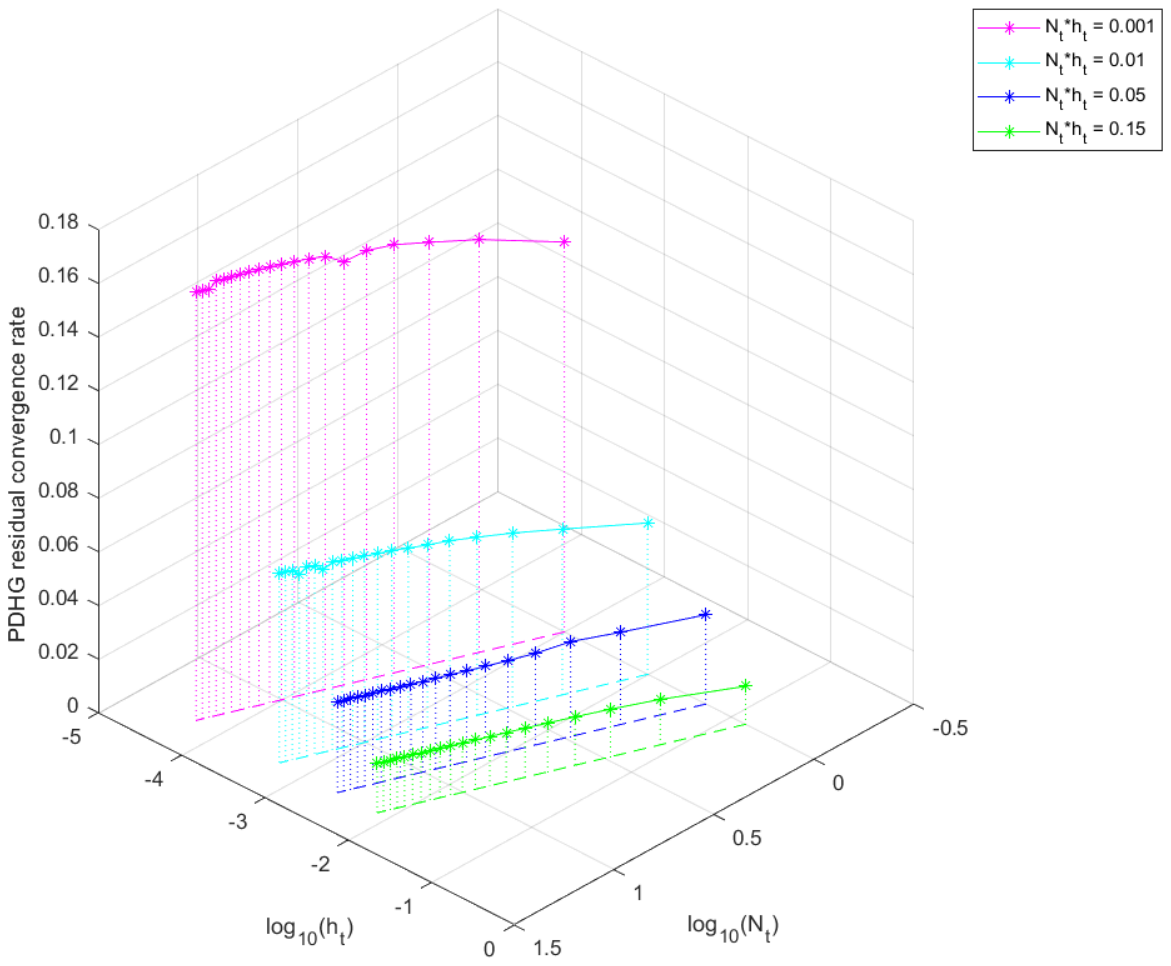}    
        \subcaption{We solve Allen-Cahn equation \eqref{AC_equ}\\
        Plot of $\bar{r}$ vs $( \log_{10} N_t, \log_{10} h_t)$, \\
        with $N_th_t = 0.15, 0.05, 0.01, 0.001.$
        }
    \end{subfigure}
    \begin{subfigure}{0.4\textwidth}
    \captionsetup{justification=centering}
        \includegraphics[trim={1cm 5cm 1cm 5cm}, clip, width=\linewidth]{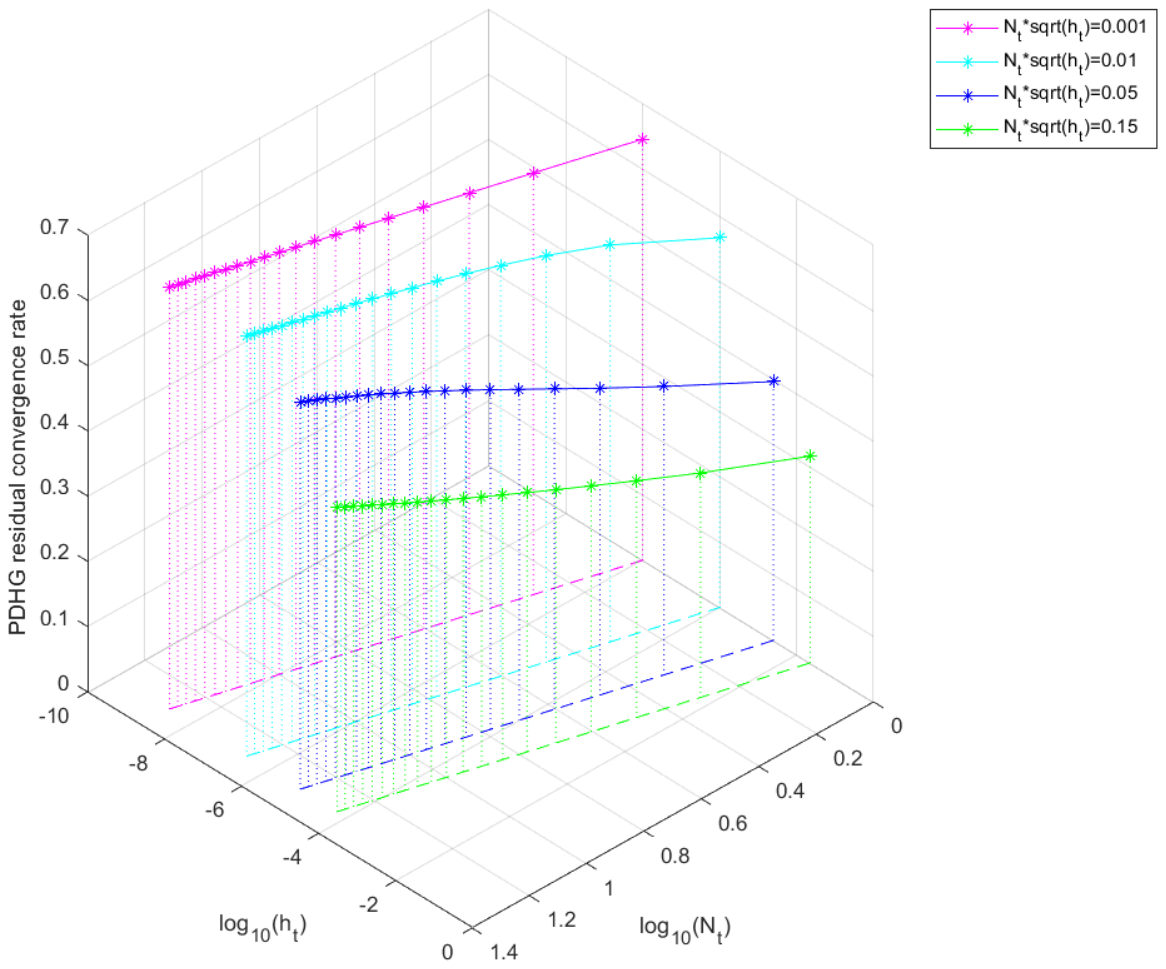}
        \subcaption{We solve Cahn-Hilliard equation \eqref{CH equ}\\
        Plot of $\bar{r}$ vs $( \log_{ 10 }N_t, \log_{10} h_t)$, \\
        with $N_t\sqrt{h_t} = 0.15, 0.05, 0.01, 0.001.$
        }
    \end{subfigure}
    \caption{Plots of $\bar{r}$ vs $(N_t, h_t)$.}\label{fig: plt rate vs log Nt vs log ht}
\end{figure}

\subsection{Lyapunov analysis for the time-discrete case}\label{sec: lyapunov time discrt case}
In this section, we discuss the convergence of the time-discrete PDHG algorithm \eqref{Preconded PDHG}. Recall that the equilibrium state of the PDHG dynamic \eqref{Preconded PDHG} is $(U_*, 0)$ with $\widehat{F}(U_*)=0$, we consider the following Lyapunov function
\begin{equation}
  \mathcal J(U, Q) = \frac{1}{2}(\|U-U_*\|^2 + \|Q - 0\|^2) = \frac{1}{2}(\|U-U_*\|^2 + \|Q \|^2).\nonumber
\end{equation}

The next theorem provides a sufficient condition on the convergence of $\mathcal J$ when $f(\cdot)$ is Lipschitz.

\begin{theorem}[Exponential convergence of the PDHG algorithm \eqref{Preconded PDHG}]\label{thm: time discrete PDHG converge with Lip }
  Consider the following assumptions,
  \begin{itemize}
    \item (On PDE \eqref{general RD type PDE}) Assume \eqref{condition: a b >= 0 }, \eqref{condition: f Lip } hold.
  
    \item (On numerical scheme \eqref{time-implicit scheme} of PDE)  Assume \eqref{condition: G,L>=0, symm, commut } holds. 
    Suppose the time step size $h_t$ and $T= N_t h_t$ satisfy $b T \mathrm{Lip}(R)\zeta_{a,b,c}(h_t)<\sqrt{2}-1$. Suppose we pick $\theta \geq bT\mathrm{Lip}(R)\zeta_{a,b,c}(h_t)$ with $\theta < \sqrt{2}-1$.
  
    \item (On PDHG algorithm \eqref{Preconded PDHG}) Suppose \eqref{condition Jf=cI c>=0 } holds. There exist $\widetilde{\gamma} = \omega \tau_P, \varrho = \frac{\tau_P}{\tau_U}, \epsilon > 0$ satisfying
    \begin{equation} 
      \varrho\widetilde\gamma\epsilon\Psi(\theta) - \frac{1}{4} \Omega(\widetilde\gamma\epsilon, \varrho, \theta)^2 > 0  \label{condition on gamma, ktau, epsilon }.
    \end{equation}
    Here we denote $\Psi(\theta)=1-2\theta-\theta^2$, and $\Omega(u, \varrho, \theta) = |1-u-\varrho| + (|1-u| + \varrho)\theta$. We choose PDHG step size for the dual variable as 
    \begin{equation}
      \tau_P = \frac{\varrho\widetilde\gamma\epsilon \Psi(\theta)  - \frac14 \Omega(\widetilde\gamma\epsilon, \varrho, \theta)^2}{4(\widetilde\gamma + \varrho\epsilon)(1+\theta)^2\max\{ \widetilde\gamma^2 (1+\theta)^2 ,  (1 -\widetilde\gamma \epsilon )^2 \}},\label{set tau }
    \end{equation}
    and set the extrapolation coefficient $\omega = \frac{\widetilde{\gamma}}{\tau_P}$, the PDHG step size for $U$ as $\tau_U = \frac{\tau_P}{\varrho}$.
  \end{itemize}
  Under the above conditions, there exists a unique $U_*$ s.t. $\widehat F(U_*)=0$. Furthermore, assume that $\{U_k, Q_k\}$ solves the PDHG algorithm \eqref{Preconded PDHG} with arbitrary initial condition $(U_0, Q_0)$. Write $\mathcal J_k = \mathcal J(U_k, Q_k)$. We have
  \begin{equation}
    \mathcal J_k \leq \left(\frac{2}{\Phi + \sqrt{\Phi^2 + 4}}\right)^{k+1} \left(\mathcal J_1 + \frac{\Phi + \sqrt{\Phi^2 + 4}}{2}\mathcal J_0\right), \label{exponential convergence of PDHG algorithm }
  \end{equation}
  where 
  $$\Phi=\frac{(\varrho\widetilde\gamma\epsilon \Psi(\theta)  - \frac14 \Omega(\widetilde\gamma\epsilon, \varrho, \theta)^2 )^2}{2(1+\theta)^2\max\{\widetilde\gamma^2(1+\theta)^2, (1-\widetilde\gamma \epsilon)^2\}(\widetilde{\gamma} + \varrho \epsilon)^2 }.$$
\end{theorem}
The proof of the theorem is provided in Appendix \ref{append: lyapunov analysis of PDHG alg }.

\vspace{0.5cm}

We can simplify the results in Theorem \ref{thm: time discrete PDHG converge with Lip } for Allen-Cahn and Cahn-Hilliard type of equations, using similar argument in the proof of Corollary \ref{practical corollary }, for Allen-Cahn (resp., Cahn-Hilliard) type equations. Suppose $b\mathrm{Lip}(R)T<\sqrt{2}-1$ (resp., $\frac{b\mathrm{Lip}(R)T}{2\sqrt{ah_t}+bch_t}<\sqrt{2}-1$). If we set $\theta = b\mathrm{Lip}(R)T$ (resp., $\theta = \frac{b\mathrm{Lip}(R)T}{2\sqrt{ah_t}+bch_t}$), then we have $bT\zeta_{a,b,c}(h_t)\mathrm{Lip}(R) \leq \theta < \sqrt{2}-1$.

Furthermore, we can pick specific values of the hyperparameters $\tau_U, \tau_P, \omega, \epsilon$ to obtain a more concise convergence rate $\Phi$. To do so, we denote $u=\widetilde\gamma\epsilon$ and assume that $u < 1$. We set $\varrho = 1-\widetilde{\gamma}\epsilon = 1-u$. Then the condition \eqref{condition on gamma, ktau, epsilon } leads to $(1-u)u\Psi(\theta)-(1-u)^2\theta^2 > 0$, which yields $ \frac{\theta^2}{1-2\theta} < u < 1 $. Furthermore, the rate $\Phi$ equals
\begin{equation*} 
  \Phi = \frac{ (1-u)^2(u(1-2\theta-\theta^2) -  (1-u)\theta^2 )^2 }{2(1+\theta)^2 \max\{ \widetilde\gamma^2(1+\theta)^2, (1-u)^2 \} (\widetilde\gamma + (1-u ) \epsilon)^2 }.
\end{equation*}
We further pick $\widetilde\gamma = (1-u)\epsilon$. Together with $\widetilde\gamma\epsilon = u$, we have $\widetilde \gamma = \sqrt{u(1-u)}$, $\epsilon = \sqrt{\frac{u}{1-u}}$. Thus,
\begin{align*} 
  \Phi = \frac{(1-2\theta)^2}{8(1+\theta)^2} \cdot \frac{\left(1 - \frac{\theta^2}{1-2\theta}\cdot \frac{1}{u} \right)^2}{ \max\{(1+\theta)^2, {(1-u)}/{u}\}}.
\end{align*}
Now the value of $\tau_P$ is determined by \eqref{set tau }, $\tau_U = \frac{\tau_P}{\varrho}$, $\omega = \frac{\widetilde{\gamma}}{\tau_P}$ can also be determined. In summary, we have the following Corollary.
\begin{corollary}[$N_x$-independent convergence rate for specific RD equations] \label{coro: simplify PDHG alg converge}
Suppose \eqref{condition: a b >= 0 }, \eqref{condition: f Lip }, \eqref{condition: G,L>=0, symm, commut }, and \eqref{condition Jf=cI c>=0 } hold. Assume $h_t$, $N_t$ and $T= N_t h_t$ satisfy 
\begin{itemize}
      \item (Allen-Cahn type, $\mathcal G_h = I$, $\mathcal L_h$ is self-adjoint, non-negative definite)
      Pick $T < \frac{\sqrt{2}-1}{b\mathrm{Lip}(R)}$, or equivalently, $h_t < \frac{\sqrt{2}-1}{b\mathrm{Lip}(R)}, ~ N_t \leq \Bigl\lfloor\frac{\sqrt{2}-1}{b\mathrm{Lip}(R)h_t}\Bigr\rfloor.$ We denote $\theta=b\mathrm{Lip}(R)T<\sqrt{2}-1$;
     \item (Cahn-Hilliard type, $\mathcal G_h = \mathcal L_h$ is self-adjoint, and non-negative definite) \\
     Pick $T <\frac{(\sqrt{2} - 1)(2\sqrt{ah_t} + bc h_t)}{b\mathrm{Lip}(R)}$, or equivalently, $h_t  <  \frac{4(\sqrt{2}-1)^2a}{b^2(\mathrm{Lip}(R)-(\sqrt{2}-1)c)_{+}^2}, ~ N_t \leq \Bigl\lfloor{ (\sqrt{2} - 1) \frac{2\sqrt{a/h_t} + bc }{b\mathrm{Lip}(R)}} \Bigr\rfloor$.
     We denote $\theta = \frac{b\mathrm{Lip}(R)T}{2\sqrt{ah_t}+bch_t} = \frac{b\mathrm{Lip}(R)N_t\sqrt{h_t}}{2\sqrt{a}+bc\sqrt{h_t}} < \sqrt{2}-1$.
\end{itemize}
Then, there is unique $U_*$ with $\widehat{F}(U_*)=0$. Furthermore, if we choose $u \in (\frac{\theta^2}{1-2\theta}, 1)$ and set 
  {\small 
\begin{align}
  \tau_P = \frac{u(1-2\theta)-\theta^2}{8\sqrt{u(1-u)}(1+\theta)^2\max\{u(1+\theta)^2, 1-u\}}, \quad \tau_U = \frac{\tau_P}{1-u}, \quad\omega = \frac{\sqrt{u(1-u)}}{\tau_U}, \quad \epsilon = \sqrt{\frac{u}{1-u}},
  \label{select hyperparameters }
\end{align}  }
then $U_k$ converges exponentially fast to $U_*$, i.e.,
\begin{equation*}
  \|U_k - U_*\|^2 \leq C_0 \left(\frac{2}{\Phi + \sqrt{\Phi^2 + 4}}\right)^{k+1} . 
\end{equation*}
Here
\begin{equation*}
  C_0 = \left(\mathcal J_1 + \frac{\Phi + \sqrt{\Phi^2 + 4}}{2}\mathcal J_0\right), \quad 
  \Phi = \frac{(1-2\theta)^2}{8(1+\theta)^2} \cdot \frac{\left(1 - \frac{\theta^2}{1-2\theta}\cdot \frac{1}{u} \right)^2}{ \max\{(1+\theta)^2, {(1-u)}/{u}\}}.
\end{equation*}
\end{corollary}

In the following example, we pick the hyperparameters $h_t, N_t, \tau_U, \tau_P, \omega, \epsilon$ according to Corollary \ref{coro: simplify PDHG alg converge}, and apply it to different types of equations. Our algorithm is guaranteed to converge linearly. The theoretical results presented in Theorem \ref{thm: time discrete PDHG converge with Lip } and Corollary \ref{coro: simplify PDHG alg converge} are not necessarily the sharpest convergence rate. In practice, the actual convergence rate of our PDHG method is generally faster than the theoretical guarantee in Corollary \ref{coro: simplify PDHG alg converge}. This is reflected in the following Table \ref{tab: theoretical rate vs actual rate}. When composing Table \ref{tab: theoretical rate vs actual rate}, recall that $f(u)=u^3-u$, we set $c=f'(\pm 1)=2$, and $R(u) = f(u)-cu = u^3-3u.$ In our numerical result, we observe that $|U_{ij}^t|\leq 1$ for any spatial index $(i,j)$ and temporal index $t$. Thus we use $\underset{u\in[-1, 1]}{\sup } ~|R'(u)| = 3$ as the value of $\mathrm{Lip}(R)$ in Corollary \ref{coro: simplify PDHG alg converge} during the calculation.

\begin{table}[htb!]
\scriptsize
    \centering
    \begin{tabular}{c|c|cccccccccc}
       \multicolumn{2}{c|}{    }  & $h_t$ & $N_t$ & $u$ & $\tau_P$ & $\tau_U$ & $\omega$  & $\epsilon$ & $\widetilde\theta$ &
       \begin{tabular}{@{}c@{}} Theoretical \\ rate \end{tabular} & \begin{tabular}{@{}c@{}} Actual \\ rate \end{tabular}
         \\
         \hline \hline
      & $\epsilon_0=1.0$  & \begin{tabular}{@{}c@{}} 0.005 \\ $ < 0.1381$ \end{tabular} & \begin{tabular}{@{}c@{}} 20 \\ $ \leq 27 $ \end{tabular} & \begin{tabular}{@{}c@{}} 0.5 \\ $u \in (0.2250,1)$ \end{tabular} & 0.0498 & 0.0996 & 5.0181 & 1.0 & 0.3000 & 0.0112 & 0.0723 \\
      AC\eqref{AC_equ} & $\epsilon_0 = 0.1$  & \begin{tabular}{@{}c@{}} 0.001 \\ $ < 0.0138$ \end{tabular} & \begin{tabular}{@{}c@{}} 7 \\ $ \leq 13 $ \end{tabular} & 
       \begin{tabular}{@{}c@{}} 0.5 \\ $u \in (0.0760,1)$ \end{tabular} & 0.0574 & 0.1147 & 4.3587 & 1.0 & 0.2100 & 0.0141 & 0.0821 \\
      & $\epsilon_0 = 0.01$  & \begin{tabular}{@{}c@{}} 0.0005 \\ ($< 0.0014$) \end{tabular} & \begin{tabular}{@{}c@{}}1\\($\leq2$)\end{tabular} & \begin{tabular}{@{}c@{}} 0.5 \\ ($u\in (0.0321, 1)$) \end{tabular} & 0.0936 & 0.1872 & 2.6702 & 1.0 & 0.1500 & 0.0307 & 0.1325 \\
         \hline
        &  $\epsilon_0 = 10$         &\begin{tabular}{@{}c@{}}0.005 \\ (<1.4553) \end{tabular}& \begin{tabular}{@{}c@{}} 10 \\ ($\leq 12 $) \end{tabular}&\begin{tabular}{@{}c@{}} 0.5 \\ ($u\in (0.04, 1)$) \end{tabular}  & 0.0842 & 0.1684 & 2.9695 & 1.0 &  0.1640  & 0.0260 & 0.0537 \\
      CH\eqref{CH equ}  &  $\epsilon_0 = 1.0$         &\begin{tabular}{@{}c@{}}0.001 \\ (<0.1455) \end{tabular}& \begin{tabular}{@{}c@{}} 5 \\ ($\leq 9$) \end{tabular}&\begin{tabular}{@{}c@{}} 0.5 \\ ($u\in (0.0978, 1)$) \end{tabular}  & 0.0475 &0.0949 & 5.2662 & 1.0 & 0.2874 & 0.0103 & 0.0301 \\
        & $\epsilon_0 = 0.1$         &\begin{tabular}{@{}c@{}}0.0005 \\ (<0.0015) \end{tabular}& \begin{tabular}{@{}c@{}} 1 \\ ($\leq 1$) \end{tabular}&\begin{tabular}{@{}c@{}} 0.5 \\ ($u\in (0.1663, 1)$) \end{tabular}  & 0.0286 & 0.0572 & 8.7392 & 1.0 & 0.2741 & 0.0043 & 0.0169 \\
      \hline
    \end{tabular}
    \caption{Theoretical convergence rate vs actual convergence rate of $\|U_k - U_*\|^2_2$. The constraints in the parentheses in the columns of $h_t, N_t$, and $u$ are derived from the conditions in Corollary \ref{coro: simplify PDHG alg converge}. The actual rate $r$ is solved from the linear regression model $r\cdot k + b$ given the numerical data $\{k, \log(\|U_{k+1}-U_*\|^2/\|U_k - U_*\|^2)\}$ for $1\leq k \leq 400$ (Allen-Cahn equation \eqref{AC_equ}); and $1\leq k \leq 500$ (Cahn-Hilliard equation \eqref{CH equ}). }
    \label{tab: theoretical rate vs actual rate}
\end{table}

\begin{figure}[htb!]
\begin{subfigure}{.31\textwidth}
    \centering
    \includegraphics[trim={4cm 8.5cm 4cm 8.5cm}, clip, width=\linewidth]{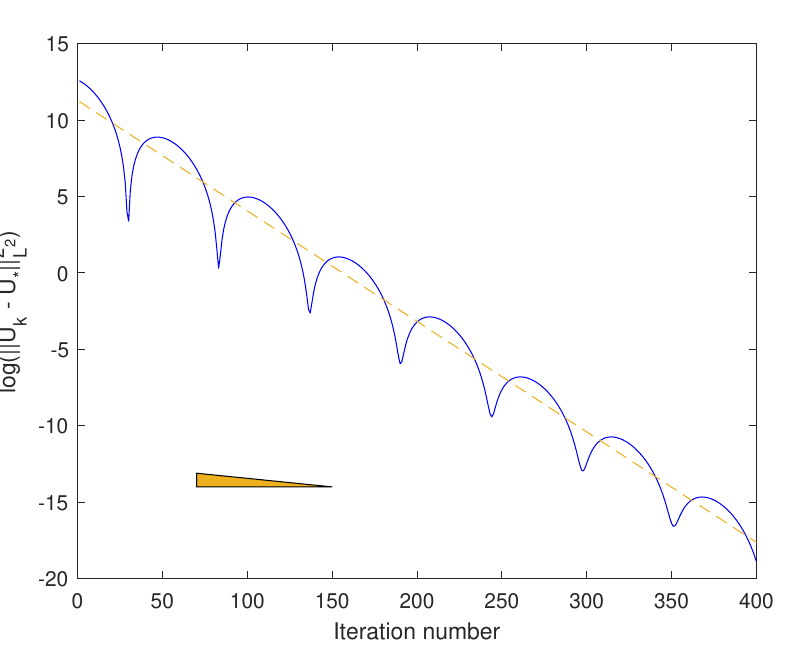}
    \subcaption{$\epsilon_0 = 1.0$.}
\end{subfigure}
\begin{subfigure}{.31\textwidth}
    \centering
    \includegraphics[trim={4cm 8.5cm 4cm 8.5cm}, clip, width=\linewidth]{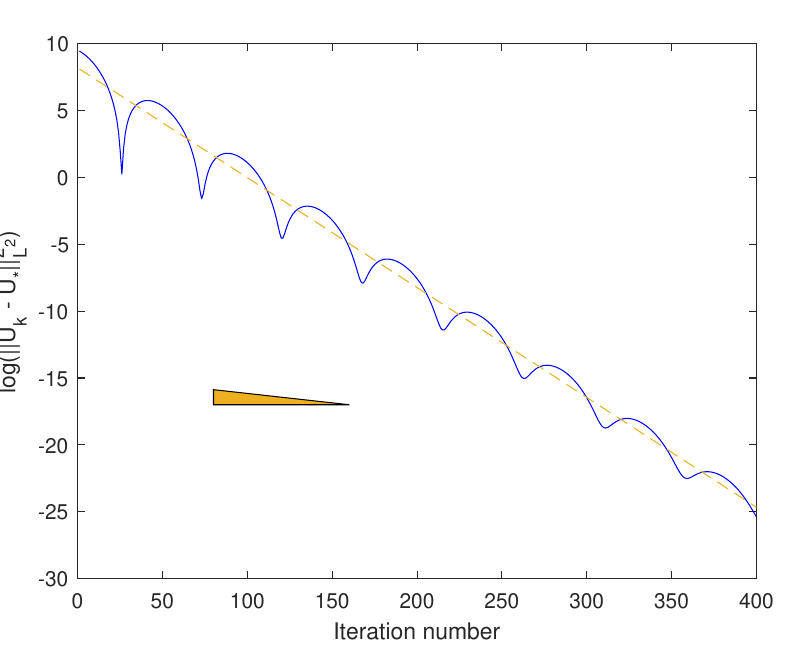}
    \subcaption{$\epsilon_0 = 0.1.$}
\end{subfigure}
\begin{subfigure}{.31\textwidth}
    \centering
    \includegraphics[trim={4cm 8.5cm 4cm 8.5cm}, clip, width=\linewidth]{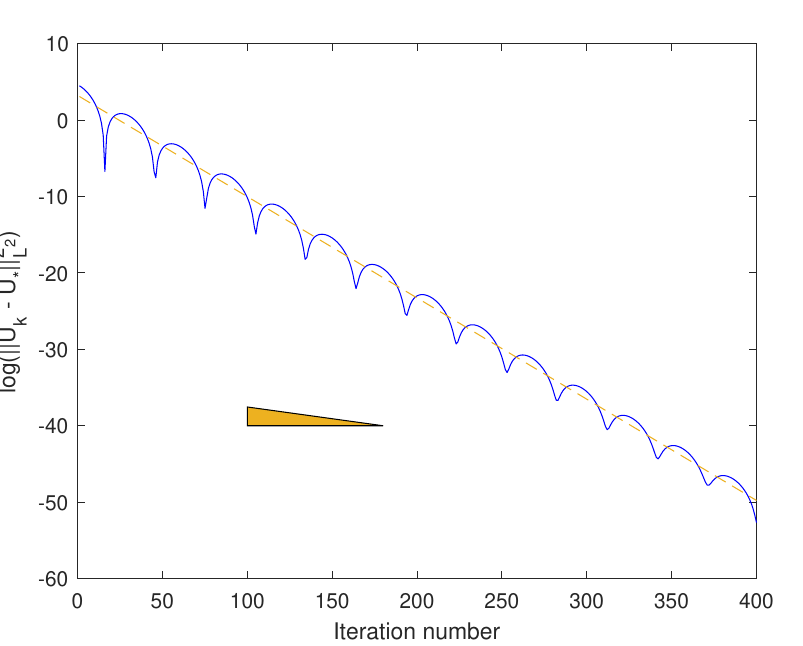}
    \subcaption{$\epsilon_0=0.01.$}
\end{subfigure}
\caption{Plot of $\log\|U_k-U_*\|^2$ vs $k$ ($1\leq k \leq 400$)  when using hyperparameters specified in Table \ref{tab: theoretical rate vs actual rate} to solve Allen-Cahn equation \eqref{AC_equ} with different $\epsilon_0$ on a $128\times 128$ grid. }
\label{fig: verify actual rate and theoretical rate A }
\end{figure}

\begin{figure}[htb!]
\begin{subfigure}{.31\textwidth}
    \centering
    \includegraphics[trim={4cm 8.5cm 4cm 8.5cm}, clip, width=\linewidth]{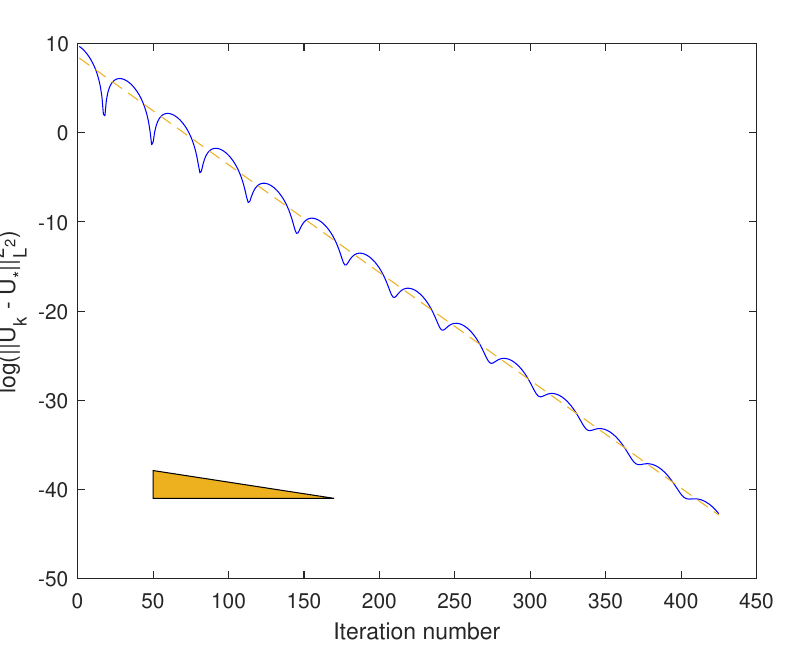}
    \subcaption{$\epsilon=10.$}
\end{subfigure}
\begin{subfigure}{.31\textwidth}
    \centering
    \includegraphics[trim={4cm 8.5cm 4cm 8.5cm}, clip, width=\linewidth]{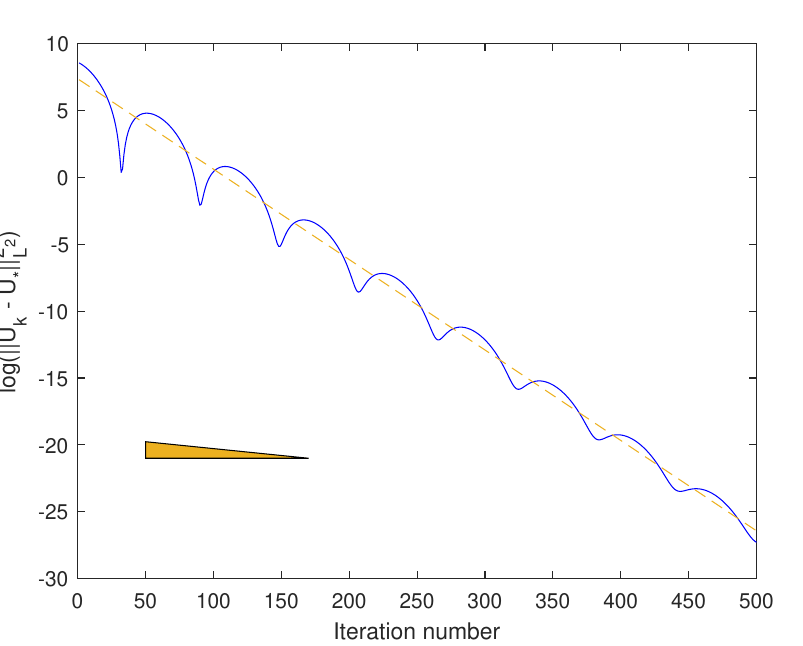}
    \subcaption{$\epsilon=1.0.$}
\end{subfigure}
\begin{subfigure}{.31\textwidth}
    \centering
    \includegraphics[trim={4cm 8.5cm 4cm 8.5cm}, clip, width=\linewidth]{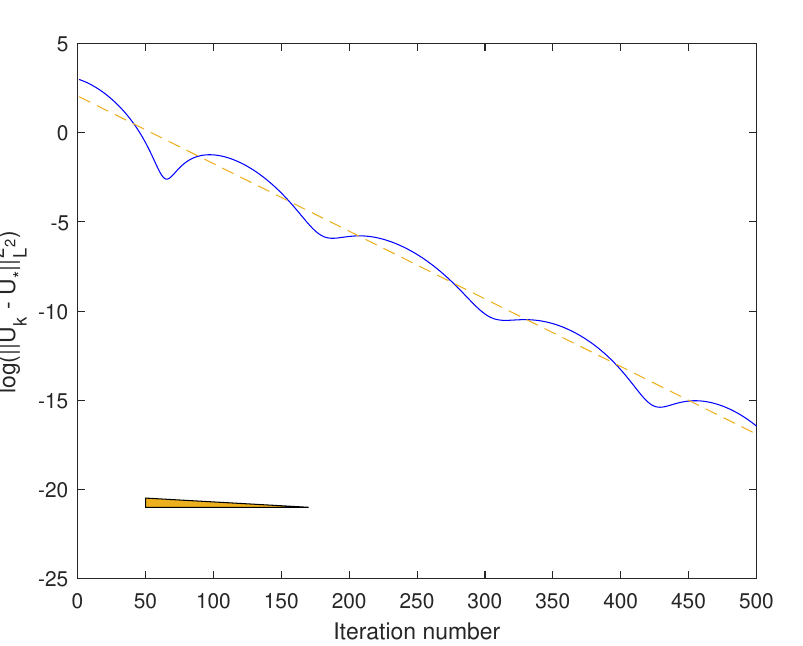}
    \subcaption{$\epsilon=0.1.$}
\end{subfigure}
\caption{Plot of $\log\|U_k-U_*\|^2$ vs $k$ ($1\leq k \leq 500$)  when using hyperparameters specified in Table \ref{tab: theoretical rate vs actual rate} to solve Cahn-Hilliard equation \eqref{CH equ} with different $\epsilon_0$ on a $128\times 128$ grid. }
\label{fig: verify actual rate and theoretical rate B }
\end{figure}

\begin{remark}
  \eqref{select hyperparameters } may also not be the optimal choice of hyperparameters. We provide suggestions on selecting the optimal hyperparameters in section \ref{sec: hyperparam select}.
\end{remark}

\section{Numerical examples}\label{sec: numerical example }
In this section, we test the proposed algorithm on four types of RD equations, namely the Allen-Cahn equation, the Cahn-Hilliard equation, an RD equation with variable coefficients (mobility term), and a 6th-order reaction-diffusion equation. We verify the independence between the convergence rate of our algorithm and the grid size $N_x$. We discuss how the hyperparameters of the proposed algorithm are chosen to achieve the optimal (or near-optimal) performance via numerical experiments. We also provide comparisons between the implicit scheme with adaptive step size $h_t$ and the IMEX scheme on long-time range computation. At the end of this section, we make comparisons with three commonly used algorithms for resolving the time-implicit schemes such as the nonlinear SOR \cite{merriman1994motion}, the preconditioned fixed point method \cite{bai2007preconditioned} and Newton's method \cite{christlieb2014high}. 

For all the numerical examples in this section, if not specified, we always set the hyperparameters $\omega=1$ and $\epsilon = 0.1$. We terminate the iteration whenever $\|\mathrm{Res}(U_k)\|_\infty < tol$ with $tol=10^{-6}$. Here the residual term $\mathrm{Res}(U_k)$ is defined in \eqref{def: residual }. All numerical examples are imposed with periodic boundary conditions. We adopt the central discretization scheme to discretize the Laplace operator $\Delta$, i.e., we set the discretized Laplace operator as $\mathrm{Lap}^P_{h_x}$ defined in \eqref{def: periodic discrete Laplace operator }.

Among four equations discussed in this section, equations \eqref{AC_equ}, \eqref{CH equ}, and \eqref{6thorder} have already been considered in \cite{liu2024first}, where more numerical results are demonstrated. In this research, we mainly use them as test equations for validating our theoretical findings and justifying the effectiveness of our method. 

All the numerical examples are computed using MATLAB on a laptop with 11th Gen Intel Core i5-1135G7 @ 2.40GHz CPU and 16.0 GB RAM. The corresponding codes are provided at \url{https://github.com/LSLSliushu/PDHG-method-for-solving-reaction-diffusion-equations/tree/main}.

\subsection{Tested equations}\label{sec: tested eq }
Throughout this section, we denote the double potential function $W(u) = \frac{1}{4}(u^2 - 1)^2$, and thus $W'(u)=u^3 - u.$
\subsubsection{Allen-Cahn equation (AC)}
We consider the Allen-Cahn equation
\begin{equation}
  \frac{\partial u}{\partial t} = a \Delta u - bW'(u), \quad \textrm{on } [0, 0.5]^2\times[0, T], \quad u(x,0)=u_0(x).   \label{AC_equ}
\end{equation}
We set $a=\epsilon_0, b=\frac{1}{\epsilon_0}$ with $\epsilon_0 = 0.01$.  We set the initial condition as $u_0 = 2\chi_{B(x_*, r)}-1$ where $x_*=(0.25, 0.25), r=0.2$. For the precondition matrix $\mathscr{M}$, $\mathcal G_h=I,$ and $\mathcal L_h = \Delta_{h_x}^P$, and $J_f = 2I$. We compare our method and the IMEX method in Figure \ref{fig: our implicit vs IMEX AC }. The zero-level set of the solution $u(\cdot, t)$ of this equation is known to be the curvature flow of a circle \cite{merriman1994motion}. A comparison among the plots of the front positions computed by our method, the Nonlinear SOR method. The real solution is presented on the right-hand side of Figure \ref{fig: our implicit vs IMEX AC }.
\begin{figure}[htb!]
    \centering
    \begin{subfigure}{0.6\textwidth}
        \includegraphics[trim={2cm 8cm 2cm 8cm}, clip, width=0.75\linewidth]{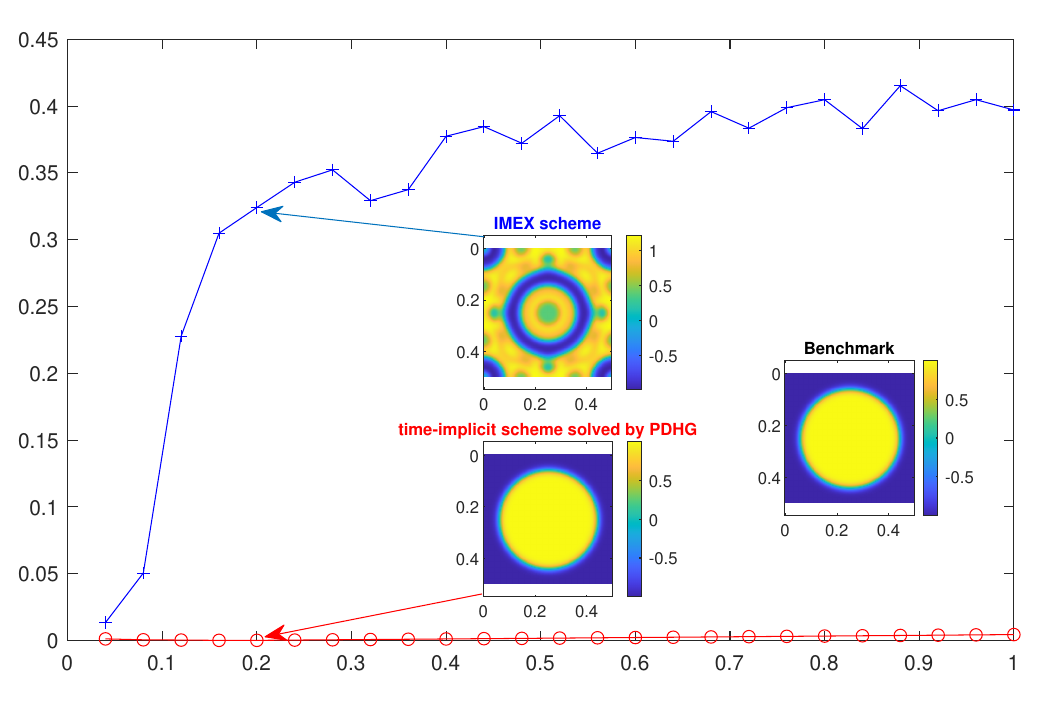}
    \end{subfigure}
    \begin{subfigure}{0.35\textwidth}
        \includegraphics[trim={4cm 8.5cm 4cm 9cm}, clip, width=\linewidth]{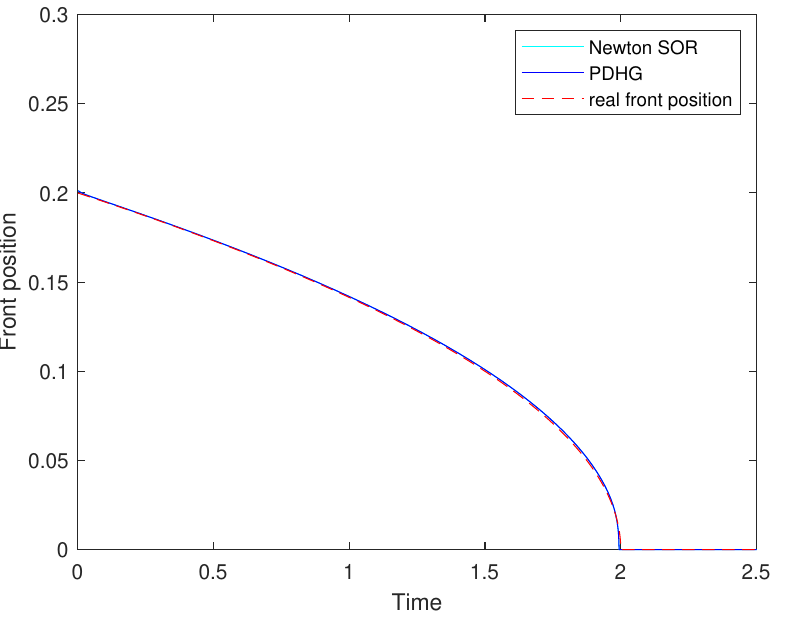}
    \end{subfigure}
    \caption{We solve equation \eqref{AC_equ} with $\epsilon_0 = 0.01$. We set $\tau_U=0.55, \tau_P=0.95$ for our PDHG method. (Left) Comparison between our method (time-implicit scheme solved by the proposed PDHG algorithm) and the IMEX scheme.  We discrete the space into $128\times 128$ lattices. We compute both schemes with large time step size $h_t = 0.02$ and compare with the benchmark solution solved from the same IMEX scheme with $h_t=0.001.$ Blue curve indicates the $L^1$ discrepancy between the IMEX solution on the coarser time grid $U_{\textrm{IMEX}}$ and the benchmark solution $U_{\star}$. Red curve indicates the $L^1$ discrepancy between the time-implicit solution $U_{\textrm{PDHG}}$ and $U_\star$. (Right) Comparison between the front position of the numerical solution solved via our PDHG method and the Nonlinear SOR method, as well as the real front position.
    }\label{fig: our implicit vs IMEX AC }
\end{figure}

\subsubsection{Cahn-Hilliard equation (CH)}
We consider the Cahn-Hilliard equation
\begin{equation}
  \frac{\partial u}{\partial t} = - a \Delta\Delta u + \Delta bW'(u), \quad \textrm{on } [0, 2\pi]^2\times[0, T], \quad u(x,0)=u_0(x).     \label{CH equ}
\end{equation}
We set $a=\epsilon_0^2$ and $b=1$. We set the initial condition $u_0$ as a modified indicator function whose value equals $+1$ if $(x,y)$ falls inside any of the seven circles and $-1$ otherwise, i.e.,
\[u_0(x, y) = -1 + \sum_{i=1}^7 \varphi(\sqrt{(x-x_i)^2 + (y-y_i)^2} - r_i),\]
where the mollifier function $\varphi$ is defined as
\[ \varphi(s) = \begin{cases}
    2e^{-\frac{\epsilon^2}{s^2}} \quad &s<0;\\
    0 \quad &s \geq 0
\end{cases}, \quad \textrm{with } \epsilon = 0.1. \]
The centers and radii of these seven circles are listed in Table~\ref{tab: seven circles}.
\begin{table}[htb!]
    \centering
    \begin{tabular}{|c||c|c|c|c|c|c|c|}
    \hline
          $i$ &  $1$     &  $2$        &  $3$        & $4$        & $5$      & $6$   & $7$  \\
          \hline
          \hline
       $x_i$  &  $\pi/2$ &  $\pi/4$    &  $\pi/2$   & $\pi$       & $3\pi/2$ & $\pi$ & $3\pi/2$\\
       $y_i$  &  $\pi/2$ &  $3\pi/4$   &  $5\pi/4$  & $\pi/4$     & $\pi/4$  & $\pi$ & $3\pi/2$\\
       $r_i$  &  $\pi/5$ &  $2\pi/15$  &  $\pi/15$  & $\pi/10$    & $\pi/10$ & $\pi/4$ & $\pi/4$\\
    \hline
    \end{tabular}
    \caption{Centers and radius of the 7 circles.}
    \label{tab: seven circles}
\end{table}
For the precondition matrix $\mathscr{M}$, $\mathcal G_h = \mathcal L_h = \Delta_{h_x}^P$, and $J_f = 2I$.

\subsubsection{A reaction-diffusion equation with variable coefficient (VarCoeff)}

We consider the following equation with variable coefficient (mobility term) $\sigma(\cdot)$,
\begin{equation}
  \frac{\partial u}{\partial t} = a \nabla\cdot(\sigma(x)\nabla u ) - bW'(u), \quad \textrm{on } [0, 2\pi]^2\times[0, T], \quad u(x,0)=u_0(x).   \label{VarCoeff}
\end{equation}
We choose $a=\epsilon_0, b=\frac{1}{\epsilon_0}$ with $\epsilon_0=0.01$. The media $\sigma(x, y) = 1+\frac{\mu}{2}(\sin^2 x + \sin^2 y)$ with $\mu=5.0.$ We set the initial condition $u_0 = \frac{1}{2}(\cos(4x)+\cos(4y))$. We adopt the following time-implicit scheme
{\footnotesize
\begin{equation} 
\frac{U_{ij}^{t+1}-U_{ij}^t}{h_t} = \frac{a}{h_x^2} (\sigma_{i+\frac{1}{2},j}(U_{i+1, j} - U_{i,j}) - \sigma_{i-\frac{1}{2},j}(U_{i,j} - U_{i-1, j}) + \sigma_{i,j+\frac12}(U_{i, j+1} - U_{i,j}) - \sigma_{i,j-\frac12}(U_{i,j} - U_{i, j-1})) - bW'(U^{t+1}_{ij}), \label{implicit schm for RD with mobility }
\end{equation}
}
where $0\leq t\leq N_t-1$, $1\leq i,j\leq N_x$, and $U_{N_x+1, j}=U_{1,j}, U_{0, j} = U_{N_x, j}; U_{i, N_x+1}=U_{i, 1}, U_{i, 0}=U_{i, N_x}$ for all $1\leq i,j\leq N_x.$ And we set $\sigma_{pq}=\sigma((p-1)h_x, (q-1)h_x)$ for any $p,q\in\mathbb{Q}$.

For the precondition matrix $\mathscr{M}$, $\mathcal G_h = I,$ we approximate $\mathcal L_h$ by $-\overline{\sigma}\Delta_{h_x}^P$, whose matrix-vector multiplication and inversion can be efficiently computed via the FFT algorithm. Here $\overline{\sigma}=\frac{1}{|\Omega|}\int_\Omega \sigma(x,y)~dxdy = 1+\frac{\mu}{2}$ denotes the average of $\sigma$ over $\Omega=[0, 2\pi]^2$. We set $J_f = 2I.$ We choose $\tau_U=0.5, \tau_P=0.95$ when applying our PDHG method to solve the time-implicit scheme \eqref{implicit schm for RD with mobility }.

The numerical solutions to \eqref{VarCoeff} are provided in Figure \ref{VarCoeff solution PDHG}. A series of residual decay plots throughout our method are demonstrated in Figure \ref{fig: res decay plots VarCoeff }. 

Furthermore, we denote 
\[ E(u) = \int_{\Omega} \frac{a}{2}\sigma(x)|\nabla u(x)|^2 + bW(u(x))~dx, \]
as the free energy functional associated with the reaction-diffusion equation \eqref{VarCoeff}. Denote
\begin{equation} 
    E_{h_x}(U) = \sum_{1\leq i,j\leq N_x} \left(\frac{a}{2}(\sigma_{i+\frac12, j}|U_{i+1, j} - U_{i, j}|^2 + \sigma_{i, j+\frac12}|U_{i, j+1} - U_{i, j}|^2) + b W(U_{i, j})\right)h_x^2    \label{def : discrete free energy }
\end{equation}
as the discrete analogy of $E(u)$. The free energy $E_{h_x}(U^{t_k})$ versus $t_k$ plot of energy decay is presented in Figure \ref{fig: free energy and relative error in free energy}. In addition, a comparison between the proposed scheme and the IMEX scheme can be found in Figure \ref{fig: our implicit vs IMEX varcoeff }.

\begin{figure}[htb!]
\begin{subfigure}{.16\textwidth}
  \centering
  \includegraphics[trim={4cm 9cm 4cm 9cm},clip, width=\linewidth]{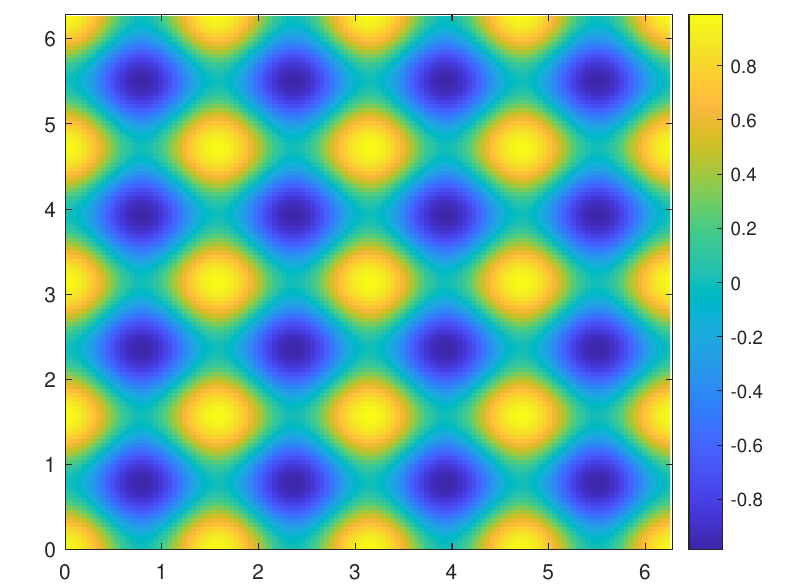}
  \caption*{$t=0.0$}
\end{subfigure}
\begin{subfigure}{.16\textwidth}
  \centering
  \includegraphics[trim={4cm 9cm 4cm 9cm},clip, width=\linewidth]{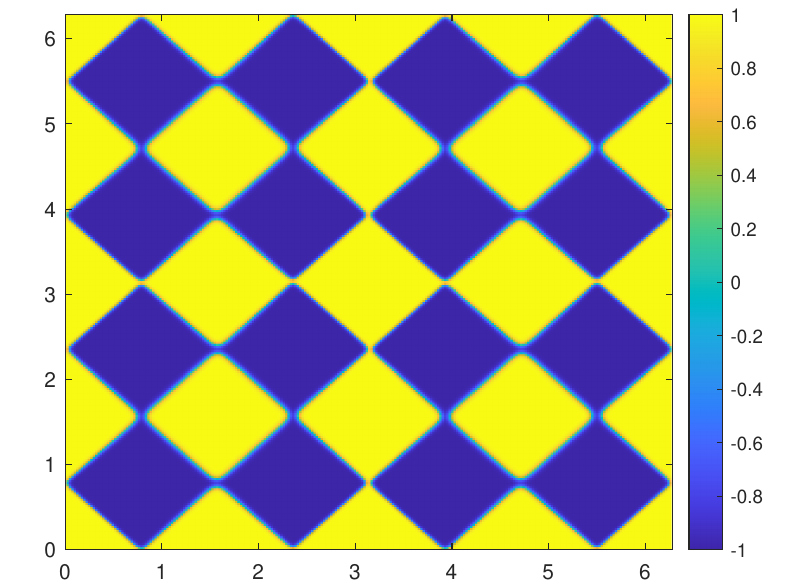}
  \caption*{$t=0.2$}
\end{subfigure}
\begin{subfigure}{.16\textwidth}
  \centering
  \includegraphics[trim={4cm 9cm 4cm 9cm},clip, width=\linewidth]{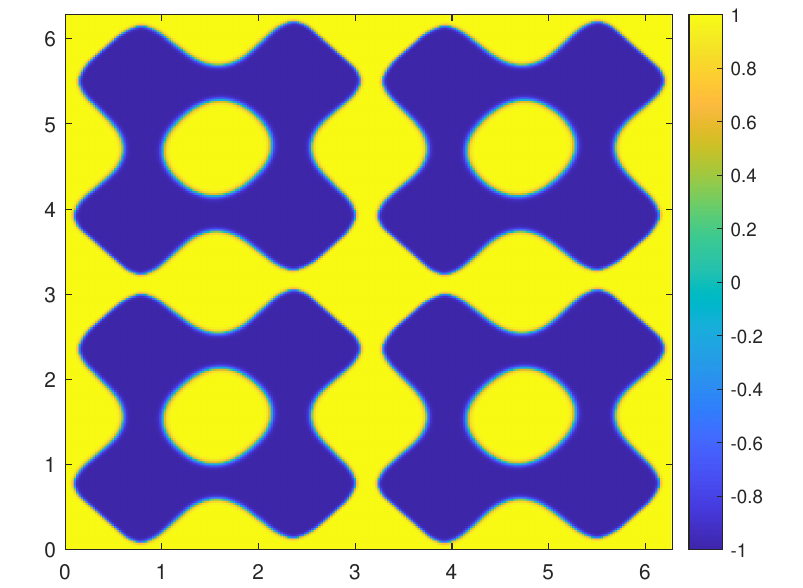}
  \caption*{$t=1.0$}
\end{subfigure}
\begin{subfigure}{.16\textwidth}
  \centering
  \includegraphics[trim={4cm 9cm 4cm 9cm},clip, width=\linewidth]{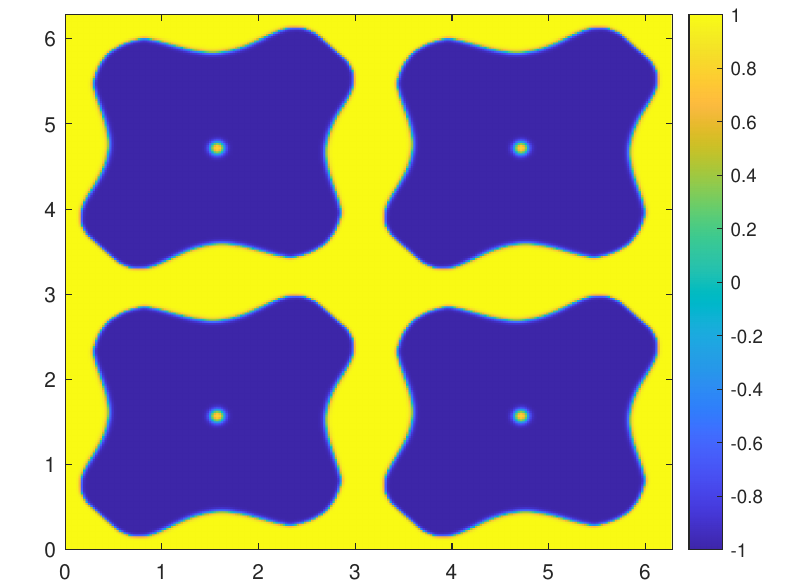}
  \caption*{$t=3.6$}
\end{subfigure}
\begin{subfigure}{.16\textwidth}
  \centering
  \includegraphics[trim={4cm 9cm 4cm 9cm},clip, width=\linewidth]{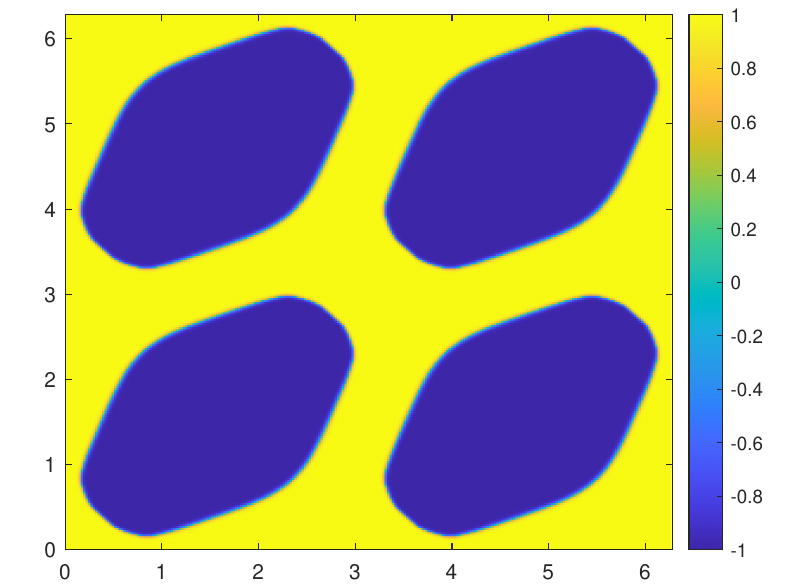}
  \caption*{$t=10.0$}
\end{subfigure}
\begin{subfigure}{.16\textwidth}
  \centering
  \includegraphics[trim={4cm 9cm 4cm 9cm},clip, width=\linewidth]{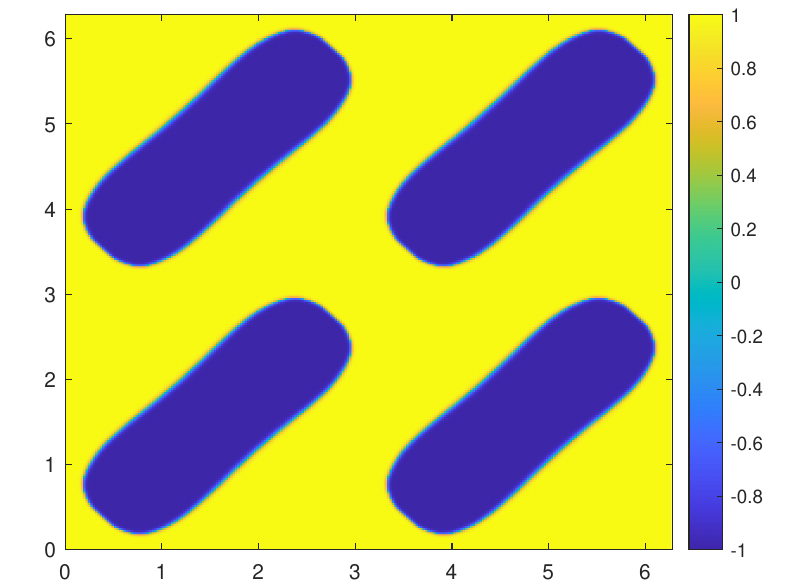}
  \caption*{$t=20.0$}
\end{subfigure}
\caption{ Numerical solution of the time-implicit scheme solved via our PDHG method on a $256\times 256$ grid at different time stages $t=0.0,0.2, 1.0, 3.6, 10.0, 20.0$.}\label{VarCoeff solution PDHG}
\end{figure}
\begin{figure}
    \centering
    \begin{subfigure}{0.31\textwidth}
    \includegraphics[trim={4cm 7.5cm 4cm 9cm},clip, width=\linewidth]{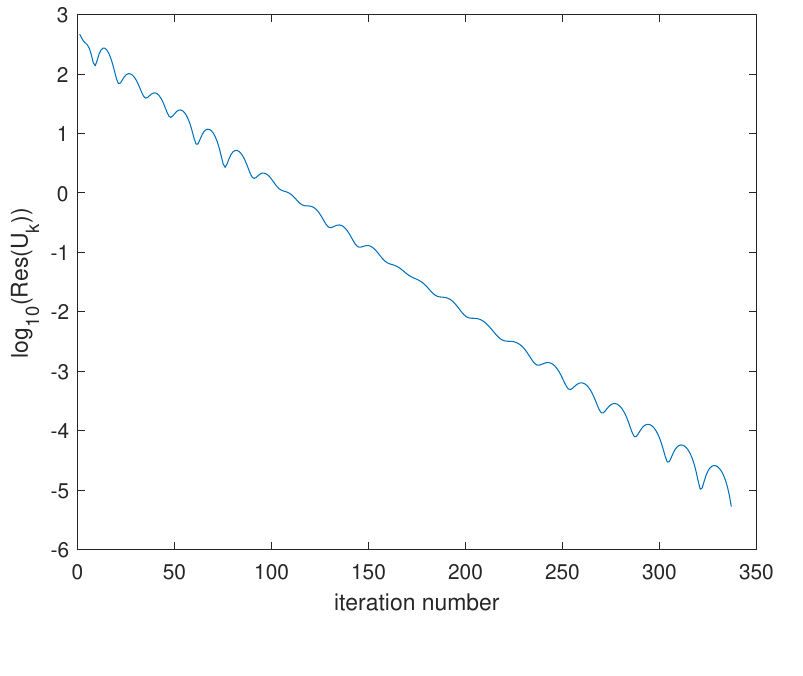}
    \end{subfigure}
    \begin{subfigure}{0.31\textwidth}
    \includegraphics[trim={4cm 7.5cm 4cm 9cm},clip, width=\linewidth]{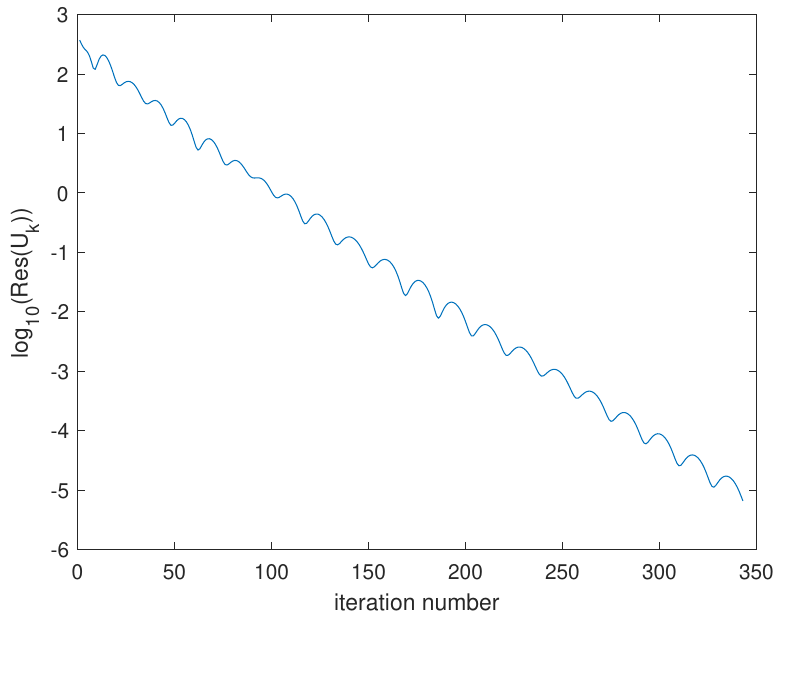}
    \end{subfigure}
    \begin{subfigure}{0.31\textwidth}
    \includegraphics[trim={4cm 7.5cm 4cm 9cm},clip, width=\linewidth]{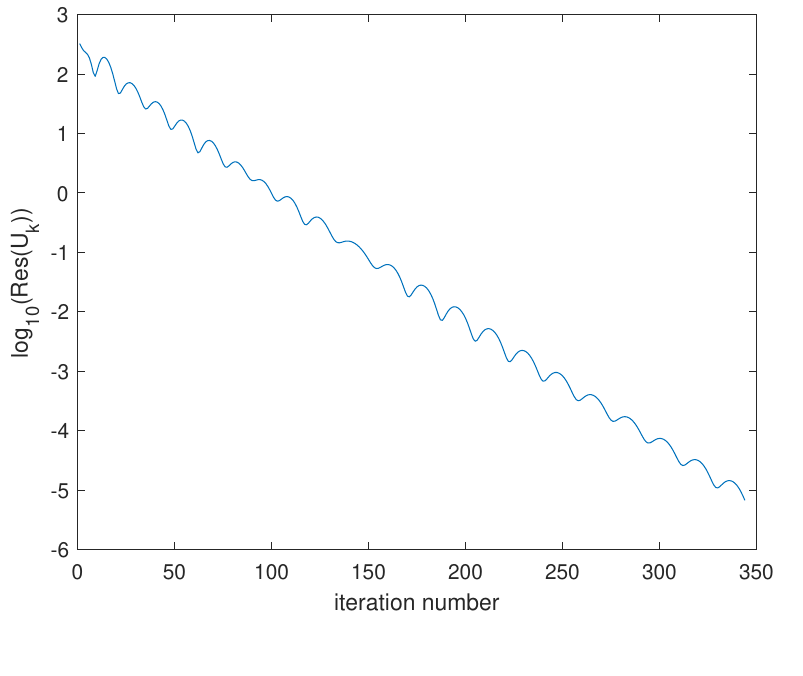}
    \end{subfigure}
    \caption{The loss plot of $\log_{10}(\mathrm{Res}(U_k))$ vs iteration number $k$. We solve \eqref{VarCoeff} with $h_t = 0.002$. The plots (from left to right) are the loss plots at $30$th, $60$th, and $90$th subinterval.}\label{fig: res decay plots VarCoeff }
\end{figure}

\begin{figure}[htb!]
\centering
\begin{subfigure}{.4\textwidth}
    \centering
    \includegraphics[trim={4cm 8.5cm 4cm 8.5cm}, clip, width=\linewidth]{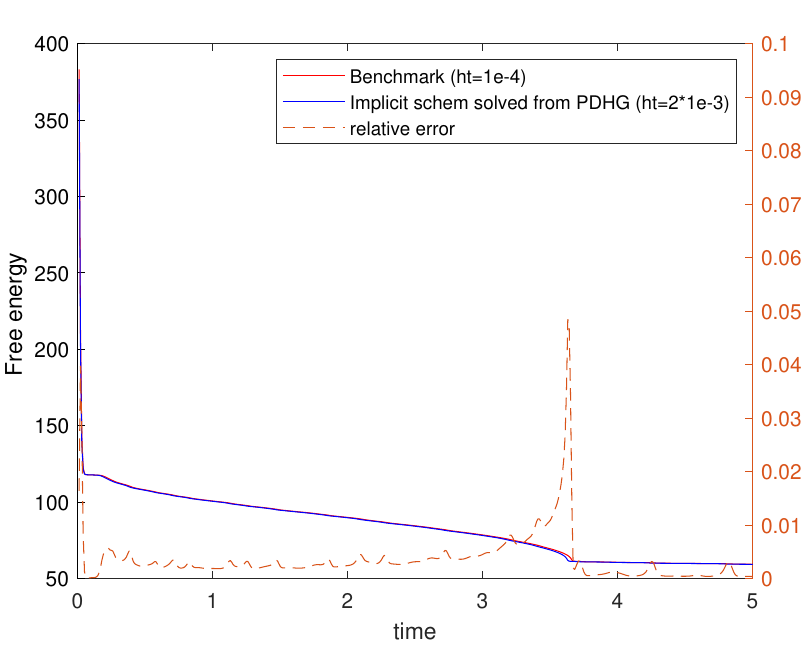}
\end{subfigure}
\hspace{0.5cm}
\begin{subfigure}{.4\textwidth}
    \centering
    \includegraphics[trim={4cm 8.5cm 4cm 8.5cm}, clip, width=\linewidth]{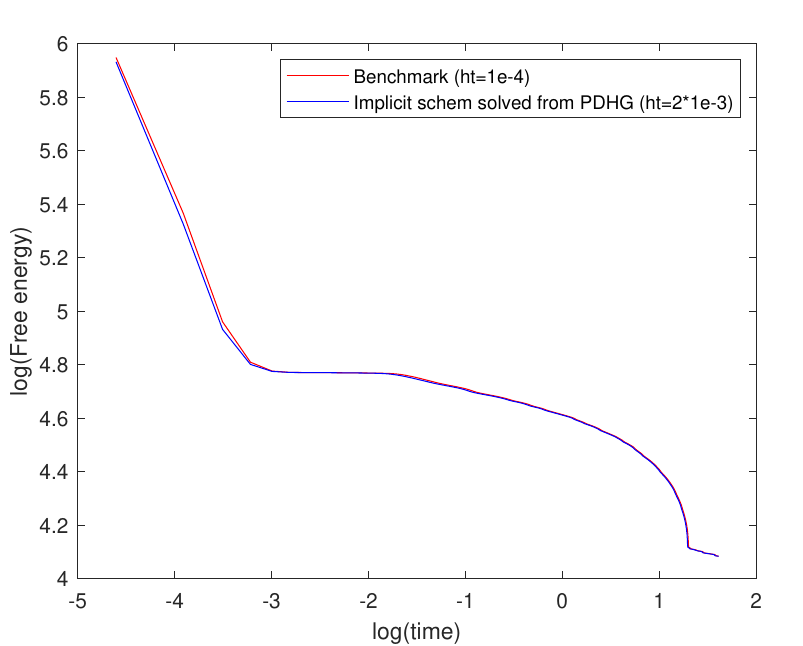}
\end{subfigure}
\caption{ We compute the free energy on $[0, 5].$ (Left) Free energy decay (blue) of the time-implicit scheme (solved by PDHG method) with $h_t=2\cdot 10^{-3}$, and the reference energy decay (red) solved from IMEX scheme with $h_t=10^{-4}.$ The relative error between them is plotted in orange. (Right) The $\log-\log$ plot of free energy.}
\label{fig: free energy and relative error in free energy}
\end{figure}

\begin{figure}[htb!]
    \centering
    \begin{subfigure}{.55\linewidth}
    \includegraphics[trim={1cm 7.5cm 1cm 7.2cm}, clip, width=\linewidth]{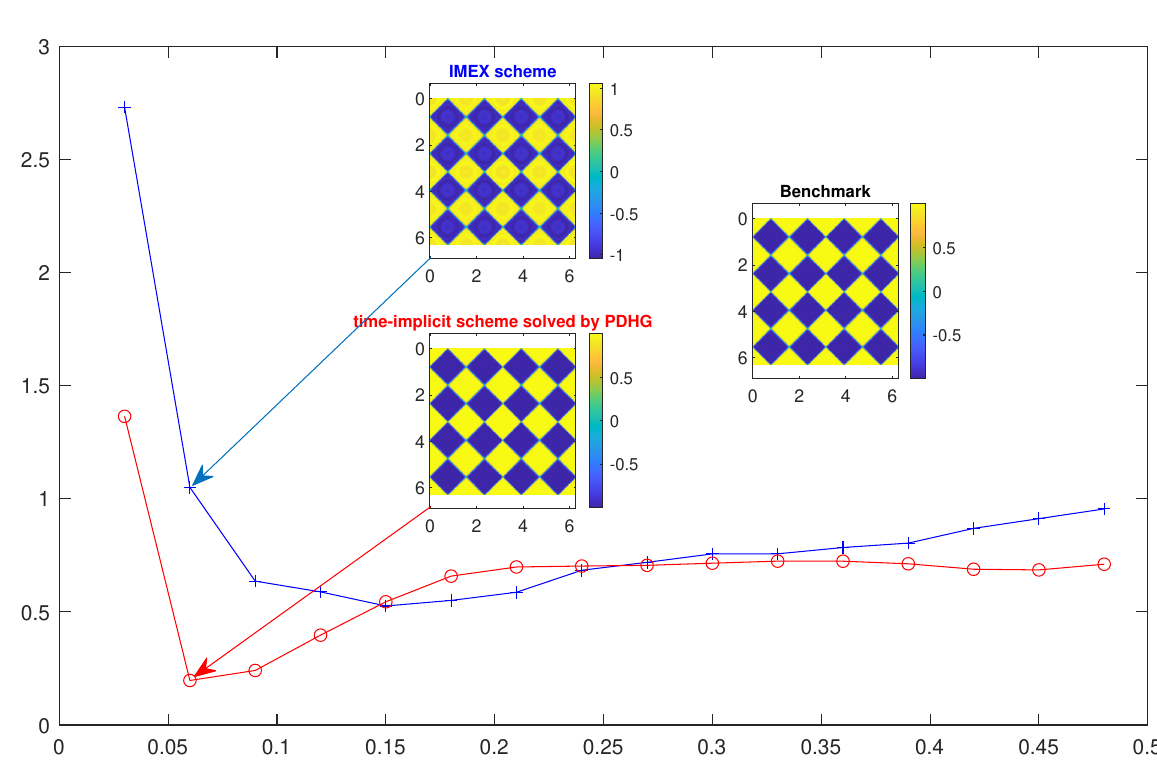}
    \end{subfigure}
    \begin{subfigure}{.27\linewidth}
        \includegraphics[trim={2cm 8cm 2cm 8cm}, clip, width=\linewidth]{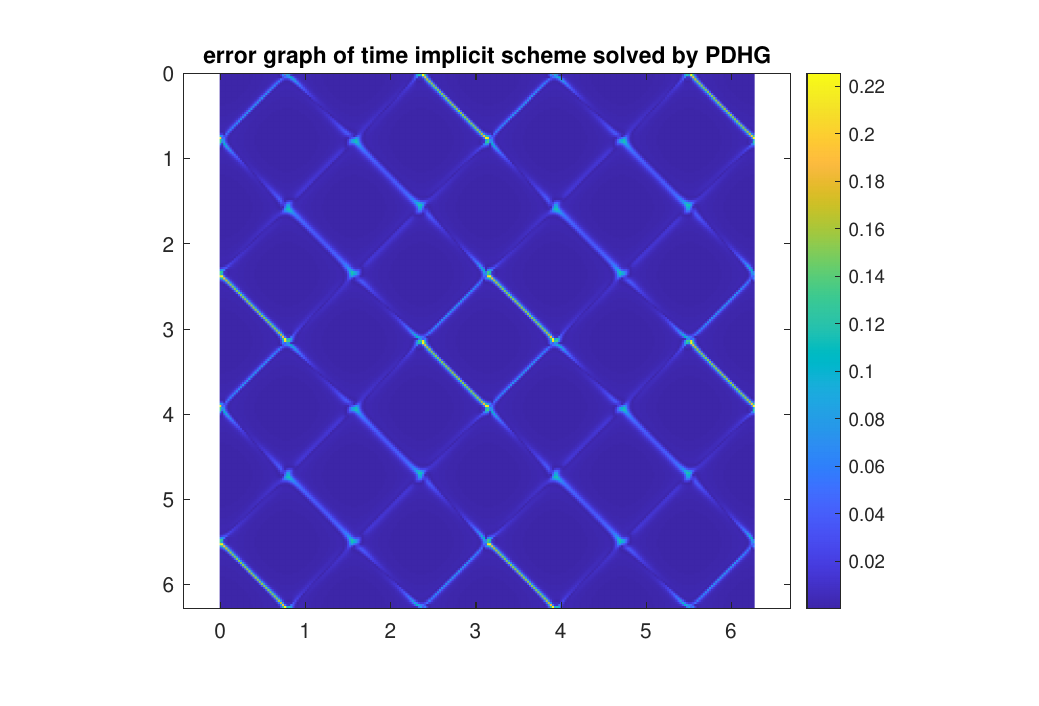}
        \includegraphics[trim={2cm 8cm 2cm 8cm}, clip, width=\linewidth]{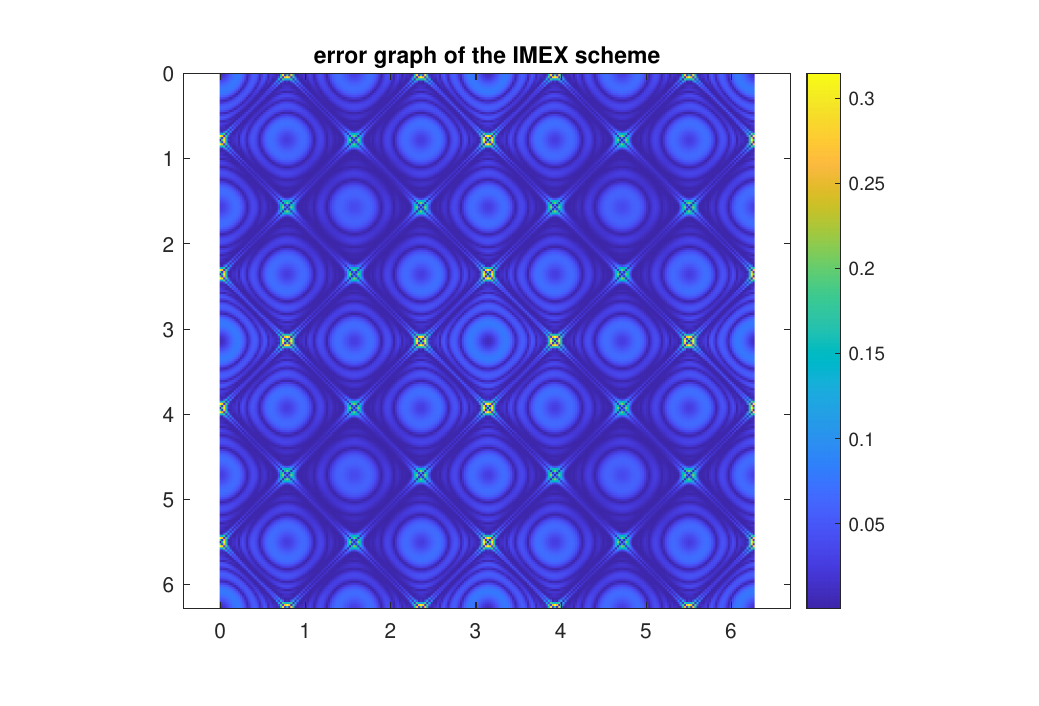}
    
    \end{subfigure}
    \caption{(Left) Comparison between our method (time-implicit scheme solved by the proposed PDHG algorithm) and the IMEX scheme. We discretize the space into a $256\times 256$ lattice. We compute both schemes with large time step size $h_t = 0.01$ and compare with the benchmark solution solved from the same IMEX scheme with $h_t=0.001.$ Blue curve indicates the $L^1$ discrepancy between the IMEX solution on the coarser time grid $U_{\textrm{IMEX}}$ and the benchmark solution $U_{\star}$. Red curve indicates the $L^1$ discrepancy between the time-implicit solution $U_{\textrm{PDHG}}$ and the benchmark $U_\star$. (Right) Plot of $|U_{\textrm{PDHG}}-U_\star|$ (up); and plot of $|U_{\textrm{IMEX}}-U_\star|$ (down). }
    \label{fig: our implicit vs IMEX varcoeff }
\end{figure}

\subsubsection{A 6th-order Reaction-Diffusion Equation (6th-order)}

We consider the following 6th-order Cahn-Hilliard-type equation:
\begin{equation}
  \frac{\partial u}{\partial t} = \Delta (\epsilon_0^2\Delta - (W''(u) - \epsilon_0^2) \mathrm{Id} ) (\epsilon_0^2\Delta u - W'(u)), \quad \textrm{on } [0, 2\pi]^2\times [0, T], \quad u(\cdot,0) = u_0.  \label{6thorder}
\end{equation}
In this example, we choose parameter $\epsilon_0 = 0.18$. We set the initial condition 
\[ u_0(x,y) = 2 e^{\sin x+\sin y-2} + 2.2 e^{-\sin x-\sin y-2}-1. \]
When we set up the precondition matrix $\mathscr{M}$, we approximate $\mathcal G_h$ by 
\[\Delta_h(\epsilon_0^2\Delta_h - W''(\pm 1) + \epsilon_0^2) = \Delta_h(\epsilon_0^2\Delta_h - 2 + \epsilon_0^2),\]
and set $\mathcal L_h = \epsilon_0^2\Delta_h$. We pick $J_f = 2I$. We choose $\tau_U = 0.5, \tau_P = 0.95$ for our PDHG method. A comparison between our proposed scheme and the IMEX scheme is provided in Figure \ref{fig: our implicit vs IMEX 6thOrder }.

\begin{figure}[htb!]
    \centering
    \begin{subfigure}{.55\linewidth}
    \includegraphics[trim={3cm  9cm 2.5cm 8cm}, clip, width=\linewidth]{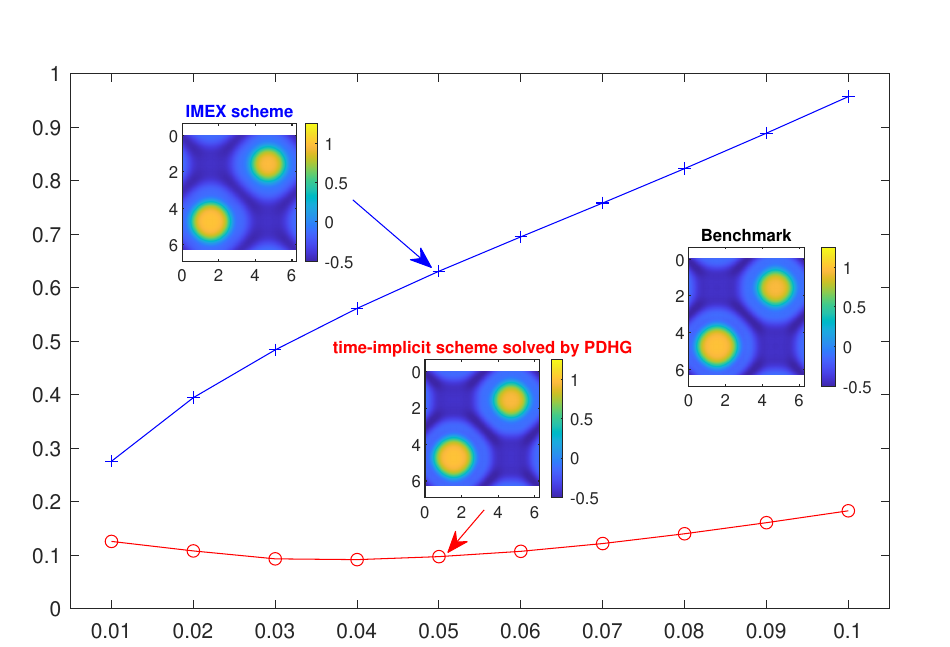}
    \end{subfigure}
    \begin{subfigure}{.27\linewidth}
        \includegraphics[trim={2cm 9cm 2cm 8cm}, clip, width=\linewidth]{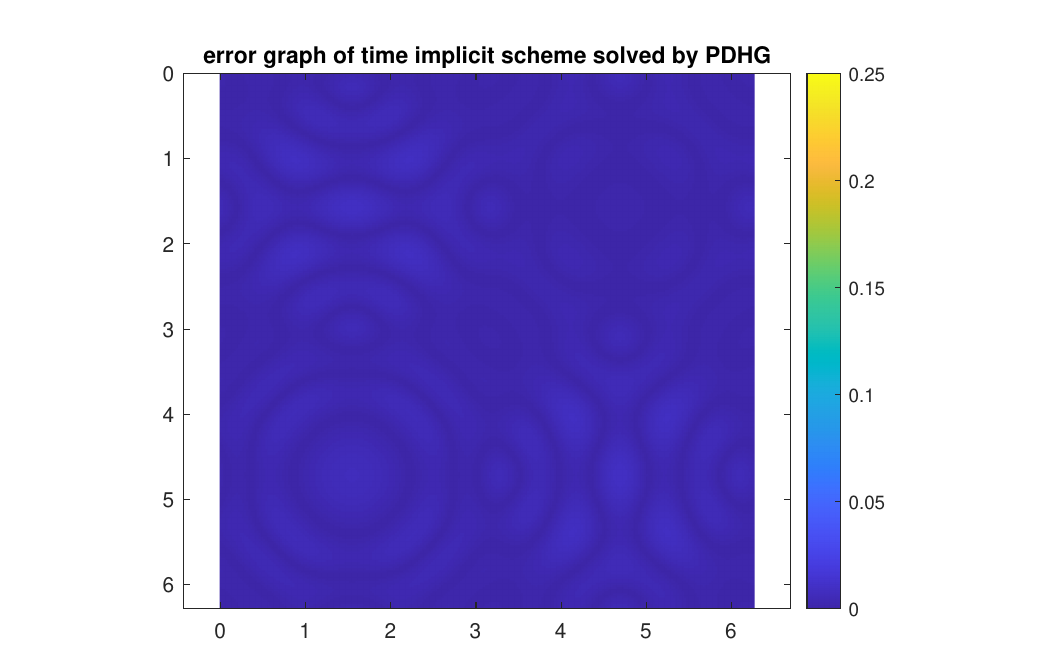}
        \includegraphics[trim={2cm 9cm 2cm 8cm}, clip, width=\linewidth]{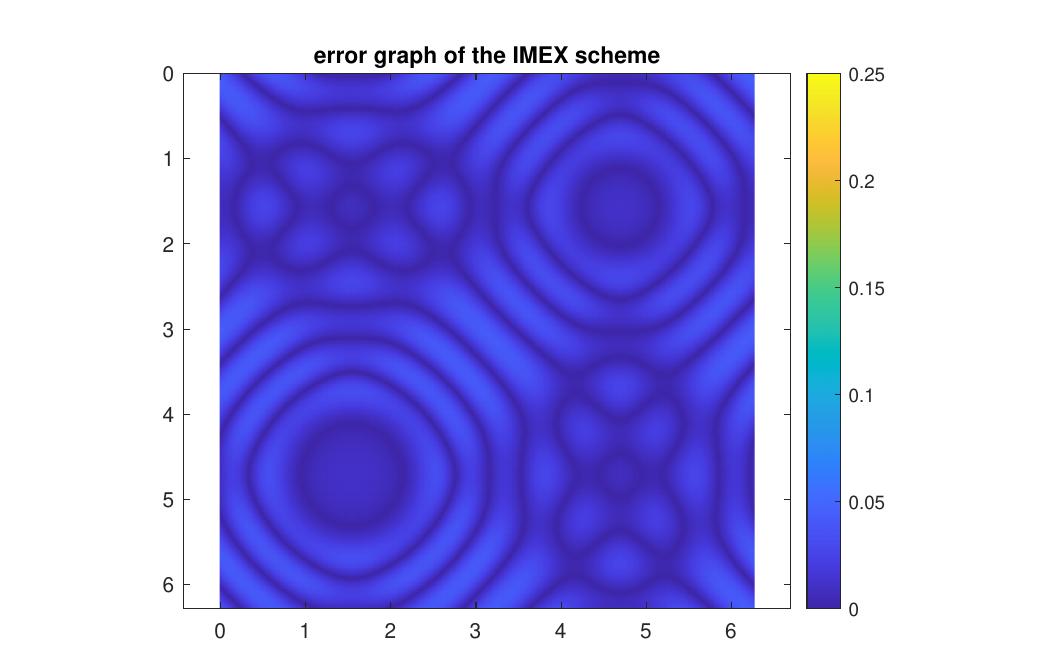}
    \end{subfigure}
    \caption{Similar to Figure \ref{fig: our implicit vs IMEX varcoeff }. (Left) Comparison between the $L^1$ discrepancy of our method and the IMEX scheme with $h_t=0.01$. (Right) Plot of $|U_{\textrm{PDHG}}-U_\star|$ (up); and $|U_{\textrm{IMEX}}-U_\star|$ (down). 
    }
    \label{fig: our implicit vs IMEX 6thOrder }
\end{figure}

\subsubsection{Grid-size-free algorithm} \label{sec: PDHG convergence rate independency wrt grid resolution }
As emphasized previously in the introduction, the convergence rate of our algorithm is independent of the grid size $N_x$. This has also been verified in Corollary \ref{practical corollary } and Corollary \ref{coro: simplify PDHG alg converge}. (Recall that the quantities $\widetilde\theta$ and $\theta$ in these corollaries are independent of $N_x$.) In this subsection, we verify such irrelevance by testing our algorithm on various types of equations with different grid sizes $N_x$. 
The numerical results are demonstrated in Figure \ref{fig: rel iter for convgn - Nx }, where the number of iterations required upon convergence directly reflects the convergence rate of our PDHG algorithm.
\begin{figure}[htb] 
    \centering
    \includegraphics[trim={2cm 8cm 2cm 8.4cm},clip, width=0.68\linewidth]{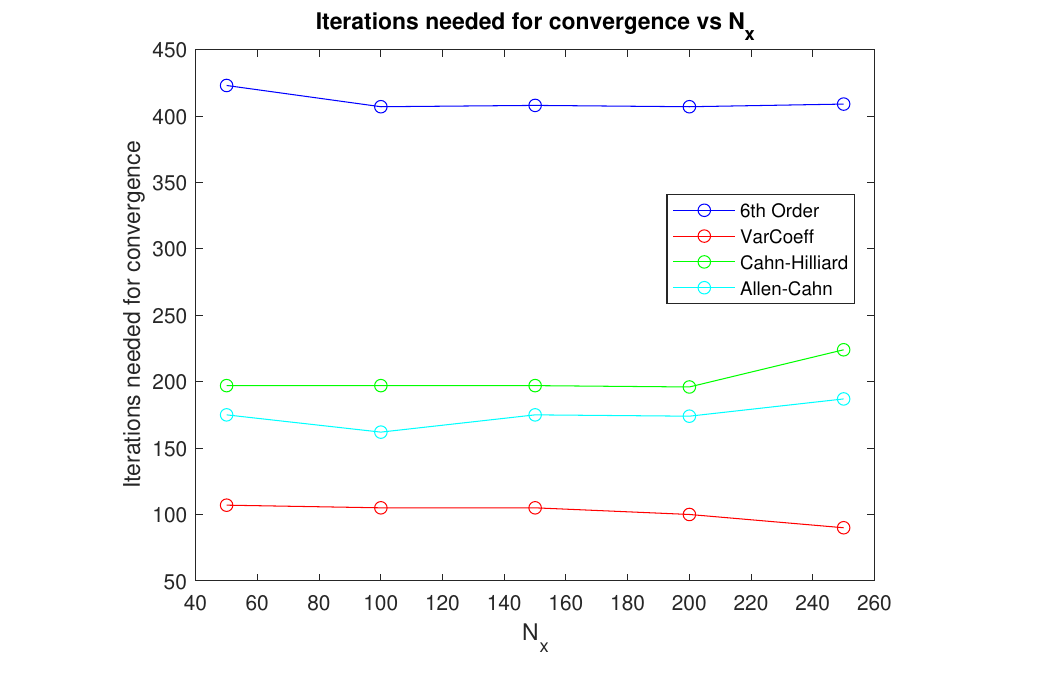}
    \caption{ Relation between the number of iterations needed for convergence and space discretization $N_x$. We verify on four different equations with $N_x = 50, 100, 150, 200, 250.$ We set $\epsilon_0 = 0.01$ for the Allen-Cahn equation and $\epsilon_0 = 0.1$ for the Cahn-Hilliard equation.}
    \label{fig: rel iter for convgn - Nx }
\end{figure}

\subsection{Comparisons with the convex splitting method}
The convex splitting method originally proposed in \cite{eyre1998unconditionally} seeks a specific decomposition of the function $F(\cdot):\mathbb{R}^n\rightarrow\mathbb{R}$, such that the semi-implicit time-discrete scheme applied to the gradient flow 
\begin{equation*}
  \frac{d}{dt} u(t) = - \nabla F(u(t)),
\end{equation*}
is energy stable. To be more specific, suppose $F$ is splitted as $F(u) = F_c(u) - F_e(u)$, where $F_c, F_e$ are convex functions on $\mathbb{R}^n$. Here ``$c$'' denotes contraction, and ``$e$'' denotes expansion, which indicates the effect of the gradient fields $\nabla F_c$ and $\nabla F_e$ in the gradient flow. Consider the scheme
\begin{equation}
  \frac{u^{t+1} - u^t}{2h_t} = - (\nabla F_c(u^{t+1}) - \nabla F_e(u^t)), \quad 0\leq t \leq N_t . \label{cvxsplit_schm}
\end{equation}
It can be shown that $F(u^{t+1})\leq F(u^t)$ for all $t=0, 1, 2, \dots$, i.e., the numerical scheme preserves the decaying of energy.

Many RD equations can be interpreted as gradient flows in certain functional spaces. The convex splitting method has been widely applied to compute these equations. We refer the readers to \cite{xu2019stability} and the references therein for more details. In comparison, we demonstrate that our proposed algorithm, which employs the PDHG method for solving the time-implicit scheme, offers notable advantages over the convex splitting methods. Specifically, it achieves higher accuracy in computing the phase-field models with weak diffusion and strong reaction.

Recall the Allen-Cahn equation \eqref{AC_equ}, which can be cast as the $L^2$ gradient flow of the free energy $\int_\Omega \frac{\epsilon_0}{2} \|\nabla u\|^2\ dx + \frac{1}{\epsilon_0}\int_\Omega W(u) \ dx$. For the numerical solution, we adopt the finite difference scheme and consider the discretized energy function\footnote{Here we denote discrete gradient $\nabla_{h_x} U_{ij} := (\frac{U_{i+1, j} - U_{i-1, j}}{h_x}, \frac{U_{i, j+1} - U_{i, j-1}}{h_x}). $} $F(U) = \frac{\epsilon_0}{2} \sum_{i,j} h_x^2 \|\nabla_{h_x} U_{ij}\|^2 + \frac{1}{\epsilon_0}\sum_{i, j} h_x^2 W(U_{ij})$. Following the discussion in \cite{gu2018convex}, we decompose the double-well potential $W(u)=\frac{1}{4}(u^2-1)^2=W_c(u) - W_e(u)$ in two ways\footnote{Although $W_e(\cdot)$ for scheme (A) is not convex on $\mathbb{R}$, it is convex on the finite interval $[-1.7, 1.7]$. This remains a reasonable splitting as long as $U_{ij}^t$ lies in this interval for arbitrary $1\leq i, j \leq N_x$, $ 0 \leq t \leq N_t$.}:
\begin{align*}
   \textrm{(A)} \quad &  W_c(u) = \frac{1}{2}u^2, \quad W_e(u) = -\frac{1}{4}u^4 + \frac32u^2 - \frac14; \\
   \textrm{(B)} \quad &  W_c(u) = \frac{1}{4}u^4+\frac14, \quad W_e(u) = \frac12u^2.  
\end{align*}
One then considers $F_c(U) = \frac{\epsilon_0}{2} \sum_{i,j} h_x^2 \|\nabla_{h_x} U_{ij}\|^2 + \frac{1}{\epsilon_0}\sum_{i, j} h_x^2 W_c(U_{ij})$ and $F_e(U) = \frac{1}{\epsilon_0}\sum_{i, j} h_x^2 W_e(U_{ij})$. The convex split scheme \eqref{cvxsplit_schm} yields
\begin{equation}
  (I - \epsilon_0 h_t \Delta_{h_x} ) U_{ij}^{t+1} + \frac{h_t}{\epsilon_0} W'_c(U_{ij}^{t+1})   = U^t_{ij} + \frac{h_t}{\epsilon_0}  W'_e(U_{ij}^t), \quad 0 \leq t \leq N_t.  \label{Allen Cahn cvxsplit scheme}
\end{equation}
It is worth mentioning that \eqref{Allen Cahn cvxsplit scheme} reduces to a linear equation if $W_c(\cdot)$ is quadratic. Otherwise, \eqref{Allen Cahn cvxsplit scheme} is a nonlinear root-finding problem and the proposed PDHG algorithm can be applicable here to resolve for $U^{t+1}$.

We apply \eqref{Allen Cahn cvxsplit scheme} using splitting schemes (A) and (B) to (4.1) with \(\epsilon_0 = 0.01\) and compare the results with the time-implicit scheme. The numerical results are presented in Figure \ref{fig: compare fis and css}. As shown in the results, a small \(\epsilon_0\) in this phase-field model poses a challenge for the convex splitting methods, as they are unable to accurately capture the movement of the zero-level set of \(u(\cdot, t)\). In contrast, the time-implicit scheme maintains computational accuracy. Further comparisons between the time-implicit scheme and the convex splitting method can be found in \cite{xu2019stability}.
\begin{figure}[htb!]
\centering
\begin{subfigure}{.4\textwidth}
    \centering
    \includegraphics[trim={4cm 8.5cm 4cm 8.5cm}, clip, width=\linewidth]{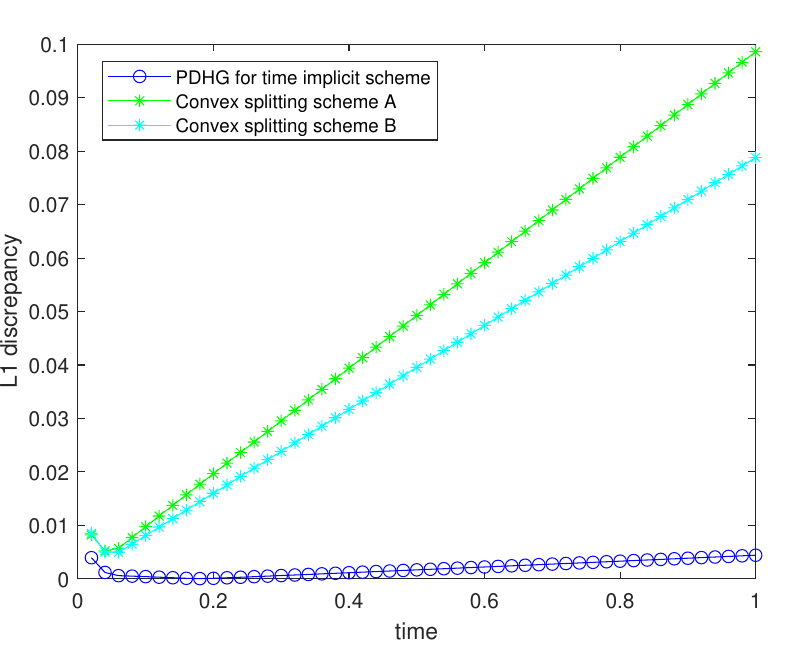}
\end{subfigure}
\hspace{0.5cm}
\begin{subfigure}{.4\textwidth}
    \centering
    \includegraphics[trim={4cm 8.5cm 4cm 8.5cm}, clip, width=\linewidth]{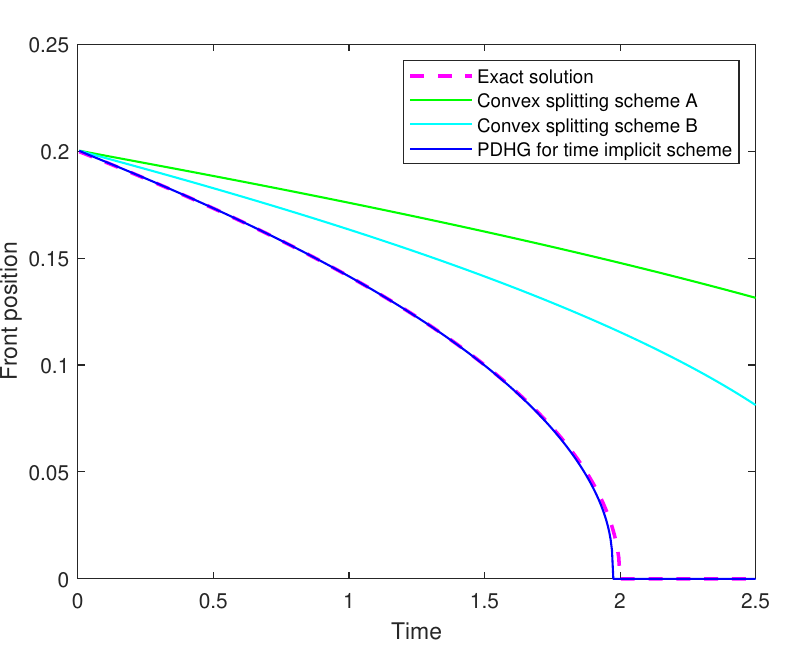}
\end{subfigure}
\caption{ Comparison between the time implicit scheme and the convex splitting scheme: (Left) We discrete the space into $128\times 128$ lattices. Similar to Figure \ref{fig: our implicit vs IMEX AC }, we compute both schemes with rather large time step size $h_t = 0.02$ and plot the $L^1$ discrepancy curve versus time (with the benchmark solution solved from the IMEX scheme with $h_t = 0.001$). (Right) We discretize the space into $256\times 256$ lattices, pick $h_t=0.005$, and plot the front position versus time for different numerical schemes.}
\label{fig: compare fis and css}
\end{figure}

\subsection{Hyperparameter selection}\label{sec: hyperparam select}
Given the spatial and the temporal step sizes $h_x, h_t$ of the implicit scheme, there are 5 hyperparameters to be determined for our algorithm: $N_t, \tau_U, \tau_P, \omega$, and $\epsilon$. In the following, we discuss the choice of these hyperparameters.
\vspace{0.1cm}
\begin{enumerate}[wide, labelindent=0pt]
    \item (Choosing $N_t$) {As mentioned previously in section \ref{subsection  compute with time causality}, one can distribute the computational task into multiple blocks and apply PDHG algorithm to evaluate each block of solutions sequentially.} Suppose we aim to solve an equation on $[0, T_{\textrm{total}}]$. We may divide the time interval into $M\cdot N_t$ subintervals, i.e.,
    \begin{align*}
        [0, T_{\textrm{total}}] = \bigcup_{k=1}^M I_k = \bigcup_{k=1}^M \left(\bigcup_{j=1}^{N_t} I_{k, j}\right), \quad  &  \textrm{where each } I_{k,j} = [(k-1)T+(j-1)h_t,(k-1)T+jh_t].\\
        & \textrm{with } T = T_{\textrm{total}}/M, h_t = T/N_t.
    \end{align*}
    We then apply our proposed method to each subinterval $I_k$ in order to obtain the entire numerical solution on $[0, T_{\textrm{total}}]$. We test our algorithm with different combinations of $M\cdot N_t$ on various types of equations. Unless specified otherwise, we choose $\omega=1, \epsilon=0.1$. We set the stopping criteria as $\|\mathrm{Res}(U_k)\|_{\infty}<10^{-6}$. The efficiency of our algorithm under different scenarios is reflected in CPU time demonstrated in Table \ref{tab: CPU time with diff Nt }. {Among the series of experiments, we observe that it is usually efficient to pick $N_t\leq 5$.}
    \begin{table}[htb!]
     \small
     \centering
     \begin{tabular}{ c|ccccccccc }
        \multirow{2}{*}{Equation Name [$\tau_U,\tau_P,{T_{\textrm{total}}}$]}& \multicolumn{9}{|c}{$M\times N_t$} \\
        \cline{2-10}
          & $1\times 100$ & $2\times 50$ & $4\times 25$ & $10\times 10$ & $20\times 5$ & $25 \times 4$ & $33 \times 3 + 1$ & $50\times 2 $ & $100\times 1$ \\
        \hline\hline
        AC($\epsilon_0 = 0.01$) [$0.5, 0.5, 1.0$] & -- & -- & $ 1198.41 $  & $219.52$ &$137.71$  & $138.65$ & \textbf{88.53}    &$106.41$  &$92.72$ \\
        \hline
        AC($\epsilon_0 = 0.1$) [$0.5, 0.5, 1.0$] & -- & -- & 90.28 & 57.73 & 34.37 & 50.43 & 41.37 & 26.62&\textbf{24.20}\\
        \hline
        AC($\epsilon_0 = 1$) [$0.5, 0.5, 1.0$] & 64.28 & 38.11 & 23.42 & 24.24 & 13.05 & 13.29 & 12.51 & 10.89 &\textbf{10.72}  \\
        \hline
        CH [$0.5, 0.5, 1.0$] & 775.15 &  208.93  & 170.77 & 252.99 &148.96&  183.34  & 101.41 & \textbf{77.35} & 86.37 \\
         \hline
        6th Order [$0.8, 0.8, 0.1$] & -- & -- &  374.82  & 389.90 & 285.12 & 384.52 & 199.11 & \textbf{188.58} & 208.30 \\
        \hline
        { Varcoeff [$0.95, 0.5,1.0$]}  & -- & -- & 305.73 &  206.72 & 204.34 & 153.88 & 144.67 & 142.22 & \textbf{61.46 } \\
        \hline
    \end{tabular}
    \caption{Comparison of CPU time (s) with different $N_t$s (All problems are solved on $256\times 256$ grids).}
    \label{tab: CPU time with diff Nt }
   \end{table}
    \vspace{0.2cm}
    \item (Choosing $\tau_U, \tau_P$) Theoretically, choosing $\tau_U, \tau_P$ as suggested in Corollary \ref{coro: simplify PDHG alg converge} will guarantee the convergence of our method. In practice, we can pick a larger $\tau_U, \tau_P$ to achieve faster convergence. Generally speaking, the optimal step size $\tau_P$ is around $0.9$, and the optimal ratio $\varrho = \frac{\tau_P}{\tau_U}$ should be  slightly less than $2$. The intuition of choosing $\varrho>1$ is that we want to treat the inner optimization of the functional $\widehat{L}(U, Q)$ defined in \eqref{def: hat L(U, Q)} w.r.t. the dual variable $Q$ more thoroughly. In fact, it is common in bi-level optimization to choose a larger, more aggressive step size for the inner-level optimization problem both practically \cite{finn2017model}  and theoretically \cite{lin2020gradient,zuo2022understanding}. A rather efficient choice of the step sizes $(\tau_U, \tau_P)$ is $(0.5, 0.9)$. This is verified in Table \ref{tab: compare diff ratio of tau_u tau_p }, in which we compare the choice $(0.5, 0.9)$ with other combinations of $(\tau_U, \tau_P)$. 
    \begin{table}[htb!]
    \centering
    \begin{tabular}{c|c|ccc }
       \multicolumn{2}{c|}{ $ \epsilon = 0.1$ for all problems }   & $\tau_U = 0.9, \tau_P=0.5$ & $\tau_U = 0.65, \tau_P=0.65$ & $\tau_U = 0.5, \tau_P=0.9$\\
                         \hline\hline
    \multirow{2}{*}{6th Order [$T=0.5$]} & $N_x=256, N_t=50$ & 62.28 & 47.92 & \textbf{30.53}\\
        \cline{2-5}
       & $N_x=128, N_t=50$ & 12.31 & 9.47 &  \textbf{8.54}  \\
        \hline
      \multirow{2}{*}{VarCoeff [$T=0.5$]} 
    & $N_x = 256, N_t=50$ & 103.23 & 109.38 & \textbf{82.38} \\
        \cline{2-5}
       & $N_x = 128, N_t = 50$ & 15.92 & 13.35 & \textbf{9.54} \\
        \hline
    \end{tabular}
    \caption{Comparison on speeds among different ratios $\varrho = \frac{\tau_P}{\tau_U}$ for different equations.}\label{tab: compare diff ratio of tau_u tau_p }
    \end{table}
    \vspace{0.2cm}
    \item (Choosing $\omega$) We pick $\omega=1$ in our experiments. If one increases or decreases $\omega$, one should modify $\tau_P$ correspondingly so that $\widetilde\gamma = \omega  \tau_P$ remains unchanged. Once $\widetilde\gamma\approx 0.9$ is fixed, we generally achieve the optimal (or near-optimal) performance of our algorithm.
    \vspace{0.2cm}
    \item (Choosing $\epsilon$) We set $\epsilon$ around $0.1$. Recall that ${\sup}_Q ~ \{ \widehat{L}(U, Q) \} = \frac{\|\widehat{F}(U)\|^2}{2\epsilon}$. Increasing $\epsilon$ will decrease the convexity of the functional $\frac{\|\widehat{F}(U)\|^2}{2\epsilon}$, which will slow down our algorithm. Decreasing $\epsilon$ brings our algorithm closer to our original version of PDHG method \cite{liu2024first}, in which we discover stronger oscillations towards convergence, which may also affect the efficiency.
\end{enumerate}

\subsection{Long-time computation via adaptive time step size}\label{sec : long-range}
It is an important topic how one can efficiently compute the RD equation for large time $T$ to study its behavior near the equilibrium state. Since 
we can pick large time step size $h_t$ under the implicit scheme, our proposed method offers an opportunity for faster computations to approximate the equilibrium state of RD equations. 

To be more precise, we adopt adaptive time step size $h_t$ during the update of time-implicit scheme \eqref{time-implicit scheme}. Suppose we set up an upper bound $\bar{h}_t>0$ for time step size $h_t$. As $h_t < \bar{h}_t$, we increase $h_t$ by $10\%$ if the proposed PDHG algorithm converges in less than $\bar{n}$ steps. Otherwise, we decrease $h_t$ by $50\%$. If $h_t$ exceeds $\bar h_t$, we reset $h_t=\bar h_t$. 

We implement this strategy of adaptive time step size on equation \eqref{VarCoeff} with $T=20$. As we pick $\epsilon_0=0.01$, \eqref{VarCoeff} possesses weak mobility-diffusion and strong reaction. We solve the equation with $N_x=256$, and set the initial time step size $h_t=0.01$, we set $\bar{h}_t=0.08$. As shown in Figure \ref{fig:  long-range }, our method works efficiently in this example, with an average $h_t\approx0.04$. We also compute the same equation by using the classical IMEX method \cite{hundsdorfer2003numerical} in which we treat the linear part as implicit and the nonlinear part as explicit. We apply the preconditioned conjugate gradient (PCG) algorithm with tolerance\footnote{Suppose we apply PCG algorithm to solve the linear equation $Ax=b$ with $A$ positive definite. Denote $x_k$ as the solution obtained at the $k$-th iteration of the PCG algorithm, then we terminate the PCG iteration if $\|Ax_k-b\|_\infty\leq \eta.$} $\eta = 10^{-10}$ to solve the linear system at each IMEX step. For \eqref{VarCoeff}, the IMEX method only works stably for a rather small time step size $h_t\leq 0.5\cdot 10^{-3}$. 
\begin{figure}[htb!]
    \centering
    \begin{subfigure}{.31\linewidth}
    \includegraphics[trim={2cm  8.5cm 1.5cm 8cm}, clip, width=\linewidth]{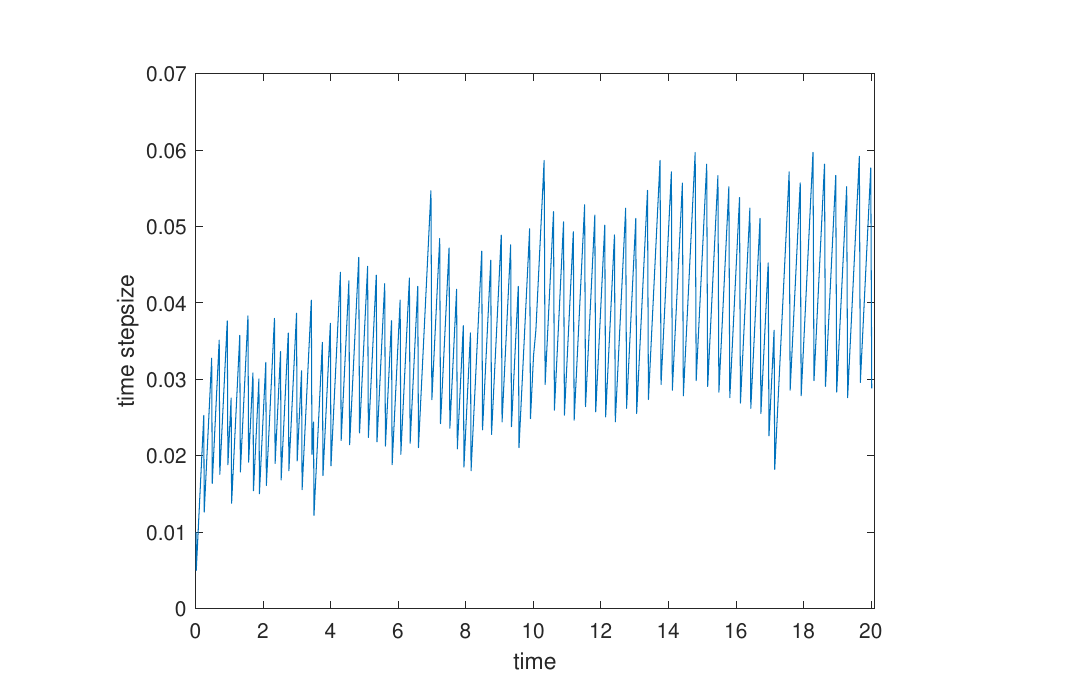}
    \end{subfigure}
    \begin{subfigure}{.31\linewidth}
        \includegraphics[trim={2cm  8.5cm 1.5cm 8cm}, clip, width=\linewidth]{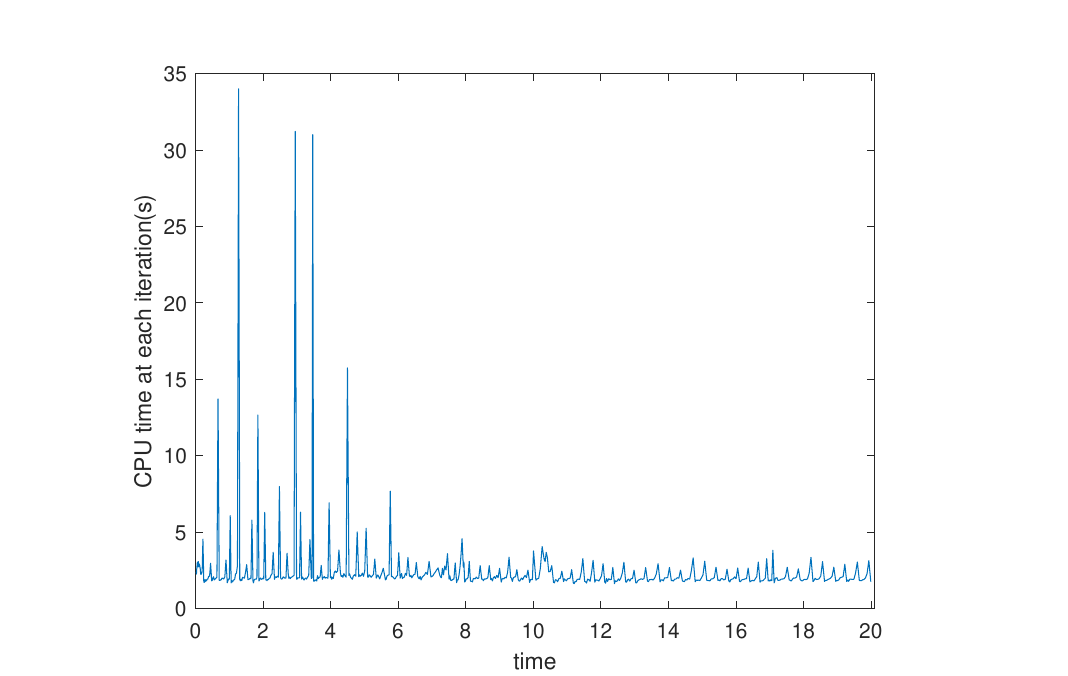}
    \end{subfigure}
   \begin{subfigure}{.31\linewidth}
        \includegraphics[trim={2cm  8.5cm 1.5cm 8cm}, clip, width=\linewidth]{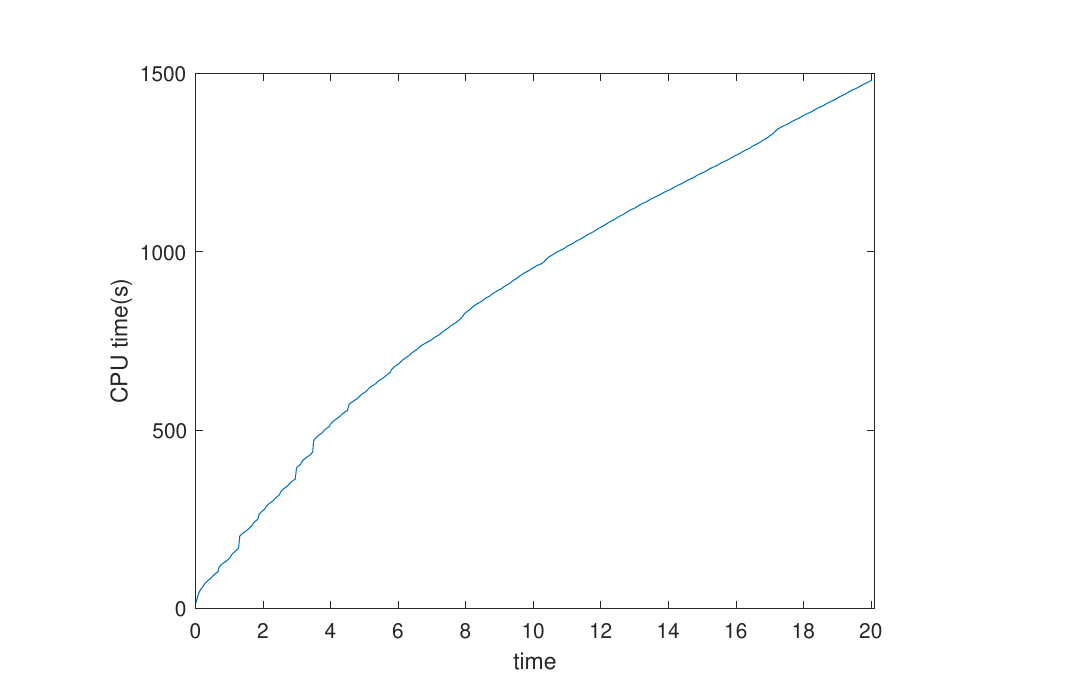}
    \end{subfigure}
    \caption{(Left) Plot of time step size $h_t$ versus physical time $t$; (Middle) Plot of the time cost of PDHG iterations versus physical time $t$; (Right) Plot of accumulated CPU time (s) versus physical time $t$.}
    \label{fig:  long-range }
\end{figure}
As reflected in Table \ref{tab:  compare time PDHG vs IMEX}, our method works better on long-time computation.
\begin{table}
    \small
    \centering
    \begin{tabular}{c|ccc}
        \multirow{2}{*}{ Our method } & \multicolumn{3}{|c}{ IMEX  }\\
        \cline{2-4}
         &  $h_t=0.5\cdot10^{-3}$  & $h_t = 0.2 \cdot 10^{-3}$ & $h_t = 10^{-4}$  \\
        \hline\hline
        1481.76 s & 1814.40 s &  4158.18 s  &  6216.81 s     \\
        \hline
    \end{tabular}
    \caption{Comparison of CPU time (s) between our treatment and the classical IMEX method on computing the equation \eqref{VarCoeff} on $[0, 20]$.
    }
    \label{tab:  compare time PDHG vs IMEX}
\end{table}

\subsection{Comparison on computational efficiency}\label{sec: compare speed with 3 methods }
In this section, we compare the computational efficiency (in CPU time) of the proposed method with some classical algorithms used for solving time-implicit schemes of the reaction-diffusion equations.
\begin{enumerate}
\item (\textbf{Nonlinear SOR}) The Nonlinear SOR (NL SOR) method is the nonlinear version of the successive over-relaxation (SOR) algorithm. It is used to solve the implicit scheme of the Allen-Cahn equation \eqref{AC_equ} in \cite{merriman1994motion}. We set the tolerance of the Newton's method used in NL SOR as $10^{-10}$. We set $\tau_U=0.55, \tau_P=0.95$ for our PDHG method. We compare NL SOR with our algorithm in Figure \ref{fig: compare CPU time NL SOR vs PDHG }.
\begin{figure}[htb!]
  \centering
  \includegraphics[trim={3.4cm 8.2cm 3.4cm 9cm}, clip, width=0.44\linewidth]{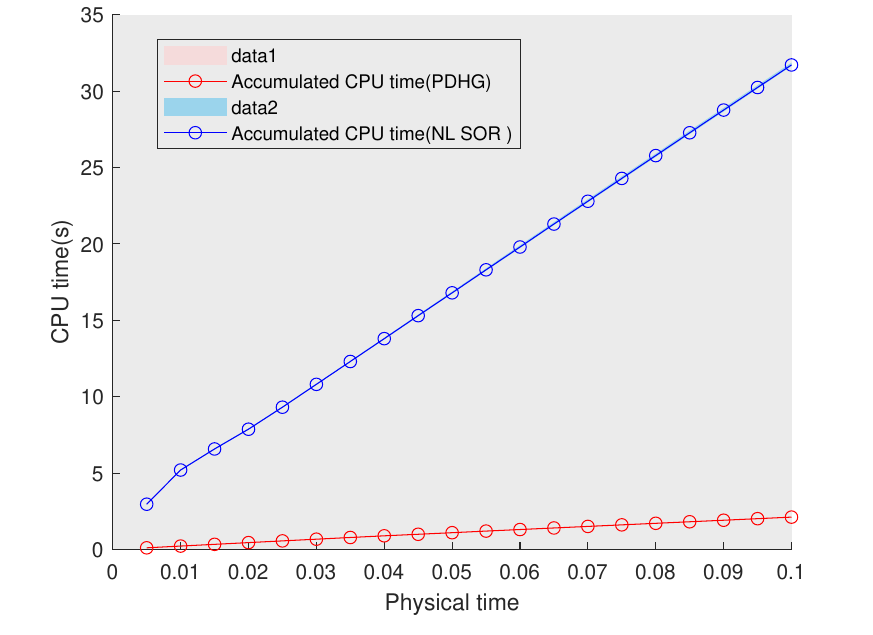}
  \caption{Accumulated CPU time comparison between our method (red) and Nonlinear SOR method (blue) applied to Allen-Cahn equation \eqref{AC_equ} with $\epsilon_0 = 0.1$ and $h_t=0.005.$ We solve the equation on a $128\times 128$ grid. The quantile plots are composed based on $40$ independent runs of both algorithms. }\label{fig: compare CPU time NL SOR vs PDHG }
\end{figure}

\item (\textbf{Fixed point method}) The fixed point method is also a frequently used algorithm to solve the time-implicit scheme of the RD equation. We reformulate the time-implicit scheme \eqref{time-implicit scheme} as
\begin{equation*}
  (I + a h_t \mathcal G_h \mathcal L_h) U^{t+1} = U^t - b h_t \mathcal G_h f(U^{t+1}). \label{fixed pt schemee }
\end{equation*}
For fixed $U^t$, we establish the following fixed point iteration for solving $U^{t+1}$,
\begin{equation*}
  U_{k+1} = (I +  h_t \mathcal G_h ( a \mathcal L_h + b c I) )^{-1} (U^t - b h_t \mathcal G_h (f(U_k) - cU_k)), \quad \textrm{with initial guess } U_0=U^t.
\end{equation*}
Here $c$ is a tunable constant that can be chosen as the value of $f'(\cdot)$ at equilibrium state. When $f(u)=W(u)=\frac{1}{4}(1-u^2)^2$, we set $c=f'(\pm 1) = 2.$ The linear system is solved by the PCG algorithm with tolerance $\eta = 10^{-10}$. We set $\tau_U=0.5, \tau_P=0.95$ for our PDHG method. We apply both algorithms to \eqref{VarCoeff} with $\epsilon_0=0.1$. We compare the fixed point method with our algorithm in Figure \ref{PCGFP computn time}. 
\begin{figure}[htb!]
  \centering
    \includegraphics[trim={3.4cm 8.2cm 3.4cm 9cm}, clip, width=0.48\linewidth]{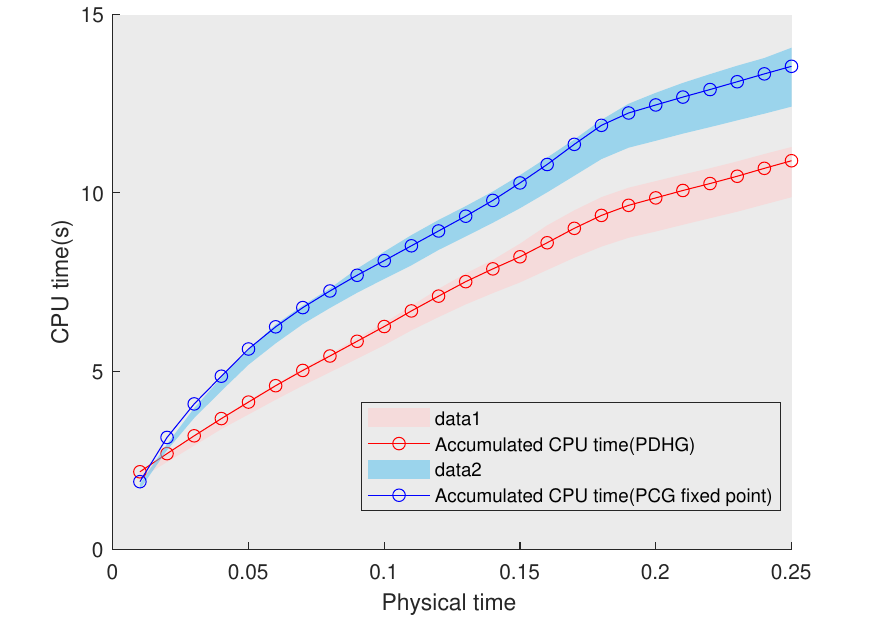}
\caption{Accumulated CPU time comparison between our method (red) and PCG-fixed point iteration (blue). We solve \eqref{VarCoeff} with $\epsilon_0=0.1$ and $h_t=0.01$ on a $256\times 256$ grid. These quantile plots are composed based on $40$ independent runs of both algorithms. }\label{PCGFP computn time}
\end{figure}

\item (\textbf{Newton's method}) Newton's method with the PCG algorithm as its linear solver serves as a popular tool for solving implicit schemes of RD equations with a higher order of spatial differentiation. Here we consider Newton's method introduced in section 3 of \cite{christlieb2014high}. In \cite{christlieb2014high}, Newton's method is applied to the spectral discretization of the solution while here we apply Newton's method to the finite difference scheme. We set $\tau_U=0.5, \tau_P = 0.95$ for our PDHG method. We apply both methods to \eqref{6thorder}. According to our experiments, we observe that when the time step size $h_t\leq 0.005$, Newton's method works more efficiently than the PDHG algorithm. When $h_t > 0.005$,
the PDHG method is faster. Such observation is reflected in Figure \ref{fig: Newton computn time}. Table \ref{tab: PDHG vs PCG Newton } demonstrates that the PDHG method is more efficient than Newton's method when the latter is applied to multi-interval computation with smaller time step sizes.
\begin{figure}[htb!]
\centering
\begin{subfigure}{0.46\textwidth}
    \centering
    \includegraphics[trim={3.4cm 8.2cm 3.4cm 9cm}, clip, width=\linewidth]{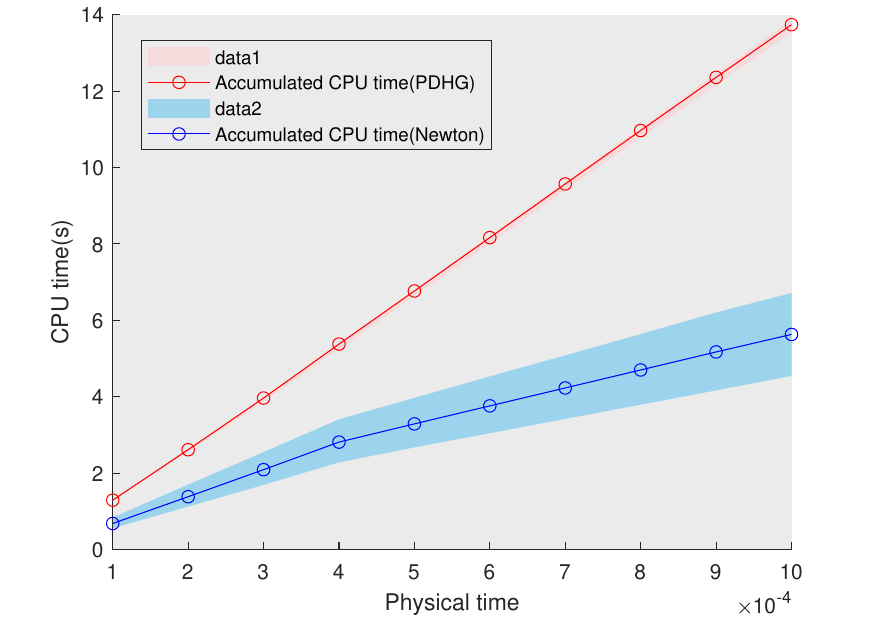}    
\end{subfigure}
\begin{subfigure}{0.46\textwidth}
    \centering
    \includegraphics[trim={3.4cm 8.2cm 3.4cm 9cm}, clip, width=\linewidth]{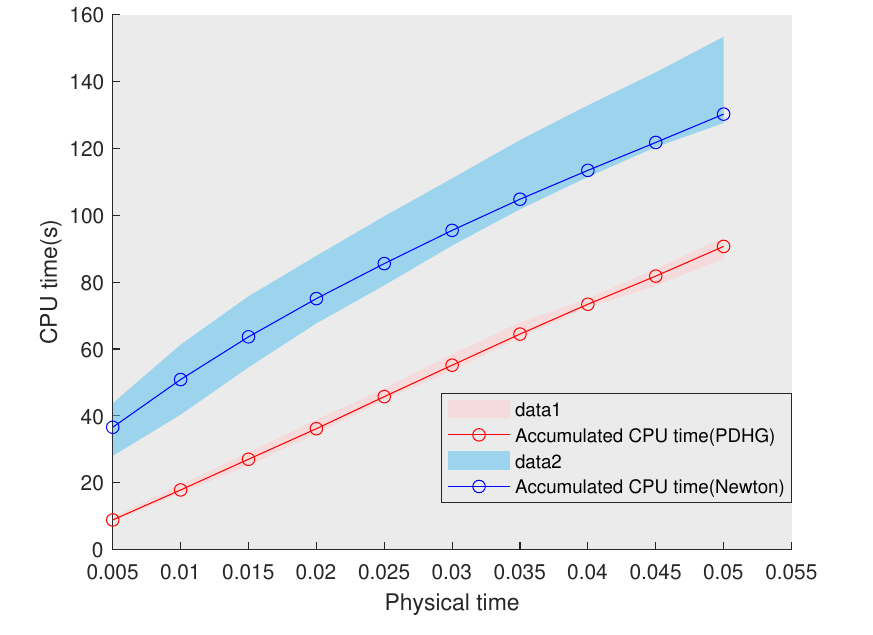}
\end{subfigure}
\caption{Accumulated CPU time comparison between our method (red) and Newton's method (blue). Solving equation \eqref{6thorder} with $h_t=0.001$ (Left) and $h_t=0.005$ (Right). We solve the equation on a $256\times 256$ grid. These quantile plots are composed based on $40$ independent runs of both algorithms. }\label{fig: Newton computn time}
\end{figure}

\begin{table}[htb!]
    \centering
    \begin{tabular}{c|c|cccc}
       Method  & PDHG & \multicolumn{4}{c}{ PCG Newton's method } \\
       \hline
       $h_t\times n$  & $0.01\times 50$ & $0.005\times 100$ & $0.001\times 500$ & $0.0005\times 1000$ & $0.00025\times 2000$\\
       \hline\hline 
       CPU time(s)  & 263.90 & 422.28 & 299.02 & 470.71 & 773.01   \\
    \end{tabular}
    \caption{Time costs of applying the PDHG method and  Newton's method to \eqref{6thorder} on $256\times 256$ grid.  }
    \label{tab: PDHG vs PCG Newton }
\end{table}
\end{enumerate}

\section{Conclusion}
In this research, we reformulate the PDHG algorithm proposed in \cite{liu2024first} by introducing a quadratic regularization term to solve implicit schemes of RD equations. Theoretically, we establish unique existence results for the time-implicit schemes of general RD equations. We further prove the exponential convergence for both the PDHG flow and the proposed discrete-time PDHG algorithm. In addition, we show that the convergence rates are independent of the grid size $N_x$. Our theoretical results  are also supported by numerous numerical experiments. We test the proposed PDHG method via four different types of reaction-diffusion equations. Based on these numerical examples, we verify the optimal (or near-optimal) way to set the hyperparameters of our algorithm. We also verify the efficiency of our method by comparing it with several classical root-finding algorithms, such as the nonlinear SOR method, the fixed point method, and Newton's method. 

We end the discussion by mentioning important future directions.
\begin{itemize}
    \item The convergence rate achieved in this research is not the sharpest rate. Can we establish a sharp convergence rate in terms of the algorithm's hyperparameters?
    \item Currently, all of the proposed preconditioners are time-independent. How can we design a more sophisticated time-dependent preconditioner to assist the convergence of the generalized PDHG algorithm?
    \item As we accumulate multiple time intervals together to formulate a saddle-point scheme for the root-finding problem, we cancel the causalities among different time nodes. Will this causality-free optimization strategy render the possibility of parallel computing for the proposed PDHG time-implicit solvers?
    \item {While the proposed algorithm performs efficiently on reaction-diffusion equations, we aim to extend our approach to simulate general equations in physical modeling, including Fokker-Planck equations and their generalizations in complex systems.}
    \item {Extend the proposed primal-dual approach to nonlinear, high-dimensional equations by integrating deep learning algorithms.}
\end{itemize}

\bibliographystyle{plain}
\bibliography{references}


\newpage
\appendix

\section{Proofs of section \ref{sec: unique existence}}

\subsection{Proof of Theorem \ref{thm: unique existence root-finding}}\label{app: proof th}
\begin{proof}[Proof of Theorem \ref{thm: unique existence root-finding}]
To prove this result, we only need to prove that the following single-step scheme
\begin{equation}
  \frac{U-U^0}{h_t} = - \mathcal G_h ( a\mathcal L_h U + b f(U) ),\label{one step implicit scheme }
\end{equation}
admits a unique solution $U$ for arbitrary $U^0.$ By writing $\xi = U - U^0$, we reformulate \eqref{one step implicit scheme } as
\begin{equation}
  \frac{\xi}{h_t} + \mathcal G_h(a\mathcal L_h (U^0 + \xi) + b f(U^0 + \xi)) = 0.\label{reformulate root-finding}
\end{equation}
We first show that $\xi$ solves \eqref{reformulate root-finding} iff $\xi$ is the critical point of the following variational problem
\begin{equation}
  \min_{\xi\in\mathrm{Ran}(\mathcal G_h)} \; \left\{  \frac{\xi^\top \mathcal{G}_h^\dagger \xi }{2h_t} + \frac{a}{2} (U^0 + \xi)^\top\mathcal L_h (U^0 + \xi) + bW(U^0 + \xi)^\top \boldsymbol{1}  \right\}.\label{proximal scheme reaction diffusion}
\end{equation}
Here we denote $W(\cdot)$ as the primitive function of $f(\cdot).$ Let us define $\mathcal V = \mathrm{Ran}(\mathcal G_h)$ and $\mathcal J(\xi)$ as the function in \eqref{proximal scheme reaction diffusion} for simplicity. Define $\Pi_{\mathcal V}$ as the orthogonal projection from $\mathbb{R}^{N_x\times N_x}$ onto the subspace $\mathcal V$. 

We know that $\xi$ is a critical point of $\mathcal J$ on space $\mathcal V$ iff 
\[ \Pi_{\mathcal V}\nabla\mathcal J(\xi) = 0. \]
By direct calculation, this is equivalent to
\[ \frac{\mathcal G_h^\dagger \xi}{h_t} + a\;\Pi_{\mathcal V} \mathcal L_h (U_0 + \xi) + b\;\Pi_{\mathcal V} f(U+\xi) = 0. \]
Writing the projection $\Pi_{\mathcal V} = \mathcal{G}_h^\dagger\mathcal{G}_h$, we obtain 
\[ \mathcal G^\dagger_h\left(\frac{\mathcal\xi}{h_t} + a\; \mathcal G_h \mathcal L_h (U_0 + \xi) + b \; \mathcal G_h f(U+\xi)\right) = 0. \]
Since the vector inside the above bracket belongs to $\mathcal V$, the above is equivalent to \eqref{reformulate root-finding}. 

We now prove the existence and uniqueness of the minimizer to the variational problem \eqref{proximal scheme reaction diffusion} under condition \eqref{condition on existence and uniqueness}, which implies the theorem. 

By a change of variable $\xi = Q_1 x$, where $Q_1$ is defined as in the spectral decomposition \eqref{spectral decomp   Gh} of $\mathcal G_h$, and $x\in\mathbb{R}^r$, \eqref{proximal scheme reaction diffusion} is equivalent to the following non-constrained optimization problem
\begin{equation}
  \min_{x\in\mathbb R^r} \left\{ \frac{x^\top \Lambda^{-1} x}{2h_t} + \frac{a}{2}x^\top Q_1^\top \mathcal L_h Q_1 x + a\;U^{0\top}\mathcal L_h Q_1 x + b\;W(U^0+Q_1x)^\top\boldsymbol{1} \right\}. 
\end{equation}
Denote $\widetilde{\mathcal J}(x)$ as the function in the above problem. 
Computing $\nabla\widetilde{\mathcal J}$ yields
\begin{align*}
  \nabla\widetilde{\mathcal J}(x) = & \frac{\Lambda^{-1}}{h_t} x + a \; Q_1^\top\mathcal L_h Q_1 x + a Q_1^\top \mathcal L_h U^{0}+ b \; Q_1^\top (V'(U^0+Q_1x) + \phi(U^0 + Q_1x)). 
\end{align*}
Then
\begin{align*}
  & (x-y, \; \nabla\widetilde{\mathcal{J}}(x) - \nabla\widetilde{\mathcal{J}}(y)) \\
= \; & \frac{1}{h_t} (x-y)^\top\Lambda(x-y) + a (x-y)^\top Q_1^\top \mathcal L_h Q_1(x-y) + b (Q_1(x-y))^\top (V'(U^0+Q_1x) - V'(U^0+Q_1y))\\
  \; & + b (Q_1(x-y))^\top (\phi(U^0+Q_1x) - \phi(U^0+Q_1y)) \\
\geq \; & (x-y)^\top \left(\frac{\Lambda}{h_t} + a Q_1^\top \mathcal L_h Q_1 \right)(x-y) + b K\|x-y\|^2 - b \mathrm{Lip}(\phi)\|x-y\|^2\\
\geq \; & \left(\lambda_{\min}\left(\frac{\Lambda^{-1}}{h_t} + a Q_1^\top \mathcal L_h Q_1 \right) + bK - b\mathrm{Lip}(\phi)\right) \|x-y\|^2.
\end{align*}
Then the condition \eqref{condition on existence and uniqueness} leads to 
$$\alpha = \lambda_{\min}\left(\frac{\Lambda^{-1}}{h_t} + a Q_1^\top \mathcal L_h Q_1 \right) + b K - b \mathrm{Lip}(\phi)> 0.$$ 
This shows the $\alpha$-strongly convexity of $\widetilde{\mathcal J}$, which leads to the existence and uniqueness of the minimizer to \eqref{proximal scheme reaction diffusion}, which accomplishes the proof.
\end{proof}

\subsection{Simplified conditions for specific reaction-diffusion equations} \label{append: discuss on condition on ht}
The condition \eqref{condition on existence and uniqueness} can be simplified for specific types of equations. We discuss two examples.
\begin{itemize}
    \item (Allen-Cahn equation with periodic boundary condition) In this case, $\mathcal G = \mathrm{Id}, \mathcal L = - \Delta.$ $f(x) = x^3 - x$. We set $\mathcal G_h = I_{N_x^2}$, and $\mathcal L_h = - \Delta_{h_x}^P=I_{N_x} \otimes (-\mathrm{Lap}_{h_x}^P) + (-\mathrm{Lap}_{h_x}^P) \otimes I_{N_x}$, where $\mathrm{Lap}_{h_x}^P$ is defined in \eqref{def: periodic discrete Laplace operator }.
    Then
    \begin{equation*}
         \lambda_{k}^P = \frac{4}{h_x^2}\sin^2\left(\frac{\pi k}{N_x}\right),\quad \textrm{with } 1 \leq k  \leq N_x,
    \end{equation*}
    are the eigenvalues of $-\mathrm{Lap}_{h_x}^P$. And the eigenvalues of $\frac{\Lambda^{-1}}{h_t} + a \; Q_1^\top \mathcal L_h Q_1 = \frac{I}{h_t} + a\mathcal L_h$ are $\lambda_{k,l} = \frac{1}{h_t} + a (\lambda^P_k + \lambda^P_l)$, with $1 \leq k, l \leq N_x$. Thus, $\lambda_{\min}(\frac{\Lambda^{-1}}{h_t} + a \; Q_1^\top \mathcal L_h Q_1) = \frac{1}{h_t}$. 
    
    Furthermore, we can decompose $f(x)=V'(x)+\phi(x)$, where
    \begin{equation*}
      V(x) = \begin{cases}
         \frac{1}{4}(x^2-1)^2,    
          & |x| > 1;\\
         0,  
         & |x| \leq 1.
      \end{cases} \quad \phi(x) = \begin{cases}
        0,
        & |x| > 1; \\
        x^3 - x, 
        & |x| \leq 1.
      \end{cases}
    \end{equation*}
    Then one can verify that $K=0$ and $\mathrm{Lip}(\phi) = 2$. In this case, condition \eqref{condition on existence and uniqueness} implies 
    \begin{equation}
      h_t < \frac{1}{\mathrm{Lip}(\phi) b} = \frac{1}{2b}. \nonumber 
    \end{equation}

    \item (Cahn-Hilliard equation with periodic boundary condition) In this case, $\mathcal G = -\Delta$, $\mathcal L= -\Delta$. $f(x)=x^3 - x.$ We set $\mathcal G_h = \mathcal L_h = I\otimes (-\mathrm{Lap}_{h_x}^P) + (-\mathrm{Lap}_{h_x}^P) \otimes I.$ We have 
    $$\lambda_{\min}\left(\frac{\Lambda^{-1}}{h_t} + a \; \Lambda\right) = \underset{1\leq k,l\leq N_x-1}{\min}\left\{\frac{1}{(\lambda_{k}^P+\lambda_l^P) h_t} + a (\lambda_k^P + \lambda_l^P) \right\} \geq 2\sqrt{\frac{a}{h_t}}.$$ Thus, a sufficient condition for \eqref{condition on existence and uniqueness} is
    \begin{equation}
     h_t < \frac{4 a^2 }{b^2\mathrm{Lip}(\phi)^2} = \frac{a^2}{b^2}. \nonumber 
    \end{equation}
\end{itemize}
It is worth mentioning that the conditions on $h_t$ for both Allen-Cahn and Cahn-Hilliard equations are independent of the spatial step size $h$, which makes it possible for our scheme to overcome the CFL condition required in the time-explicit scheme.

\section{Proofs of section \ref{sec: Lyapunov analysis PDHG flow time continue }}\label{appen: lemma and proof}

\subsection{Proofs of section \ref{subsec: lyapunov analysis for general root-finding }}\label{append: lyapunov analysis of general root finding }
To prove Lemma \ref{theorem : exponential decay}, we need the following Lemma \ref{lemma: pos def of mat B lambda} and Lemma \ref{lemma: pos def of H mu }.

\begin{lemma}\label{lemma: pos def of mat B lambda}
Suppose $\overline\lambda\geq \underline\lambda>0$. Assume $\mu>0$ satisfies $\frac{1}{\sqrt{\underline\lambda} } - \frac{1}{\sqrt{\overline\lambda}} < \frac{2}{\sqrt{\mu}}$. Define
\begin{equation*} 
    A = \max\left\{\left(1-\frac{\sqrt{\mu}}{\sqrt{\overline\lambda}}\right)^2, \left(1-\frac{\sqrt{\mu}}{\sqrt{\underline\lambda}}\right)^2\right\}, \quad \textrm{and } B =  \left(1+\frac{\sqrt{\mu}}{\sqrt{\overline\lambda}}\right)^2,   
\end{equation*}
then we always have $A<B$. Then for any $\lambda\in[\underline\lambda, \; \overline\lambda]$, and $\gamma, \epsilon > 0$ with $A < \gamma\epsilon < B$, the matrix $B_\lambda$
\begin{equation}
    B_\lambda = \left[\begin{array}{cc}
      \gamma \lambda & - \frac{1}{2}(\mu-(1-\gamma\epsilon)\lambda ) \\
        - \frac{1}{2}(\mu-(1-\gamma\epsilon)\lambda )    &  \mu\epsilon
  \end{array}\right]  \label{def: B lambda}
\end{equation}
is always positive definite.
\end{lemma}
\begin{proof}[Proof of Lemma \ref{lemma: pos def of mat B lambda}]
First, we have $\left|1-\frac{\sqrt{\mu}}{\sqrt{\overline\lambda}}\right|<1+\frac{\sqrt{\mu}}{\sqrt{\overline\lambda}}, 1-\frac{\sqrt{\mu}}{\sqrt{\underline\lambda}} < 1+\frac{\sqrt{\mu}}{\sqrt{\underline\lambda}};$ and the condition $\frac{1}{\sqrt{\underline\lambda}} - \frac{1}{\sqrt{\overline\lambda}} < \frac{2}{\sqrt{\mu}} $ yields $-(1-\frac{\sqrt{\mu}}{\sqrt{\underline\lambda}}) < 1  +  \frac{\sqrt{\mu}}{\sqrt{\overline\lambda}}.$ This yields
\begin{equation*}
    \max\left\{\left|1-\frac{\sqrt{\mu}}{\sqrt{\overline\lambda}}\right|, \left|1-\frac{\sqrt{\mu}}{\sqrt{\underline\lambda}}\right|\right\} < 1+\frac{\sqrt{\mu}}{\sqrt{\overline\lambda}}.
\end{equation*}
Taking squares on both sides of the above inequality gives $A<B$.

On the other hand, since $\gamma\lambda > 0$, and $\mu\epsilon>0$, we know $B_\lambda$ is positive definite if and only if $\textrm{det}(B_\lambda)>0$. In order to alleviate our notations, let us denote the quadratic polynomial $q_{\mu,\lambda}(\cdot)$ as
\begin{equation*}
     q_{\mu, \lambda}(x) = \lambda^2 x^2 - 2\lambda(\mu+\lambda)x + (\mu-\lambda)^2.
\end{equation*}
Then we know $\mathrm{det}(B_\lambda) = -\frac{1}{4} q_{\mu, \lambda}(\gamma\epsilon)$.
  
Now, for fixed $\lambda\in[\underline\lambda, \overline\lambda ]$, the two roots of $q_{\mu, \lambda}(x)$ are $\left(1\pm\frac{\sqrt{\mu}}{\sqrt{\lambda}}\right)^2.$ Thus $q_{\mu, \lambda}(x)<0$ if
\[ x\in I_\lambda \triangleq \left( \left(1-\frac{\sqrt{\mu}}{\sqrt{\lambda}}\right)^2, \left(1+\frac{\sqrt{\mu}}{\sqrt{\lambda}}\right)^2 \right).\]
On the other hand, we have
\[ \sup_{\lambda\in[\underline\lambda, \overline\lambda]} \left\{ \left(1-\frac{\sqrt{\mu}}{\sqrt{\lambda}}\right)^2  \right\}  = \max\left\{ \left(1-\frac{\sqrt{\mu}}{\sqrt{\overline\lambda}}\right)^2, \left(1-\frac{\sqrt{\mu}}{\sqrt{\underline\lambda}}\right)^2 \right\} = A,  \]
and
\[ \inf_{\lambda\in [\underline\lambda, \overline\lambda]} \left\{ \left(1+\frac{\sqrt{\mu}}{\sqrt{\lambda}}\right) \right\} = \left(1+\frac{\sqrt{\mu}}{\sqrt{\overline\lambda}}\right)^2 = B.  \]
As a result, $\bigcap_{\lambda\in[\underline\lambda, \overline\lambda]} \; I_\lambda = (A, B)$. Thus, we have shown that for any $\lambda\in[\underline\lambda, \overline\lambda]$, and $A<\gamma\epsilon<B$, $q_{\mu, \lambda}(\gamma\epsilon)<0$. This directly leads to the assertion of the lemma.
\end{proof}

\begin{lemma}[Positive definiteness of $\boldsymbol{H}_\mu$]\label{lemma: pos def of H mu }
Consider the matrix $\boldsymbol{H}_\mu$,
\[ \boldsymbol{H}_\mu = \left[\begin{array}{cc}
  \gamma\Sigma    &    - \frac{1}{2} (\mu I - (1-\gamma\epsilon)\Sigma )  \\
   - \frac{1}{2} ( \mu  I  -(1-\gamma\epsilon)\Sigma )  &  \mu\epsilon I
\end{array}\right],\]
with $\Sigma$ symmetric and positive definite. Suppose $0<\underline\lambda \leq \overline\lambda$ are two positive numbers such that the spectrum $\lambda(\Sigma) \subset [\underline\lambda, \overline\lambda].$ We further assume that  $\frac{1}{\sqrt{\underline\lambda}} - \frac{1}{\sqrt{\overline\lambda}} < \frac{2}{\sqrt{\mu}}$. We adopt the notation $A, B$ in Lemma \ref{lemma: pos def of mat B lambda}, i.e.,
\begin{equation*} 
    A = \max\left\{\left(1-\frac{\sqrt{\mu}}{\sqrt{\overline\lambda}}\right)^2, \left(1-\frac{\sqrt{\mu}}{\sqrt{\underline\lambda}}\right)^2\right\}, \quad \textrm{and } B =  \left(1+\frac{\sqrt{\mu}}{\sqrt{\overline\lambda}}\right)^2.  
\end{equation*}
By Lemma \ref{lemma: pos def of mat B lambda}, we have $A < B$. We also assume that $\gamma, \epsilon>0$ satisfy $A < \gamma\epsilon < B.$

Define the function $\varphi_{\mu, \gamma, \epsilon}(\cdot)$ as
\begin{equation}
    \varphi_{\mu, \gamma, \epsilon}(z) = \frac{1}{2}(\gamma z + \mu \epsilon - \sqrt{(\gamma z - \mu \epsilon)^2 + (\mu - (1-\gamma\epsilon)z)^2} ). \label{def varphi}
\end{equation}
We denote $\beta = \underset{ \lambda \in [\underline\lambda,\; \overline\lambda ] }{\min} \; \{\varphi_{\mu, \gamma,\epsilon}(\lambda)\}$, then $\beta > 0$. And we have $\boldsymbol{H}_\mu\succeq \beta I$.
\end{lemma}
\begin{proof}[Proof of Lemma \ref{lemma: pos def of H mu }]
For any $\lambda\in[\underline
\lambda, \overline\lambda]$, consider the matrix $B_\lambda$ as defined in \eqref{def: B lambda}, i.e.,
\begin{equation*}
    B_\lambda = \left[\begin{array}{cc}
      \gamma \lambda & - \frac{1}{2}(\mu-(1-\gamma\epsilon)\lambda ) \\
       - \frac{1}{2}(\mu-(1-\gamma\epsilon)\lambda )    &  \mu\epsilon
  \end{array}\right]\,.
\end{equation*}
By Lemma \ref{lemma: pos def of mat B lambda}, we know $B_\lambda$ is positive definite. By a directly calculation, the eigenvalues of $B_\lambda$ are given by (we assume $\lambda_1(B_\lambda) \geq \lambda_2(B_\lambda)  $),
\begin{align}
    \lambda_{1, 2}(B_\lambda) 
    =\frac{\gamma\lambda + \mu\epsilon \pm \sqrt{(\gamma\lambda - \mu\epsilon)^2 + (\mu - (1-\gamma\epsilon)\lambda)^2}}{2}. \label{def: B lambda }
\end{align}
Thus $\lambda_2(B_\lambda)=\varphi_{\mu, \gamma, \epsilon}(\lambda)$. Since $B_\lambda$ is positive definite, $\lambda_2(B_\lambda) = \varphi_{\mu,\gamma,\epsilon}(\lambda)>0$.

As a result, $\varphi_{\mu, \gamma, \epsilon}(\lambda)>0$ for $\lambda\in[\underline\lambda, \overline\lambda]$. Since $\varphi_{\mu, \gamma, \epsilon}(\cdot)$ is continuous on the compact set $[\underline\lambda, \overline\lambda]$, we know the infimum value $\beta > 0$. At the same time, it not hard to verify that $B_{\lambda}\succ \beta I$ for any $\lambda\in [\underline\lambda, \overline\lambda].$

To estimate $\boldsymbol{H}_\mu$ from below, let us denote $\lambda(\Sigma) = \{\lambda_1,\lambda_2, \dots ,\lambda_N\}$ with $\lambda_1\geq \lambda_2 \geq \dots \geq \lambda_N>0$ as the eigenvalues of matrix $\Sigma$. Since $\boldsymbol{H}_\mu$ is symmetric, $\boldsymbol{H}_\mu$ is similar to the following block diagonal matrix via an orthogonal transform
\begin{equation*}
  \left[\begin{array}{cccc}
     B_{\lambda_1}  &  &  &     \\
            & B_{\lambda_2}  & &  \\
            &  & \ddots &  \\
            &  &    &  B_{\lambda_N}
  \end{array}\right],
\end{equation*}
with each $B_{\lambda_j}$ defined as in \eqref{def: B lambda}. Since each $\lambda_j\in\lambda(\Sigma)\subset[\underline\lambda, \overline\lambda]$, the above argument applies to every $B_{\lambda_j}$, i.e., $B_{\lambda_j}\succ \beta I$ for any $1\leq j\leq N$. This leads to $\boldsymbol{H}_\mu\succ\beta I$. 
\end{proof}

We are ready to prove Lemma \ref{theorem : exponential decay}.
\begin{proof}[Proof of Lemma \ref{theorem : exponential decay}]
We denote
\begin{equation*}
  \Sigma = D \widehat F(U_t) D \widehat F(U_t)^\top = (I + D\eta(U_t)) (I + D\eta(U_t)^\top),
\end{equation*}
and compute
\begin{align}
\label{lyapunov time continuous}
\begin{split}
  \frac{d}{dt} \mathcal{I}_\mu  (U_t, Q_t) &= \widehat F(U)^\top D\widehat F(U) \dot U + \mu \; Q^\top \dot Q \\
  &=  - \widehat F(U)^\top D\widehat F(U) D\widehat F(U)^\top (Q + \gamma \dot Q) + \mu \; Q^\top (-\epsilon Q + \widehat F(U)) \\
  &=  - \widehat F(U)^\top \Sigma (Q + \gamma (-\epsilon Q + \widehat F(U))) - \mu\epsilon \|Q\|^2 + \mu Q^\top \widehat F(U)\\
  &=  \widehat F(U)^\top (\mu I-(1-\gamma \epsilon)\Sigma) Q - \gamma \widehat F(U)^\top \Sigma \widehat F(U) - \mu  \epsilon \|Q\|^2 \\
  &=  - [\widehat F(U)^\top, Q^\top ]  \underbrace{\left[\begin{array}{cc}
      \gamma\Sigma    &   -  \frac{1}{2} ( \mu I -(1-\gamma\epsilon)\Sigma )  \\
     -  \frac{1}{2} ( \mu I -(1-\gamma\epsilon)\Sigma )  &  \mu \epsilon I
  \end{array}\right]}_{\textrm{denote as } \boldsymbol{H}_{\mu}} \left[\begin{array}{c}
       \widehat F(U) \\
       Q
  \end{array}\right] \\
  & = 
    -[\widehat F(U)^\top, Q^\top] \; \boldsymbol{H}_\mu \; [\widehat F(U)^\top, Q^\top]^\top.
\end{split}
\end{align}
 We denote $\sigma_1(U_t)\geq\dots\geq\sigma_N(U_t)$ as the singular values of the Jacobian matrix $D\widehat F(U_t).$ It is not hard to verify that the spectrum of $\Sigma$ 
 \[ \lambda(\Sigma) = \{\sigma_1^2(U_t), \dots, \sigma^2_N(U_t)\}. \]
 According to definition \eqref{def: low bdd sigma} and \eqref{def: up bdd sigma}, we have 
 \[ \lambda(\Sigma)\subset [\underline\sigma^2, \; \overline\sigma^2]. \]
 Now we apply Lemma \ref{lemma: pos def of H mu } with $\underline\lambda = \underline\sigma^2$, $\overline\lambda = \overline\sigma^2$. We prove that $\boldsymbol{H}_\mu \succ \beta I$ for any $U_t\in\mathbb{R}^N$. As a result, we obtain the following inequality:
 \[ \frac{d}{dt}\mathcal I_\mu(U_t, Q_t) = -[\widehat F(U)^\top, Q^\top] \; \boldsymbol{H}_\mu \; [\widehat F(U)^\top, Q^\top]^\top \leq - \beta(\|\widehat F(U_t)\|^2 + \| Q_t \|^2).\]
 Furthermore, one has
 \[ {\max\{1, \mu\}} (\|\widehat F(U)\|^2 + \| Q \|^2) \geq \|\widehat F(U)\|^2 + \mu \|Q\|^2 , \]
 which yields
 \[ \|\widehat F(U)\|^2 + \|Q\|^2 \geq \frac{2}{\max\{1, \mu\}} \; \mathcal I_\mu(U, Q) . \]
 This finally leads to
 \[ \frac{d}{dt} \mathcal I_\mu (U_t, Q_t) \leq -\frac{2\;\beta}{\max\{1, \mu\}} \; \mathcal I_\mu(U_t, Q_t). \]
 And the Gr\"{o}nwall's inequality gives
 \[ \mathcal I_\mu (U_t, Q_t) \leq \exp\left( -\frac{2\;\beta}{\max\{1, \mu\}}t  \right)  \mathcal I_\mu (U_0, Q_0). \]
\end{proof}

We now prove Theorem \ref{main coro}.
\begin{proof}[Proof of Theorem \ref{main coro}]
Let us pick the hyperparameter $\mu = \underline \sigma^2$, one can verify that $\mu$ satisfies \eqref{condition on low sigma up sigma}. Furthermore, $\sqrt{\gamma\epsilon} = 1-\delta$. Since $|\delta|<\frac{1}{ \kappa }$, $1-\frac{1}{\kappa}< \sqrt{\gamma\epsilon} < 1+\frac{1}{\kappa} $. This verifies that 
$\sqrt{\gamma\epsilon}$ satisfies \eqref{condition on gamma epsilon}. 
Now Theorem \ref{theorem : exponential decay} guarantees that $\varphi_{\mu,\gamma,\epsilon}>0$ on $ [ \underline\sigma^2, \overline{\sigma}^2 ]$. For $z\in[\underline\sigma^2, \overline{\sigma}^2] $, we further calculate
\begin{align} 
 \varphi_{\mu, \gamma, \epsilon}(z) = & \frac{1}{2}(\gamma z + \mu \epsilon - \sqrt{(\gamma z + \mu \epsilon)^2 - (4\gamma\epsilon \mu z - (\mu - (1-\gamma\epsilon)z)^2)} )   \nonumber  \\
 = &\frac12 \frac{4\gamma\epsilon \mu z - (\mu - (1-\gamma\epsilon)z)^2}{\gamma z + \mu \epsilon  + \sqrt{(\gamma z + \mu \epsilon)^2 - (4\gamma\epsilon \mu z - (\mu - (1-\gamma\epsilon)z)^2)}  }\nonumber\\
 \geq & \frac{4\gamma\epsilon \mu z - (\mu - (1-\gamma\epsilon)z)^2}{ {4}(\gamma z + \mu \epsilon) }\nonumber\\
 = & \frac{-(1-\gamma\epsilon)^2z^2+2\mu(1+\gamma\epsilon)z-\mu^2}{ {4}(\gamma z + \mu \epsilon)} \nonumber\\
 = & \frac{-((1+\gamma\epsilon)z - \mu)^2 + 4\gamma\epsilon z^2 }{4(\gamma z + \mu \epsilon)} \nonumber \\
 = & \frac{(2\sqrt{\gamma\epsilon}z - (1+\gamma\epsilon) z + \mu)(2\sqrt{\gamma\epsilon}z + (1+\gamma\epsilon) z - \mu) }{4(\gamma z + \mu \epsilon)} \nonumber \\
 =& \frac{(\sqrt{\mu}-|1-\sqrt{\gamma\epsilon}|\sqrt{z})(\sqrt{\mu}+|1-\sqrt{\gamma\epsilon}|\sqrt{z})((1+\sqrt{\gamma\epsilon})^2z-\mu)}{{4}(\gamma z + \mu \epsilon)} \nonumber \\
 \overset{1-\sqrt{\gamma\epsilon} = \delta,  ~ z\leq \overline{\sigma}^2}{\geq} & \frac{(\sqrt{\mu} - |\delta|\sqrt{z})(\sqrt{\mu} + |\delta|\sqrt{z})((2-\delta)^2{z} - {\mu})}{{4}(\gamma  \overline{\sigma}^2  +  \mu\epsilon)}. \label{low bound of varphi }
\end{align} 
Since we have set 
$$\gamma = \frac{1 - \delta}{\kappa}, \;  \epsilon = (1-\delta)\kappa,  \; \mu = \underline{\sigma}^2.$$
Substituting them into \eqref{low bound of varphi } yields
\begin{align*}
  \varphi_{\mu, \gamma, \epsilon}( z )& \geq \frac{(\underline\sigma - |\delta|\sqrt{z})(|\delta|\sqrt{z} + \underline\sigma)((2-\delta)^2{z} - {\underline\sigma^2})}{ 8(1-\delta) \; \underline\sigma \; \overline{\sigma} }\\
  &= \frac{1}{8(1-\delta)} \left( 1-|\delta|\frac{\sqrt{z}}{\underline \sigma} \right) \left(|\delta| \frac{\sqrt{z}}{\overline{\sigma}} + \frac{\underline\sigma}{\overline{\sigma}}\right) ((2-\delta)^2z-\underline\sigma^2)\\
  &\geq\frac{1}{8(1-\delta)}(1-\kappa|\delta|) \left(\frac{ |\delta|+1}{\kappa}\right)(1-\delta)(3-\delta)\underline\sigma^2\\
  & \geq \frac{1}{8\kappa} (1-\kappa|\delta|) (3-\delta)\underline\sigma^2 .
\end{align*}
If we denote $\beta = \underset{z\in [\underline\sigma^2, \overline\sigma^2]}{\min}  \{\varphi_{ \mu, \gamma, \epsilon }(z)\}$, then we have
\[ \frac{\beta}{\max\{1, \mu\}} \geq \frac{(1-\kappa|\delta|) (3-\delta)}{8\kappa} \frac{\underline\sigma^2}{\max\{1,\underline\sigma^2\}} = \frac{1}{8}(1-\kappa|\delta|) (3-\delta)\frac{\min\{\underline\sigma^2, 1\}}{\kappa}. \]
Thus, the result of Theorem \ref{theorem : exponential decay} yields
\[ \mathcal{I}_\mu(U_t, Q_t) \leq \exp\left( -\frac{1}{4}(1-\kappa|\delta|) (3-\delta)\frac{\min\{\underline\sigma^2, 1\}}{\kappa} \; t \right) \mathcal{I}_\mu(U_0, Q_0). \]
Taking square root on both sides of the above inequality and using the fact that 
\[ \|\widehat F(U_t)\| \leq \sqrt{\mathcal I_\mu(U_t, Q_t)}  ,\]
we obtain
\[ \|\widehat F(U_t)\| \leq \exp\left( -\frac{1}{8}(1-\kappa|\delta|) (3-\delta)\frac{\min\{\underline\sigma^2, 1\}}{\kappa} \; t \right) \sqrt{\mathcal{I}_\mu(U_0, Q_0)}. \]
This implies our theorem.
\end{proof}

\begin{theorem}[Exponential decay of $\mathcal I_\mu(U_t, Q_t)$]\label{thm : exponential decay version 2}
   Assume that $(U_t, Q_t)$ solves \eqref{PDHG cont time} with arbitrary initial position $(U_0, Q_0)$. 
  Then we have the exponential decay of the Lyapunov function $\mathcal I_\mu(U_t, Q_t)$, i.e.,
  \begin{equation*}
    \mathcal I_\mu (U_t, Q_t) \leq \exp \left( {- 2 \lambda  t   } \right) \; \mathcal I_\mu (U_0, Q_0)\,,
  \end{equation*}
where 
\begin{align*}
    \lambda = \min\{&\epsilon-\frac{1}{2}|(1-\gamma \epsilon)\sigma^2_1/\mu -1|,\epsilon-\frac{1}{2}|(1-\gamma \epsilon)\sigma^2_n/\mu -1|, \\
    &\gamma \sigma_1^2 - \frac{1}{2}|(1-\gamma\epsilon) \sigma_1^2 - \mu|\},\gamma \sigma_n^2 - \frac{1}{2}|(1-\gamma\epsilon) \sigma_n^2 - \mu|\}\,.
\end{align*}
In particular, when $\gamma\epsilon=1$, $\mu=0$, and 
$$
\gamma = \frac{-\frac{1}{2}\frac{\sigma_1^2-\sigma_n^2}{\sigma_1^2 + \sigma_n^2}+ \sqrt{\frac{1}{4}\left( \frac{\sigma_1^2-\sigma_n^2}{\sigma_1^2 + \sigma_n^2}\right)^2 + 4\sigma_n^2}}{2\sigma_n^2},
$$
we have $\lambda = 2\sigma_n^2 \frac{\sigma_1^2+\sigma_n^2}{\sigma_1^2 - \sigma_n^2}-\frac{1}{2}\sigma_n^4 \left( \frac{\sigma_1^2+\sigma_n^2}{\sigma_1^2 - \sigma_n^2}\right)^3 + \mathcal{O}(\sigma_n^6)$. 
\end{theorem}
\begin{proof}[Proof of Theorem \ref{thm : exponential decay version 2}]
We would like to find $\lambda$ such that 
$$
\frac{d \mathcal{ I}}{dt} + 2\lambda \mathcal{ I} \leq 0\,.
$$
Then by Gronwall's inequality, we obtain exponential convergence. We have 
$$
\frac{d \mathcal{ I}}{dt} + 2\lambda \mathcal{ I} =  [\widehat F(U)^\top, Q^\top ]  \left[\begin{array}{cc}
      \lambda I- \gamma\Sigma    &     \frac{1}{2} ( \mu I -(1-\gamma\epsilon)\Sigma )  \\
       \frac{1}{2} ( \mu I -(1-\gamma\epsilon)\Sigma )  & \lambda \mu I - \mu \epsilon I
  \end{array}\right] \left[\begin{array}{c}
       \widehat F(U) \\
       Q
  \end{array}\right] \,.
$$
Using Lemma A.1 from \cite{zuo2023primal}, it suffices to have 
\begin{subequations}\label{eq:lam_constraints}
\begin{align}
    \lambda-\gamma \sigma_i^2 + \frac{1}{2}|(1-\gamma\epsilon) \sigma_i^2 - \mu| &\leq 0 \,,\\
    \lambda\mu - \mu \epsilon + \frac{1}{2}|(1-\gamma\epsilon) \sigma_i^2 - \mu| &\leq 0 \,,
\end{align}
\end{subequations}
for all $\overline{ \sigma}^2 = \sigma_1^2 \geq \sigma_2^2 \geq \cdots \geq \sigma_n^2 = \underline\sigma^2 $. Let us define $g_1(\sigma) = \epsilon-\frac{1}{2}|(1-\gamma \epsilon)\sigma^2/\mu -1|$, and $g_2(\sigma) = \gamma \sigma^2 - \frac{1}{2}|(1-\gamma\epsilon) \sigma^2 - \mu|$. Then \eqref{eq:lam_constraints} implies that 
$$\lambda \leq \min_{i=1,2}\min_{\sigma_n\leq \sigma\leq \sigma_1} g_i(\sigma) \,.$$
Since $g_i(\sigma)$'s are piece-wise linear and have only one kink, it is easy to check that 
$$\min_{\sigma_n\leq \sigma\leq \sigma_1} g_i(\sigma) = \min \{ g_i(\sigma_1),g_i(\sigma_n)\}\,.$$
This proves the first part of our lemma. When taking $\mu = \frac{1}{2}(1-\gamma\epsilon)(\sigma_1^2 + \sigma_n^2)$, one can show by a straightforward calculation that $g_1(\sigma_n) = g_1(\sigma_1)$. This also implies that $g_2(\sigma_1) \geq g_2(\sigma_n)$. Therefore, to make $\lambda$ large, we would like to equate $g_1(\sigma_n)$ and $g_2(\sigma_n)$. This yields 
\begin{align}
     \epsilon-\frac{1}{2}\frac{\sigma_1^2-\sigma_n^2}{\sigma_1^2 + \sigma_n^2} &= \gamma \sigma_n^2 - \frac{1}{4}(1-\gamma\epsilon)(\sigma_1^2 - \sigma_n^2) \nonumber \\
     \epsilon &= \frac{\gamma \sigma_n^2 + \frac{1}{2}\frac{\sigma_1^2-\sigma_n^2}{\sigma_1^2 + \sigma_n^2}-\frac{1}{4}(\sigma_1^2 - \sigma_n^2)}{1-\frac{1}{4}\gamma (\sigma_1^2 - \sigma_n^2)}\,.\label{eq:epsilon}
\end{align}
In the special case of $\gamma \epsilon = 1$, we obtain 
\begin{align}
   1= \gamma \epsilon = \frac{\gamma^2 \sigma_n^2 + \frac{1}{2}\gamma\frac{\sigma_1^2-\sigma_n^2}{\sigma_1^2 + \sigma_n^2}-\frac{1}{4}\gamma(\sigma_1^2 - \sigma_n^2)}{1-\frac{1}{4}\gamma (\sigma_1^2 - \sigma_n^2)}\,. 
\end{align}
We can solve for $\gamma$ and we get (keeping the positive root)
$$
\gamma = \frac{-\frac{1}{2}\frac{\sigma_1^2-\sigma_n^2}{\sigma_1^2 + \sigma_n^2}+ \sqrt{\frac{1}{4}\left( \frac{\sigma_1^2-\sigma_n^2}{\sigma_1^2 + \sigma_n^2}\right)^2 + 4\sigma_n^2}}{2\sigma_n^2}\,.
$$
Consequently, the convergence rate is 
\begin{align}
    \lambda = \gamma \sigma_n^2 &= -\frac{1}{4}\frac{\sigma_1^2-\sigma_n^2}{\sigma_1^2 + \sigma_n^2}+\frac{1}{2}\sqrt{\frac{1}{4}\left( \frac{\sigma_1^2-\sigma_n^2}{\sigma_1^2 + \sigma_n^2}\right)^2 + 4\sigma_n^2}\nonumber \\
    & = 2\sigma_n^2 \frac{\sigma_1^2+\sigma_n^2}{\sigma_1^2 - \sigma_n^2}-\frac{1}{2}\sigma_n^4 \left( \frac{\sigma_1^2+\sigma_n^2}{\sigma_1^2 - \sigma_n^2}\right)^3 + \mathcal{O}(\sigma_n^6)\,.
\end{align}
\end{proof}

\subsection{Proofs of section \ref{subsec: convergence result for time continuous case }}\label{append: lyapunov analysis of specific root-finding }

To prove Lemma \ref{lemm: est singular values DF }, we need the following Lemma \ref{lemm: est singularvalues } and Lemma \ref{lemm: posdef I+GL }.

\begin{lemma}\label{lemm: est singularvalues }
  Suppose $A$ is an $nm \times nm$ matrix defined as
  \[ A = \left[\begin{array}{ccccc}
                 A_1  &   &  &  & \\
                -I & A_2 &  & & \\
                   & -I & A_3 & & \\
                  &    & \ddots & \ddots &  \\
                 &  &   &   -I & A_n
               \end{array}
  \right], 
  \]
    where each $A_k$ is an$m\times m$ matrix with $\sigma_{\min}(A_k) \geq \underline{\alpha} > 0$ and $\sigma_{\max}(A_k ) \leq \overline{\alpha}$, i.e., $\|A_k v\|\geq \underline\alpha \|v\|$, $\|A_kv\|\leq \overline{\alpha}\|v\|$ for any $v\in\mathbb{R}^m$. Then $\|A^{-1}\|\leq \sum_{k=1}^n \underline{\alpha}^{- k} $, and $\|A\|\leq \overline{\alpha}+1$, i.e., $\sigma_{\min}(A) \geq \frac{1}{\sum_{k=1}^n \underline{\alpha}^{-k}}, $ and $\sigma_{\max}(A)\leq \overline{\alpha}+1.$ 
\end{lemma}
\begin{proof}[Proof of Lemma \ref{lemm: est singularvalues }]
  By a direct calculation, we have
  \[ A^{-1} = \left[ \begin{array}{ccccc}
      A_1^{-1} &  &  & & \\
      (A_1A_2)^{-1} & A_2^{-1} &  &  & \\
      (A_1A_2A_3)^{-1} & (A_2A_3)^{-1} & A_3^{-1} & & \\
      \vdots & \vdots &\vdots& \ddots &  \\
      (A_1A_2\dots A_n)^{-1}  &  (A_2\dots A_n)^{-1} & (A_3\dots A_n)^{-1} & \dots & A_n^{-1} 
  \end{array} \right]\,.
  \]
  Thus we can write $A^{-1}$ as
  \begin{align*} 
    A^{-1} =& \left[\begin{array}{cccc}
       {A}_{11}  &  & & O \\
       &  A_{22} & & \\
       &   &  \ddots & \\
       O & & & A_{nn}
    \end{array}\right] + \left[\begin{array}{cccc}
      O  &  &  &    \\
      A_{21} & \ddots & &   \\
       &  \ddots  &   \ddots & \\ 
      O &  &   A_{n, n-1} & O 
    \end{array}\right] + \dots + \left[ \begin{array}{cccc}
      O   &  &  &  \\
      \vdots &  \ddots  & & \\
      O  &     &  \ddots & \\ 
      A_{n1} & O &  \dots & O 
    \end{array} \right]\\
    \overset{\textrm{denote as}}{=}& J_1 + J_2 + \dots + J_n.  
  \end{align*}
  Here, each $J_k$ ($1\leq k\leq n$) is an $nm\times nm$ block-(sub)diagonal matrix whose $k$-th subdiagonal is $$\mathrm{diag}(A_{k, 1}, A_{k+1, 2}, \dots, A_{n, n-k+1}).$$ And each $A_{ij}$ is defined as
  \[ A_{ij} = (A_jA_{j+1}\dots A_i)^{-1}, \quad \textrm{if } i\geq j. \]
  Then one can bound $\|A^{-1}\|$ as
  \[ \|A^{-1}\| \leq \sum_{k=1}^{n} \|J_k\|. \]
  To bound each $\|J_k\|$ from above, consider any $v = [v_1^\top, v_2^\top, \dots, v_n^\top]^\top\in\mathbb{R}^{nm}$ with each $v_j\in\mathbb{R}^m$, we have
  \[ \|J_k v\|^2 = \sum_{j=k}^{n} \|A_{j, j-k+1} v_j\|^2 = \sum_{j=k}^n \|(A_{j-k+1}\dots A_j)^{-1} v_j\|^2 \leq \underline{\alpha}^{-2k} \sum_{j=k}^n\|v_j\|^2 \leq \underline{\alpha}^{-2k}  \|v\|^2 . \]
  This yields $\|J_k v\|\leq \underline{\alpha}^{-k} \|v\|$ which further gives $\|J_k\|\leq \underline\alpha^{-k}$. Thus, we have proved
  $\|A^{-1}\| \leq \sum_{k=1}^n\underline{\alpha}^{-k}$, which directly leads to the result $\sigma_{min} ( A )  \geq \frac{1}{\sum_{k = 1 }^n \underline{ \alpha }^{-k} } $.

  On the other hand, we write $A$ as
  \[ A = \textrm{diag}(A_1, \dots, A_n) - J \otimes I, \]
  where $J$ is an $n\times n$ matrix defined as 
  \begin{equation}
   J = \left[\begin{array}{cccc}
      0  &    &   &  \\
      1  & \ddots  &  &  \\
         &    \ddots  & \ddots & \\
         &            &   1  & 0
  \end{array}\right],  \label{def: subdiag matrix J }
  \end{equation}
  and $I$ is an $n\times n$ identity matrix. Then we have
  \[ \|A\|\leq\|\mathrm{diag}(A_1, \dots, A_n)\| + \|J\otimes I\| \leq \overline{\alpha} + 1.\]
\end{proof}

\begin{lemma}\label{lemm: posdef I+GL }
  Suppose $G,L$ are self-adjoint, nonnegative definite matrices. Assume $GL = LG$. Then $I+GL$ (or $I+LG$) is orthogonally equivalent to $I+\Lambda_G\Lambda_L$, where $\Lambda_G, \Lambda_L$ are the diagonal matrices equivalent to $G, L$. Furthermore, $\sigma_{\min}(I+GL) = \sigma_{\min}(I+LG) \geq 1+\lambda_{\min}(G)\lambda_{\min}(L)\geq 1.$
\end{lemma}
\begin{proof}[Proof of Lemma \ref{lemm: posdef I+GL }]
  Since $G, L$ commutes, they can be diagonalized simultaneously, i.e., there exists an orthogonal matrix $Q$, s.t. $G =  Q \Lambda_G Q^\top, $ and $L = Q \Lambda_L Q^\top$, where $\Lambda_G, \Lambda_L\succeq O$ are diagonal matrices. Then $I+GL = I + LG = Q(I+\Lambda_G\Lambda_L)Q^\top$. And thus $\sigma_{\min}(I+GL) = \sigma_{\min}(I+\Lambda_G\Lambda_L) \geq 1+\lambda_{\min}(G)\lambda_{\min}(L)\geq 1.$
\end{proof}

We now prove Lemma \ref{lemm: est singular values DF }.
\begin{proof}[Proof of Lemma \ref{lemm: est singular values DF }]
We first recall
\begin{align*}
 \underline \sigma = & \inf_{U\in \mathbb{R}^{N_x^2}} \{ \sigma_{\min}(D\widehat F(U)) \} = \inf_{U\in \mathbb{R}^{N_x^2}} \{\sigma_{\min}(\mathscr{M}^{-1}DF(U))\},\\
 \overline{\sigma} = & \sup_{U\in \mathbb{R}^{N_x^2}} \{ \sigma_{\max}(D\widehat F(U)) \} = \sup_{U\in \mathbb{R}^{N_x^2}} \{\sigma_{\max}(\mathscr{M}^{-1}DF(U))\},
\end{align*}
where we denote $F(U) = \mathscr{D} U + h_t \mathscr{G}_h(a\mathscr{L}_h U +bf(U)).$

We have
\begin{equation} 
\sigma_{\min}(\mathscr{M}^{-1}DF(U)) = \frac{1}{\sigma_{\max}(DF(U)^{-1}\mathscr{M})} \geq \frac{1}{\|DF(U)^{-1}\|     \|\mathscr{M}\|_2 } = \frac{ \sigma_{\min}(DF(U)) }{ \| \mathscr{M}  \| }  . \label{sigma min D hat F} 
\end{equation}
And
\begin{equation}
\sigma_{\max}(\mathscr{M}^{-1}DF(U)) \leq  \sigma_{\max}(DF(U))\|  \mathscr{M}^{-1}  \|. \label{sigma max D hat F}
\end{equation}

Now we estimate the singular values of $DF(U)$, since 
\[ DF(U) = \left[\begin{array}{ccccc}
    X_1  &   &  &  & \\
    -I & X_2 &  & & \\
       & -I & X_3 & & \\
       &    & \ddots & \ddots &  \\
       &  &   &   -I & X_{N_t}
\end{array}\right],
\]
where each $X_i = I + a h_t \mathcal G_h \mathcal L_h + b h_t \mathcal G_h \mathrm{diag}(f'(U^i))$. (Here we denote $U = (U^{1 \top},\dots,U^{N_t \top})^\top$.)  

Then for each $X_i$, we have 
\begin{align*}
  \sigma_{\min}(X_i)&\geq \sigma_{\min}(I + ah_t\mathcal G_h \mathcal L_h) - \sigma_{\max}(b h_t \mathcal G_h \mathrm{diag}(f'(U^i)))\\
  & \geq \sigma_{\min}(I + ah_t\mathcal G_h \mathcal L_h) - h_t|b|\|\mathcal G_h\|\|\mathrm{diag}(f'(U^i))\|.
\end{align*}
By Lemma \ref{lemm: posdef I+GL }, the first term above is no less than $1+ah_t\lambda_{\min}(\mathcal G_h)\lambda_{\min}(\mathcal L_h)\geq 1$. It is not hard to verify that $\|\mathcal G_h\|=\lambda_{\max}(\mathcal G_h)$,  $\|\mathrm{diag}(f'(U^i))\|\leq \mathrm{Lip}(f)$. This leads to
$$\sigma_{\min}(X_i) \geq 1-h_t|b|\lambda_{\max}(\mathcal G_h)\mathrm{Lip}(f). $$
We denote $\underline{\alpha} = 1-h_t|b|\lambda_{\max}(\mathcal G_h)\mathrm{Lip}(f)$. Then $\underline\alpha>0$, and is independent of $U$. 

On the other hand, one can also verify that
\[ \sigma_{\max}(X_i) = \|X_i\|\leq \|I+ah_t\mathcal G_h\mathcal L_h\| + h_t|b|\|\mathcal G_h\|\mathrm{Lip}(f), \]
by denoting $\overline{\alpha} = \|I+ah_t\mathcal G_h\mathcal L_h\| + h_t|b|\|\mathcal G_h\|\mathrm{Lip}(f)$, we know $\overline{ \alpha}$ is also independent of $U$.

We now apply Lemma \ref{lemm: est singularvalues } to $DF(U)$ with $\sigma_{\min}(X_i)\geq \underline{ \alpha }$ and $\sigma_{\max}(X_i)\leq  \overline{\alpha}$. Together with \eqref{sigma min D hat F} and \eqref{sigma max D hat F}, we have
\[ \sigma_{\min}(D\widehat F(U)) \geq \frac{1}{\left(\sum_{k=1}^{N_t} \underline \alpha^{-k}\right) \|\mathscr{M}\| }, \quad \sigma_{\max}(D\widehat F(U)) \leq (1+\overline{\alpha}) \|\mathscr{M}^{-1}\|. \]
Since $\underline \alpha$, $\overline{\alpha}$, $\|\mathscr{M}\|$ and $\|\mathscr{M}^{-1}\|$ are all independent of $U$, we are done.
\end{proof}

To prove Lemma \ref{lemm: more sophisticate est on singular values of DF }, we need the following Lemma \ref{lemm:  est on inv_M G_h }.
\begin{lemma}\label{lemm:  est on inv_M G_h }
  Suppose we keep all the assumptions from Lemma \ref{lemm: more sophisticate est on singular values of DF }. Let $\mathscr{G}_h$ be defined as in \eqref{def; caligraph Gh, Lh }, and $\mathscr{M}$ be defined as in \eqref{def: scr M }. Then 
  \begin{equation*}
    \|\mathscr{M}^{-1}\mathscr{G}_h\|\leq N_t\left(\max_{1\leq k\leq N_x^2} \left\{ \frac{\lambda_k(\mathcal G_h)}{1 + h_t(a\lambda_k(\mathcal{G}_h)\lambda_k(\mathcal{L}_h) + bc\lambda_k(\mathcal{G}_h)) }\right\}\right).
  \end{equation*}
\end{lemma}
\begin{proof}[Proof of Lemma \ref{lemm:  est on inv_M G_h }]
  Recall that we have
  \[ \mathscr M = \left[\begin{array}{ccccc}
    X & & & &   \\
   -I & X & & &  \\
      & -I& X & & \\
      &   & \ddots & \ddots & \\
      &   &        &     -I & X
  \end{array}\right], \quad  X = I + a h_t \mathcal{G}_h\mathcal{L}_h + b h_t 
     \mathcal G_h J_f. \]  
  By Lemma \ref{lemm: posdef I+GL }, we have $X = Q(I + ah_t\Lambda_{\mathcal G_h}\Lambda_{\mathcal L_h}+bch_t\Lambda_{\mathcal G_h})Q^\top$, where we have also used that $\mathcal G_h, \mathcal L_h$ commute, and $J_f  =  c  I$. Here we write $\Lambda_{\mathcal{G}_h}, \Lambda_{\mathcal{L}_h}$ as the diagonal matrices which are orthogonally similar to $\mathcal G_h, \mathcal L_h$ w.r.t. orthogonal matrix $Q$. It is not hard to verify that 
  \begin{equation}  
    \|X^{-1}\|\leq \frac{1}{1+h_t(\lambda_{\min}(a\mathcal G_h\mathcal L_h + bc \mathcal G_h))} \leq 1\label{est low bdd sigmamin X}.
  \end{equation}
  Now one can compute
  \begin{align*}
     \mathscr{M}^{-1}\mathscr{G}_h & = \left[ \begin{array}{ccccc}
      X^{-1} &  &  & & \\
      X^{-2} & X^{-1} &  &  & \\
      X^{-3} & X^{-2} & X^{-1} & & \\
      \vdots & \vdots &\vdots& \ddots &  \\
      X^{-N_t}  &  X^{-(N_t - 1)} & X^{-(N_t-2)} & \dots & X^{-1} 
  \end{array} \right]  \left[\begin{array}{ccccc}
     \mathcal G_h  & & & & \\
       & \mathcal G_h & & &\\
       & & \mathcal G_h&&\\
       & & & \ddots &\\
       &&&&\mathcal G_h
  \end{array}\right] \\
  & = \underbrace{\left[ \begin{array}{ccccc}
      I &  &  & & \\
      X^{-1} & I &  &  & \\
      X^{-2} & X^{-1} &  I  & & \\
      \vdots & \vdots &\vdots& \ddots &  \\
      X^{-(N_t-1)}  &  X^{-(N_t - 2)} & X^{-(N_t-3)} & \dots & I  
  \end{array} \right]}_{\mathscr{N}}  \underbrace{\left[\begin{array}{ccccc}
     X^{-1} \mathcal G_h  & & & & \\
       & X^{-1} \mathcal G_h & & &\\
       & & X^{-1} \mathcal G_h&&\\
       & & & \ddots &\\
       &&&& X^{-1}\mathcal G_h
  \end{array}\right]}_{\widetilde{\mathscr{G}}_h}\\
  & \overset{\textrm{denote as}}{=}  {\mathscr{N}} \widetilde{\mathscr{G}}_h.
  \end{align*}
  Similar to the treatment in Lemma \ref{lemm: est singular values DF }, we estimate $\|\mathscr N\|$ by decomposing $\mathscr{N}$ as
  \begin{align*} 
    \mathscr{N} = & I\otimes X^{0} + J \otimes X^{-1} + J^2 \otimes X^{-2} + \dots + J^{N_t-1} \otimes X^{-(N_t-1)}\,,
  \end{align*}
  where we recall that $J$ is defined as in \eqref{def: subdiag matrix J }; And $X^0$ is treated as the identity matrix.
  
  Then we estimate $\|\mathscr{N}\|$ as
  \begin{equation*}
    \|\mathscr N \|\leq \left(\sum_{k=0}^{N_t-1} \|J^k\otimes (X^{-1})^k \|\right).
  \end{equation*}
  Since $\|A\otimes B\| = \|A\|\cdot \|B\|$ for any dimensions of square matrices $A,B$, using \eqref{est low bdd sigmamin X} and $\|J\|\leq 1$ yields
  \begin{equation}
      \|\mathscr{N}\|\leq \sum_{k =0 }^{N_t - 1 }\|(X^{-1})^{k}\| \leq \sum_{k =0 }^{N_t - 1 }\|X^{-1}\|^k \leq N_t.\label{est norm matrix N }
  \end{equation}
  
  On the other hand, we have 
  \begin{equation*}
    \widetilde{\mathscr{G}}_h = X^{-1}\mathcal G_h = Q((I + ah_t\Lambda_{\mathcal G_h}\Lambda_{\mathcal L_h}+bch_t\Lambda_{\mathcal G_h})^{-1}\Lambda_{\mathcal G_h})Q^\top. 
  \end{equation*}
  If we denote $\{\lambda_k(\mathcal G_h)\}$, $\{\lambda_k(\mathcal L_h)\}$ ($1\leq k \leq N_x^2$) as the corresponding eigenvalues of $\mathcal G_h, \mathcal L_h$ w.r.t. $Q$, we know
  \begin{equation}
    \|\widetilde{\mathscr{G}_h}\| = \max_{1\leq k\leq N_x^2} \left\{ \frac{\lambda_k(\mathcal G_h)}{1 + h_t(a\lambda_k(\mathcal{G}_h)\lambda_k(\mathcal{L}_h) + bc\lambda_k(\mathcal{G}_h))}  \right\} . \label{est norm matrix tilde Gh }
  \end{equation}
  Now combining \eqref{est norm matrix N } and \eqref{est norm matrix tilde Gh } and using $\|\mathscr{M}^{-1}\mathscr{G}_h\|\leq \|\mathscr{N}\|\|\widetilde{\mathscr G}_h\|$, we finish the proof.
\end{proof}

We now prove Lemma \ref{lemm: more sophisticate est on singular values of DF }.
\begin{proof}[Proof of Lemma \ref{lemm: more sophisticate est on singular values of DF }]
By Lemma \ref{lemm:  est on inv_M G_h } and the fact that $\|DR(\cdot)\|\leq \mathrm{Lip}(R)$, we have
\begin{equation*}
  \|D\eta(U)\| = \|b h_t \mathscr{M}^{-1} \mathscr{G}_h DR(U)\| \leq bh_t\cdot\|\mathscr{M}^{-1}\mathscr G_h\|\cdot\mathrm{Lip}(R) \leq b T \zeta_{a,b,c}(h_t)\mathrm{Lip}(R).
\end{equation*}
 Recall that
 \[ D\widehat F(U) = I + D\eta(U). \]
  Now for any $v\in\mathbb{R}^{N_x^2}$, we have
  \begin{equation}
       \|D\widehat F(U)v\| = \|v + D\eta(U)v\| \geq \|v\| - \|D\eta(U)\|\|v\|.  \geq (1-bT\mathrm{Lip}(R)\zeta_{a,b,c}(h_t))\|v\|. \label{est low bd D hat F }
  \end{equation}
  Since the right-hand side of \eqref{est low bd D hat F } is independent of $U$, this will lead to a lower bound on $\underline\sigma$, i.e. 
  \[ \underline \sigma \geq 1-bT \zeta_{a,b,c}(h_t) \mathrm{Lip}(R) . \]
By a similar argument, we have
  \[ \|D\widehat F(U) v\| \leq \|v\| + \|D\eta(U)\|\|v\|\leq (1+bT \zeta_{a,b,c}(h_t) \mathrm{Lip}(R))\|v\|. \]
  This will finally lead to
  \[ \overline \sigma \leq 1 + bT \zeta_{a,b,c}(h_t) \mathrm{Lip}(R). \] 
\end{proof}

\begin{lemma}[Sufficient condition on the unique solvability of $\widehat{F}(U) = 0$]\label{lemm: sufficient cond unique existence of root }
Suppose conditions \eqref{condition: a b >= 0 }, \eqref{condition: f Lip }, \eqref{condition: G,L>=0, symm, commut } and \eqref{condition Jf=cI c>=0 } hold. We pick $h_t$ and $T=N_t h_t$ ($N_t\in\mathbb{N}_+$) satisfying $bT\mathrm{Lip}(R) \zeta_{a,b,c}(h_t) < 1$. Then there exists a unique root of $\widehat{F}$.
\end{lemma}
\begin{proof}[Proof of Lemma \ref{lemm: sufficient cond unique existence of root }]
\eqref{condition on ht zeta less than theta } leads to 
  \begin{equation*}
     \max_{1\leq k\leq N_x^2} \left\{ \frac{\lambda_k(\mathcal G_h)}{1 + h_t(a\lambda_k(\mathcal{G}_h)\lambda_k(\mathcal{L}_h) + bc\lambda_k(\mathcal{G}_h))}  \right\} < \frac{1}{bT\mathrm{Lip}(R)},
  \end{equation*}
  which is equivalent to
  \begin{equation*}
     \min_{1\leq k\leq N_x^2, \lambda_k(\mathcal G_h)>0} \left\{ \frac{1}{\lambda_k(\mathcal G_h)} +h_t(a\lambda_k(\mathcal{L}_h) + bc) \right\}  >  bT\mathrm{Lip}(R).
  \end{equation*}
  Since $T\geq h_t$, the right-hand side of the above inequality is larger than or equal to $bh_t\mathrm{Lip}(R)$. Thus the above inequality yields
  \begin{equation}
     \min_{1\leq k\leq N_x^2, \lambda_k(\mathcal G_h)>0} \left\{ \frac{1}{\lambda_k(\mathcal G_h)h_t} + a\lambda_k(\mathcal{L}_h) + bc\right\}  >  b\mathrm{Lip}(R). \label{equiv to unique existence root condition}
  \end{equation}
  Recall the decomposition of $f(u)=cu+(f(u)-cu)=cu+R(u)$. By \eqref{condition Jf=cI c>=0 }, $c\geq 0$. We can then set $K=c, \phi=R$ in Theorem \ref{thm: unique existence root-finding}. Furthermore, \eqref{condition: G,L>=0, symm, commut } implies $\lambda_k(Q_1^\top \mathcal{L}_hQ_1) = \lambda_k(\mathcal L_h)$. As a result, \eqref{equiv to unique existence root condition} is equivalent to \eqref{condition on existence and uniqueness} in Theorem \ref{thm: unique existence root-finding}, which leads to the unique existence of the root-finding problem $\widehat{F}(U)=0.$
\end{proof}

\subsection{Proofs of section \ref{sec: lyapunov time discrt case}} \label{append: lyapunov analysis of PDHG alg }
Before we prove Theorem \ref{thm: time discrete PDHG converge with Lip }, we need Lemma \ref{lemm: solution set for ineq is nonempty }, \ref{vector function mean value thm } and \ref{Lemm: exponential decay of sequence with 2nd order recurrence ineq }.

\begin{lemma}\label{lemm: solution set for ineq is nonempty }
  Suppose $\theta\in[0, \sqrt{2}-1)$, there exist $u, k>0$, s.t.
  \[ k u \Psi(\theta) - \frac14 \Omega(u, k, \theta)^2 > 0,\]
  where $\Psi(\theta) = 1-2\theta-\theta^2, ~ \Omega(u, k, \theta) =|1-u-k| + \theta(|1-u|+k).$ 
\end{lemma}
\begin{proof}[Proof of Lemma \ref{lemm: solution set for ineq is nonempty }]
We note that $\Omega(u, k, \theta)^2\leq ((1+\theta)(|1-u|+k))^2\leq 2(1+\theta)^2((1-u)^2 + k^2 )$. Then for any $u, k >0$, we have
\begin{align*}
  ku\Psi(\theta) - \frac14 \Omega(u, k, \theta)^2 & \geq ku\Psi(\theta) - \frac12(1+\theta)^2((1-u)^2 + k^2)) \\
  & = ku(1+\theta)^2\left(\frac{\Psi(\theta)}{(1+\theta)^2} - \frac{((1-u)^2 + k^2)}{2ku}\right) \\
  & \geq ku(1+\theta)^2\left(\frac{\Psi(\theta)}{(1+\theta)^2} - \frac{\sqrt{k^2+1} - 1}{k}\right).
\end{align*}
Denote $c=\frac{\Psi(\theta)}{(1+\theta)^2}$. For any $\theta\in[0, \sqrt{2}-1)$, $c\in(0, 1]$. As shown in Figure \ref{fig: function}, it is not hard to verify that $\frac{\sqrt{k^2+1}-1}{k}$ increases monotonically from $0$ to $1$ on $\mathbb{R}_+$. Thus, $\frac{\Psi(\theta)}{(1+\theta)^2} - \frac{\sqrt{k^2+1} - 1}{k}>0$ is guaranteed to have a positive solution $k>0$. This proves the lemma.
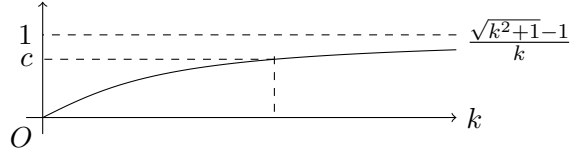
\begin{figure}[htb!]
\centering
\begin{tikzpicture}[scale=1.1]
\centering
  \draw (0, 0) node[below left] {$O$};
  \draw[->] (-0.2, 0) -- (5.0, 0) node[right] {$k$};
  \draw[->] (0, -0.2) -- (0, 1.4) ;
  \draw (5.0, 0.95) node[right] {$\frac{\sqrt{k^2+1}-1}{k}$};
  \draw[scale=1, domain=0.001:4.999, smooth, variable=\x, black] plot ({\x}, {(sqrt(\x*\x+1) - 1)/\x});
  \draw [dashed] (2.8, 0.7047) -- (0, 0.7047) node [left] {$c$};
  \draw [dashed] (2.8, 0.7407) -- (2.8, 0) ;
  \draw [dashed] (5, 1.0) -- (0.0, 1.0)  node [left] {$ 1 $}; 
\end{tikzpicture}
\caption{ Graph of $\frac{\sqrt{k^2+1}-1}{k}$.}\label{fig: function}
\end{figure}
\end{proof}

\begin{lemma}\label{vector function mean value thm }
  Suppose $F:\mathbb{R}^d\rightarrow \mathbb{R}^d$ is differentiable on $\mathbb{R}^d$. Let $\textbf{v} \in \mathbb{R}^d$. Then, for any $x,y\in\mathbb{R}^d$, there exists $t_{\textbf{v}}\in (0, 1)$ such that 
  \begin{equation*} 
    \textbf{v}^\top (F(y) - F(x)) = \textbf{v}^\top DF(x+t_{\textbf{v} } (y-x) ) (y - x).
  \end{equation*}
\end{lemma}
\begin{proof}[Proof of Lemma \ref{vector function mean value thm }]
 Define $h(t) = \boldsymbol{v}^\top (F(x+t(y-x)) - F(x))$. Since $h(\cdot)$ is differentiable on $(0, 1)$, by mean value theorem, there exists $t_{\boldsymbol{v}}\in(0, 1)$ such that $h(1) - h(0) = h'(t_{\boldsymbol{v}})$.
\end{proof}

\begin{lemma}\label{Lemm: exponential decay of sequence with 2nd order recurrence ineq }
  Suppose a positive sequence $\{a_k\}_{k\geq 0}$ satisfies the following recurrence inequality
  \begin{equation}
     a_{k+2} - a_k \leq - \Phi \; a_{k+1}, \quad k \geq 0  \label{recurrence ineq }
  \end{equation}
  with $\Phi > 0$. Then 
  \begin{equation*}
    a_k \leq \left(\frac{2}{\Phi + \sqrt{\Phi^2 + 4}}\right)^{k+1} \left(a_1 + \frac{\Phi + \sqrt{\Phi^2 + 4}}{2}a_0\right) \quad\textrm{for } k\geq 1.
  \end{equation*}
\end{lemma}

\begin{proof}[Proof of Lemma \ref{Lemm: exponential decay of sequence with 2nd order recurrence ineq }]
We consider the characteristic polynomial $r^2 + \Phi r - 1 = 0$. It has two  roots $r_+ = \frac{-\Phi + \sqrt{\Phi^2 + 4}}{2}> 0 $ and $r_- = \frac{-\Phi - \sqrt{\Phi^2 + 4}}{2}<0$. Then $\Phi = \frac{1-r_+^2}{r_+} = \frac{1}{r_+}-r_+$. Plugging this back to \eqref{recurrence ineq } yields
\begin{equation*}
  a_{k+2} + \left(\frac{1}{r_+} - r_+ \right) a_{k+1} - a_{k} \leq 0,   \quad k \geq 0,
\end{equation*}
which further leads to
\begin{equation*}
  a_{k+2} + \frac{1}{r_+} a_{k+1} \leq  r_+\left( a_{k+1} + \frac{1}{r_+} a_{k}\right) \quad k\geq 0.
\end{equation*}
Thus, we obtain
\begin{equation}
  a_{k+1} + \frac{  1  }{r_+} a_k \leq r_+^k \left( a_1 + \frac{1}{r_+} a_0 \right), \quad \textrm{for any } k \geq 0.  \label{estimation on Jk+1 + r Jk }
\end{equation}
Taking the index in \eqref{estimation on Jk+1 + r Jk } as $k-1$ and $k$, one obtains
\begin{align}
   & r_+^{k-1} \left( a_1 + \frac{1}{r_+} a_0 \right)  \geq a_k + \frac{1}{r_+} a_{k-1} > a_k; \nonumber\\
   & r_+^k \left( a_1 + \frac{1}{r_+} a_0 \right)  \geq a_{k+1} + \frac{1}{r_+}a_k > \frac{1}{r_+} a_k. \nonumber
\end{align}
This yields
\begin{equation}
  a_k \leq r_+^{k-1}\left( a_1 + \frac{1}{r_+} a_0  \right),  \quad \textrm{and} \quad a_k \leq r_+^{k+1} \left(a_1 + \frac{1}{r_+} a_0 \right), \quad k\geq 1. \nonumber 
\end{equation}
Since $r_+ < 1$, 
we finally obtain
\begin{equation}
  a_k \leq \left(\frac{2}{\Phi + \sqrt{\Phi^2 + 4}}\right)^{k+1} \left(a_1 + \frac{\Phi + \sqrt{\Phi^2 + 4}}{2} a_0\right), \quad k \geq 1 . \nonumber
\end{equation}
\end{proof}

We now prove Theorem \ref{thm: time discrete PDHG converge with Lip }.
\begin{proof}[Proof of Theorem \ref{thm: time discrete PDHG converge with Lip }]
According to Lemma \ref{lemm: sufficient cond unique existence of root }, under conditions \eqref{condition: a b >= 0 }, \eqref{condition: f Lip }, \eqref{condition: G,L>=0, symm, commut }, \eqref{condition Jf=cI c>=0 }, and 
\[bT\zeta_{a,b,c}(h_t)\mathrm{Lip}(R)<\sqrt{2}-1 < 1,\]
it is straightforward to check the unique existence of the root-finding problem $\widehat{F}(U)=0$. 

Now we suppose $\{U_k, Q_k\}$ solves \eqref{Preconded PDHG}. We write $\mathcal J_k = \mathcal J(U_k, Q_k)$ for convenience. Then we want to bound $\mathcal J_{k+1} - \mathcal J_k$ from above. We calculate
\begin{align*}
  \mathcal J_{k+1} - \mathcal J_k = & (U_{k+1} - U_k)\cdot\left(\frac{1}{2}(U_{k+1}+U_k) - U_*\right) + (Q_{k+1}-Q_k)\cdot\left(\frac{Q_{k+1}+Q_k}{2}\right) \\
  \leq & (U_{k+1} - U_k)\cdot\left(\frac{1}{2}(U_{k+1}+U_k) - U_*\right) + (Q_{k+1}-Q_k)\cdot Q_{k+1}\\
  = & (U_{k+1}-U_k)\cdot (U_k - U_*) + \frac{1}{2}\|U_{k+1} - U_k\|^2 + (Q_{k+1}-Q_k)\cdot Q_{k+1}.
\end{align*}
The inequality is due to the convexity of the quadratic function $\|Q\|^2$. From \eqref{Preconded PDHG}, we know
\begin{align*}
  U_{k+1} - U_k & = -\tau_U  D\widehat F(U_k)^\top (Q_{k+1} + \omega\tau_P(\widehat F(U_k) - \epsilon Q_{k+1} ));\\
  & = -\tau_U  D\widehat F(U_k)^\top ((1-\widetilde{\gamma}\epsilon)Q_{k+1} + \widetilde{\gamma}\widehat F(U_k)); \\
  Q_{k+1} - Q_k & = \tau_P(\widehat F(U_k) - \epsilon Q_{k+1}).
\end{align*}
Let us define $\widetilde \gamma = \omega\tau_P$ and $\varrho = \frac{\tau_P}{\tau_U}$. Using $F(U_*) = 0$, we obtain
\begin{align*}
  & ~~~~~ \mathcal J_{k+1} - \mathcal J_k \\
  & = -\tau_U (U_k-U_*)^\top D\widehat F(U_k)^\top ((1-\widetilde{\gamma}\epsilon)Q_{k+1} + \widetilde \gamma \widehat F(U_k)) + \tau_P  Q_{k+1}^\top (\widehat F(U_k) - \epsilon Q_{k+1}) + \frac12 \|U_{k+1} - U_k\|^2   \\
  & = -\tau_U \Big( \widetilde{\gamma} (U_k-U_*)^\top D\widehat{F}(U_k)^\top \widehat F(U_k) + (1-\widetilde{\gamma}\epsilon)(U_k-U_*)^\top D\widehat{F}(U_k)^\top Q_{k+1} \\
  &  \quad  \quad  \quad  \quad  - \varrho \widehat{F}(U_k)^\top Q_{k+1} + \varrho \epsilon \|Q_{k+1}\|^2 \Big) + \frac{\tau^2_U}{2} \|D\widehat F(U_k)^\top ((1-\widetilde{\gamma}\epsilon)Q_{k+1} + \widetilde{\gamma}\widehat F(U_k))\|^2\\
  & = -\tau_U \Big( \underbrace{\widetilde{\gamma} (U_k-U_*)^\top D\widehat{F}(U_k)^\top(\widehat F(U_k) - \widehat F(U_*))}_{(A)} + \underbrace{(1-\widetilde{\gamma}\epsilon)(U_k-U_*)^\top D\widehat{F}(U_k)^\top Q_{k+1}}_{(B)}\\
  & \quad \underbrace{-\varrho(\widehat{F}(U_k) - \widehat F(U_*))^\top Q_{k+1}}_{(C)} + \underbrace{\varrho\epsilon \|Q_{k+1}\|^2}_{(D)}  \Big)  + \frac{\tau^2_U}{2} \underbrace{\|D\widehat F(U_k)^\top ((1-\widetilde{\gamma}\epsilon)Q_{k+1} + \widetilde{\gamma}\widehat F(U_k))\|^2}_{(E) }.
\end{align*}
By Lemma \ref{vector function mean value thm }, term (A) and term (C) are given by 
\begin{align*}
    (A) &= \widetilde{\gamma} (U_k-U_*)^\top D\widehat{F}(U_k)^\top D\widehat F(U_{k, \nu_1})(U_k - U_* )\,,\\
    (C) &=  -(U_k - U_*)^\top D\widehat{F}(U_{k, \nu_2})^\top Q_{k+1}\,,
\end{align*}
where $U_{k,\nu_j} = U_* + \nu_j(U_k - U_*)$ with $\nu_1, \nu_2\in(0, 1)$, $j=1,2$. 

Recall that $D\widehat F(U) = I + D\eta(U)$. To simplify the notation, we write
\[ \overline{\sigma}_\eta = \sup_{U\in\mathbb{R}^n} \left\{\|D\eta(U)\|\right\}. \]
By Lemma \ref{lemm: more sophisticate est on singular values of DF }, we have $\overline{\sigma}_\eta \leq bT\zeta_{a,b,c}(h_t)\mathrm{Lip}(R).$ We now estimate term (A) as
\begin{align*}
  (A) = & \widetilde{\gamma} (U_k - U_*)D\widehat F(U_k)^\top D\widehat F(U_{k, \nu_1}) (U_k - U_*) \\
  = & (U_k - U_*)^\top (I + D\eta(U_k)^\top)(I + D\eta(U_{k, \nu_1}))(U_k - U_*)\\
  = & \|U_k - U_*\|^2 + (U_k - U_*)^\top D\eta(U_k)^\top (U_k - U_*) + (U_k - U_*)^\top D\eta(U_{k, \nu_1} )(U_k - U_*)\\
  & + (U_k - U_*)^\top D\eta(U_k)^\top D\eta(U_{k, \nu_1})(U_k - U_*) \\
  \geq & (1-2\overline{\sigma}_\eta-\overline{\sigma}_\eta^2)\|U_k - U_*\|^2.
\end{align*}
We can further estimate the terms (B), (C), and (E) as
\begin{align*}
  (B) & =  (1-\widetilde \gamma \epsilon ) (U_k - U_*)^\top (I + D\eta(U_k)) Q_{k+1} ;\\
  (C) &  =  - \varrho (U_k - U_*)^\top D\widehat F(U_{k, \nu_2})^\top  Q_{k+1}= - \varrho (U_k - U_*)^\top (I + D\eta(U_{k, \nu_2}))^\top  Q_{k+1}. 
\end{align*}
Thus
\begin{align*}
  (B)+(C) = & (1-\widetilde \gamma \epsilon ) (U_k - U_*)^\top (I + D\eta(U_k)) Q_{k+1} -  \varrho (U_k - U_*)^\top (I + D\eta(U_{k, \nu_2}))^\top  Q_{k+1} \\
   = & (1-\widetilde \gamma \epsilon-\varrho)(U_k - U_*)^\top Q_{k+1} + (U_k - U_*)^\top ((1-\widetilde\gamma\epsilon)D\eta(U_k) - \varrho D\eta(U_{k,\nu_2}) )^\top Q_{k+1} \\
   \geq & - |1-\widetilde\gamma\epsilon-\varrho|\|U_k - U_*\| \|Q_{k+1}\| - (|1-\widetilde\gamma \epsilon| + \varrho)\bar{\sigma}_\eta \|U_k - U_*\|\|Q_{k+1}\|\\
   = & - (|1-\widetilde\gamma\epsilon-\varrho| + (|1-\widetilde\gamma \epsilon| + \varrho)\bar{\sigma}_\eta) \|U_k - U_*\| \|Q_{k+1}\|\,. 
\end{align*}
And 
\begin{align*}
  (E) & \leq \overline{\sigma}^2(|1-\widetilde{\gamma}\epsilon| \cdot \|Q_{k+1}\| + \widetilde{\gamma} \| \widehat F(U_k) \| )^2 \\
  & \leq \overline{\sigma}^2(|1-\widetilde{\gamma}\epsilon| \cdot \|Q_{k+1}\| + \widetilde{\gamma}\overline{\sigma}\|U_k - U_*\| )^2 \\ 
  & \leq 2\overline{\sigma}^2 ((1-\widetilde{\gamma}\epsilon)^2\|Q_{k+1}\|^2 + \widetilde{\gamma}^2\overline{\sigma}^2\|U_k - U_*\|^2). 
\end{align*}
The second inequality on (E) is due to
\begin{align*}
  \|\widehat F(U_k)\|  =  &  \|\widehat F(U_k) - \widehat F(U_*)\| = \Big\|\int_0^1 \left( \frac{d}{ds} \widehat{F}(U_*+s(U_k - U_*)) \right) ~ ds \Big\| \\
   = & \Big\| \int_0^1 D \widehat F(U_*+s(U_k-U_*))(U_k - U_*)~ds \Big\| \\
   \leq & \int_0^1 \overline\sigma \|U_k - U_*\|~ds = \overline{\sigma } \|U_k - U_*\|.
\end{align*}
Combining the estimations on term (A)-(E), we obtain
\begin{align}
  &  ~~ ~  \mathcal J_{k+1} - \mathcal J_k  \nonumber  \\
  & = -\tau_U \Big(    \widetilde{\gamma} (1-2\overline{\sigma}_\eta-\overline{\sigma}_\eta^2)\|U_k - U_*\|^2 - (|1-\widetilde\gamma\epsilon-\varrho| + (|1-\widetilde\gamma \epsilon| + \varrho)\bar{\sigma}_\eta)\|U_k - U_* \| \|Q_{k+1}\|  \nonumber \\
     & \quad \quad  \quad   \quad  + \varrho\epsilon \|Q_{k+1}\|^2 - \tau_U( \overline{\sigma}^2 (1-\widetilde{\gamma}\epsilon)^2\|Q_{k+1}\|^2  + \widetilde{\gamma}^2\overline{\sigma}^2\|U_k - U_*\|^2 )  \Big)  \nonumber\\
   & = 
   -\tau_U [\|U_k - U_*\|, \|Q_{k+1}\|] \left({\Gamma} - \tau_U \Theta \right) \left[\begin{array}{c}
        \|U_k - U_*\| \\
        \|Q_{k+1}\|
   \end{array}\right] .\label{Jk+1-Jk}    
\end{align}
Here
\begin{gather}
  \Gamma = \left[\begin{array}{cc}
     \widetilde \gamma (1-2\overline{\sigma}_\eta -\overline{\sigma}_{\eta}^2)  &  -\frac12(|1-\widetilde\gamma\epsilon-\varrho| + (|1-\widetilde\gamma \epsilon| + \varrho)\bar{\sigma}_\eta)  \nonumber  \\
       -\frac12(|1-\widetilde\gamma\epsilon-\varrho| + (|1-\widetilde\gamma \epsilon| + \varrho)\bar{\sigma}_\eta) & \varrho \epsilon 
   \end{array}\right],  \\
   \Theta = \left[\begin{array}{cc}
          \widetilde \gamma^2 \overline \sigma^4  &  \\
        &   \overline{ \sigma}^2 (1-\widetilde \gamma\epsilon)^2
   \end{array}\right].  \nonumber 
\end{gather}
Recall that we assume $bT\zeta_{a,b,c}(h_t)\mathrm{Lip}(R) \leq \theta$, this leads to $\overline{\sigma}_\eta \leq \theta$. By Lemma \ref{lemm: more sophisticate est on singular values of DF }, we also have $\overline{\sigma} \leq 1+\theta$. Thus, $\widetilde \gamma(1-2\overline{\sigma}_\eta-\overline{\sigma}_\eta^2) > \widetilde\gamma(1-2\theta-\theta^2)>0$ as $\theta \in [0,  \sqrt{2} - 1) $. Hence, 
\begin{align} 
       \mathrm{det}(\Gamma) = & \varrho \widetilde \gamma \epsilon  (1-2\overline{\sigma}_\eta - \overline{\sigma}_\eta^2) - \frac{1}{4}(|1-\widetilde\gamma\epsilon-\varrho| + (|1-\widetilde\gamma \epsilon| + \varrho)\bar{\sigma}_\eta)^2  \nonumber   \\
       \geq & \varrho \widetilde \gamma \epsilon (1-2\theta-\theta^2) - \frac14 (|1-\widetilde\gamma\epsilon-\varrho| + (|1-\widetilde\gamma \epsilon| + \varrho)\theta)^2. \nonumber
\end{align}

We denote $\Psi(\theta) = 1 - 2\theta - \theta^2 $ and $\Omega(u, \varrho, \theta) = |1-u-\varrho|+(|1-u|+\varrho)\theta$. Lemma \ref{lemm: solution set for ineq is nonempty } guarantees that there exist $\widetilde\gamma, \omega, \epsilon$, such that \eqref{condition on gamma, ktau, epsilon } holds. The condition \eqref{condition on gamma, ktau, epsilon } leads to $\textrm{det}(\Gamma)>0$, which guarantees the positive definiteness of $\Gamma$.

Furthermore, we have $\Gamma \succeq \lambda_2(\Gamma) I$, where $\lambda_2(\Gamma)$ represents the smallest eigenvalue of $\Gamma$ and $I$ is an identity matrix. One can bound $\lambda_2(\Gamma)$ from below as
\begin{align*}
  \lambda_2(\Gamma) & = \frac{\widetilde\gamma(1-2\overline{\sigma}_\eta - \overline{\sigma}_\eta^2) + \varrho\epsilon - \sqrt{ ( \widetilde\gamma(1-2\overline{\sigma}_\eta - \overline{\sigma}_\eta^2) + \varrho\epsilon )^2 - 4 \mathrm{det}(\Gamma) }}{2} \\ 
  & \geq \frac{4 (\varrho \widetilde \gamma \epsilon (1-2\theta-\theta^2) - \frac14 (|1-\widetilde\gamma\epsilon-\varrho| + (|1-\widetilde\gamma \epsilon| + \varrho)\theta)^2) }{2(\widetilde\gamma(1-2\overline{\sigma}_\eta - \overline{\sigma}_\eta^2) + \varrho\epsilon + \sqrt{  ( \widetilde\gamma(1-2\overline{\sigma}_\eta - \overline{\sigma}_\eta^2) + \varrho\epsilon )^2 - 4 \mathrm{det}(\Gamma) }) } \\
  &  \geq  \frac{ \varrho \widetilde \gamma \epsilon (1-2\theta-\theta^2) - \frac14 (|1-\widetilde\gamma\epsilon-\varrho| + (|1-\widetilde\gamma \epsilon| + \varrho)\theta)^2 }{\widetilde{\gamma}(1-2\overline{\sigma}_\eta-\overline{\sigma}_\eta^2)  +  \varrho  \epsilon  }\\
  & \geq \frac{\varrho\widetilde\gamma\epsilon \Psi(\theta)  - \frac14 \Omega(\widetilde\gamma\epsilon, \varrho, \theta)^2  }{\widetilde\gamma + \varrho\epsilon}.
 \end{align*}

On the other hand, we have 
\begin{equation*}
  \Theta \prec \overline{\sigma}^2\max\{\widetilde\gamma^2\overline{\sigma}^2 , |1-\widetilde\gamma\epsilon|^2\} I \prec    (1+\theta)^2\max\{\widetilde{\gamma}^2(1+\theta)^2, |1-\widetilde{\gamma}\epsilon|^2 \} I .
\end{equation*}

Thus we have 
\begin{equation*}
  \Gamma - \tau \Theta \succ  \underbrace{\left( \frac{ \varrho\widetilde\gamma\epsilon \Psi(\theta)  - \frac14 \Omega(\widetilde\gamma\epsilon, \varrho, \theta)^2   }{\widetilde{\gamma} + \varrho \epsilon  }  - \tau (1+\theta)^2\max\{ \widetilde{\gamma}^2(1+\theta)^2, |1-\widetilde{\gamma}\epsilon|^2 \} \right)}_{\textrm{denote as } C(\theta, \widetilde{\gamma}, \epsilon, \varrho, \tau)} I . 
\end{equation*}
Plug this estimation to \eqref{Jk+1-Jk}, we obtain
\begin{equation*}
  \mathcal J_{k+1} - \mathcal J_k \leq - \tau C(\theta, \widetilde{\gamma}, \epsilon, \varrho, \tau) (\|U_{k}-U_*\|^2 + \|Q_{k+1}\|^2).
\end{equation*}
Since we set the PDHG step size as
\begin{equation*}
  0 < \tau < \bar{\tau}(\theta, \widetilde{\gamma}, \epsilon, \varrho, \tau) \triangleq \frac{\varrho\widetilde\gamma\epsilon \Psi(\theta)  - \frac14 \Omega(\widetilde\gamma\epsilon, \varrho, \theta)^2}{2(\widetilde\gamma + \varrho\epsilon)(1+\theta)^2\max\{ \widetilde\gamma^2 (1+\theta)^2 ,  (1 -\widetilde\gamma\epsilon)^2 \}},
\end{equation*}
this guarantees  $C(\theta, \widetilde{\gamma}, \epsilon, \varrho, \tau) > 0$.

Furthermore, as a function of $\tau$, $\tau C(\theta, \widetilde{\gamma}, \epsilon, \varrho, \tau)$ reaches its maximum value at $\tau = \frac12 \bar{\tau}(\theta, \widetilde{\gamma}, \epsilon, \varrho, \tau )$. We then set (here $\bar{\tau}$ denotes $\bar\tau(\theta, \widetilde\gamma, \epsilon, \varrho)$)
\[ \Phi = \frac{1}{2} \bar{\tau} C(\theta, \widetilde{\gamma}, \epsilon, \varrho, \frac{1}{2}\bar{\tau} ) = \frac{(\varrho\widetilde\gamma\epsilon \Psi(\theta)  - \frac14 \Omega(\widetilde\gamma\epsilon, \varrho, \theta)^2 )^2}{2(1+\theta)^2\max\{\widetilde\gamma^2(1+\theta)^2, (1-\widetilde\gamma \epsilon)^2\}(\widetilde{\gamma} + \varrho \epsilon)^2 }. \]
Thus we have
\begin{align}
  \mathcal J_{k+1} - \mathcal J_k \leq -\Phi \cdot\frac12 (\|U_k - U_*\|^2 + \|Q_{k+1}\|^2).  \nonumber    
\end{align}

Now we prove the exponential decay of $\mathcal{J}_k$. To do so, we sum up the above inequality at time index $k$ and $k+1$ to obtain,
\begin{equation}
  \mathcal{J}_{k+2} - \mathcal{J}_{k} \leq -\Phi \cdot \frac{1}{2}(\|U_{k+1} - U_*\|^2 + \|Q_{k+2}\|^2 + \|U_{k} - U_*\|^2 + \|Q_{k+1}\|^2), \quad k\geq 0. \nonumber
\end{equation}
It is not hard to see that the right-hand side of the above inequality is no larger than $-\Phi \mathcal J_{k+1}$. Hence,
\begin{equation}
  \mathcal J_{k+2} - \mathcal J_{k} \leq - \Phi  \mathcal J_{k+1}.  \label{time disct second order ineq }
\end{equation}

Now, by Lemma \ref{Lemm: exponential decay of sequence with 2nd order recurrence ineq }, we obtain 
\begin{equation}
  \mathcal J_k \leq \left(\frac{2}{\Phi + \sqrt{\Phi^2 + 4}}\right)^{k+1} \left(\mathcal J_1 + \frac{\Phi + \sqrt{\Phi^2 + 4}}{2}\mathcal J_0\right), \quad \textrm{for } k\geq 1. \nonumber 
\end{equation}
This concludes our proof.
\end{proof}

\end{document}